\documentclass[11pt]{article}
\usepackage{amsmath, amssymb, amsthm, esint, hyperref, cite}
\usepackage{geometry}

\usepackage{xcolor}

\def\R{{\mathbb R}}

\theoremstyle{definition}

\newtheorem{theorem}{Theorem}[section]
\newtheorem{corollary}{Corollary}[section]
\newtheorem{definition}{Definition}[section]
\newtheorem{proposition}{Proposition}[section]
\newtheorem{remark}{Remark}[section]
\newtheorem{lemma}{Lemma}[section]

\begin{document}

\title{\sc{Deterministic ill-posedness and probabilistic well-posedness of the viscous nonlinear wave equation
describing fluid-structure interaction}}
\author{Jeffrey Kuan and Sun\v{c}ica \v{C}ani\'{c}\\
Department of Mathematics\\
University of California Berkeley}
%\date{December 12, 2019}
\maketitle
\begin{abstract}
We study low regularity behavior of the nonlinear wave equation in $\R^2$ augmented by the viscous dissipative effects described by the 
Dirichlet-Neumann operator. Problems of this type arise in fluid-structure interaction where the Dirichlet-Neumann operator
models the coupling between a viscous, incompressible fluid and an elastic structure. 
We show that despite the viscous regularization,
the Cauchy problem with 
initial data $(u,u_t)$ in $H^s(\R^2)\times H^{s-1}(\R^2)$, 
is ill-posed whenever ${{0 < s < s_{cr}}}$,
where the critical exponent $s_{cr}$ depends on the degree of nonlinearity. 
In particular, for the quintic nonlinearity $u^5$, the critical exponent in $\R^2$ is $s_{cr} = 1/2$,
which is the same as the critical exponent for the associated nonlinear wave equation without the viscous term. 
We then show that if the initial data is perturbed using a Wiener randomization, which perturbs initial data
in the frequency space, then the Cauchy problem for the quintic nonlinear viscous wave equation is well-posed
almost surely for the supercritical exponents $s$ such that $-1/6 < s \le s_{cr} = 1/2$. 
To the best of our knowledge, this is the first result showing ill-posedness and probabilistic well-posedness 
for the nonlinear viscous wave equation arising in fluid-structure interaction. 
\end{abstract}

\section{Introduction}\label{introduction}

We study low regularity behavior of the nonlinear wave equation augmented by the viscous effects 
described by the Dirichlet-Neumann operator
typically arising in fluid-structure interaction problems:

{{}{\begin{equation}\label{model_equation}
\partial_{tt} u - \triangle u + u^p + 2\mu \sqrt{-\triangle} \partial_t u = 0, \quad (x,y) \in \R^2, t \in \R,
\end{equation}
where $p>0$ is an odd integer, and $\mu > 0$.}}
The model above can be thought of as a mathematical prototype for the interaction between 
a prestressed, stretched membrane and a viscous, incompressible fluid. 
The membrane (an infinitely large drum surface) is
modeled by the linear wave equation:
$$
\partial_{tt} u - \triangle u = f, \quad (x,y) \in \R^2, t \in \R,
$$
where $u=u(x,y)$ is a scalar function describing transverse membrane displacement. {{}{We assume for simplicity that the structure experiences displacement only in the transverse, $z$ direction and hence experiences no tangential displacements from its reference configuration.}}
The incompressible,
viscous Newtonian fluid is located under the membrane in the half space $z < 0$,
modeled by the  Stokes equations:
\begin{equation}\label{Stokes}
\left.
\begin{array}{rcl}
 \nabla \pi &=& \mu \triangle {\boldsymbol v}, \\
\nabla \cdot {\boldsymbol v} &=& 0,
\end{array}
\right\} \quad {\rm in} \ \Omega = \{ (x,y,z) \in \R^3 : z<0\},
\end{equation}
where $\pi$ is the fluid pressure, ${\boldsymbol v}$ is the fluid velocity, and $\mu$ is the kinematic viscosity coefficient.
The first equation in \eqref{Stokes} describing the second Newton's law of motion (balance of forces), can be written as
$$
 \nabla \cdot \boldsymbol\sigma(\pi,{\boldsymbol v})=0,
$$
where the Cauchy stress tensor $\boldsymbol\sigma$ is given by
{{}{\begin{equation}\label{sigma}
\boldsymbol\sigma = -\pi \boldsymbol{I} + 2\mu {\boldsymbol{D(v)}}
\end{equation}
for Newtonian fluids, $\boldsymbol{I}$ is the identity matrix, and $\boldsymbol{D(v)}$ denotes the symmetrized gradient of velocity.}}

The fluid and structure are coupled through two coupling conditions: the kinematic and dynamic coupling conditions.
For simplicity, we will be assuming that the coupling is at the fixed fluid-structure interface, 
which we denote by {{}{$\Gamma = \{ (x,y,z) \in \R^{3} : z = 0\}$}}. 
The fixed interface corresponds to the reference configuration of the stretched (prestressed) membrane. The coupling conditions read:
\begin{itemize}
\item The kinematic coupling condition (describing continuity of velocities):
\begin{equation}\label{kinematic}
{\boldsymbol v}|_{\Gamma} = u_t {\boldsymbol e}_z, \quad (x,y,z) \in \Gamma,
\end{equation}
where ${\boldsymbol e}_z$ is the unit vector pointing in the positive $z$ direction.
Notice that the kinematic coupling condition \eqref{kinematic} states that on the boundary $\Gamma$, the tangential 
components of the trace of the fluid velocity are equal to zero. Only the normal component of the fluid velocity is assumed
to be nonzero on $\Gamma$.
\item The dynamic coupling condition (describing balance of forces, i.e., the second Newton's law of motion):
\begin{equation}\label{dynamic}
\partial_{tt} u - \triangle u = - \boldsymbol\sigma \boldsymbol{e}_z \cdot \boldsymbol{e}_z +F_{\rm ext}(u).
\end{equation}
The right hand-side of \eqref{dynamic} describes the jump in the normal stress (traction) across the fluid-structure 
interface, where $F_{\rm ext}(u)$ denotes external force, which in general may depend on $u$, acting on the membrane in the 
normal direction $-\boldsymbol{e}_z$. 
\end{itemize}

In our model we will be assuming that the external force is a nonlinear function of $u$, and that it is given by
$$
F_{\rm ext}(u) = - u^p,
$$
where 
the nonlinear term $u^p$ models, for example,  the nonlinear spring-type effects,
distributed across the membrane surface by, say, a surrounding medium (e.g., structure) sitting on top of the membrane. 
These types of external source terms have been used in modeling blood flow in compliant arteries, see \cite{Astorino},
to describe the tissue surrounding pulsating arteries.
In cylindrical geometry, the nonlinear term in the cylindrical nonlinear membrane/shell 
equations appears due to the contribution of the circumferential stress/strain, as was done in \cite{CanicCMAME}. 

{{}{Next, we compute the contribution of the term arising from the Cauchy stress tensor, $-\boldsymbol{\sigma} \boldsymbol{e_{z}} \cdot \boldsymbol{e_{z}}$, on the right hand side of \eqref{dynamic}, which is evaluated on the fixed fluid-structure interface $\Gamma$. Using \eqref{sigma},
\begin{equation*}
-\boldsymbol{\sigma} \boldsymbol{e_{z}} \cdot \boldsymbol{e_{z}} = \pi - 2\mu \frac{\partial v_{z}}{\partial z}.
\end{equation*}
Because we are evaluating this quantity on $\Gamma$, we note that 
\begin{equation}\label{normal}
\frac{\partial v_{z}}{\partial z} = 0 \qquad \text{ on } \Gamma,
\end{equation}
by the divergence-free condition $\nabla \cdot \boldsymbol{v} = 0$ and the fact that $v_{x} = v_{y} = 0$ on $\Gamma$, since we are assuming that the structure experiences displacement only in the transverse $z$ direction. Therefore,
\begin{equation}\label{pressure}
-\boldsymbol{\sigma} \boldsymbol{e_{z}} \cdot \boldsymbol{e_{z}} = \pi,
\end{equation}
where $\pi$ is the fluid pressure given as a solution to the Stokes equations \eqref{Stokes}.

So it remains to find an appropriate expression for $\pi$ in terms of the structure displacement $u$ as follows. We will derive the formula
\begin{equation}\label{pressureDN}
\pi = -2\mu \sqrt{-\Delta} u_{t} \quad {{} \text{on} \ \Gamma},
\end{equation}
under the assumption that $u$ and $u_{t}$ are smooth functions, along with their spatial derivatives, that are rapidly decreasing at infinity. We will also impose the boundary conditions on \eqref{Stokes}, {{}{stating}} that the fluid velocity is bounded on the lower half {{}{space}}, and the pressure $\pi$ has a limit equal to zero as $|x| \to \infty$ in the lower half {{}{space}}.

To derive the formula \eqref{pressureDN}, we note that by taking the inner product of the first equation in \eqref{Stokes} with $\boldsymbol{e_{z}}$, we obtain
\begin{equation}\label{Neumannpressure}
\frac{\partial \pi}{\partial z} = \mu \Delta_{x, y} v_{z} + \mu \frac{\partial^{2} v_{z}}{\partial z^{{}{2}}},
\end{equation}
where $\Delta_{x, y} := \frac{\partial^{2}}{\partial x^{2}} + \frac{\partial^{2}}{\partial y^{2}}$. Furthermore, by taking the divergence of the first equation in \eqref{Stokes}, and by using the divergence-free condition, we get that the pressure $\pi$ is harmonic. Thus, if we can compute the right hand side of \eqref{Neumannpressure} on $\Gamma$, we can recover $\pi$ as the solution to a Neumann boundary {{}{value}}  problem for Laplace's equation in the lower half {{}{space}}, with the boundary condition 
{{}{requiring}} that $\pi$ goes to zero at infinity. 

To compute the right hand side of \eqref{Neumannpressure}, 
we need to compute $v_{z}$.
We use a Fourier analysis argument. {{}{From \eqref{Neumannpressure} we see that $v_{z}$ satisfies}}
\begin{equation*}
\Delta v_{z} = \frac{\partial \pi}{\partial z}.
\end{equation*}
Taking the Laplacian {{}{on both sides}}, and using the fact that $\pi$ is harmonic, we obtain
\begin{equation}\label{biharmonic}
\Delta^{2} v_{z} = 0, \qquad \text{ on } \Omega = \{(x, y, z) \in \R^{3} : z < 0\}.
\end{equation}
{{}{Thus, $v_z$ satisfies the}} biharmonic equation with {{}{the following}} two boundary conditions:
 {{}{from}} the kinematic coupling condition we get 
\begin{equation}\label{bc1}
v_{z}(x, y, 0) = u_{t}(x, y, 0), \qquad \text{ on } \Gamma = \{(x, y, z) \in \R^{3}: z = 0\},
\end{equation}
and by \eqref{normal}, {{}{we have}}
\begin{equation}\label{bc2}
\frac{\partial v_{z}}{\partial z}(x, y, 0) = 0, \qquad \text{ on } \Gamma = \{(x, y, z) \in \R^{3}: z = 0\}.
\end{equation}
{{}{We remark that the biharmonic equation to analyze the Stokes problem has been previously used
in many works,}} see for example \cite{Sim}. 

We solve \eqref{biharmonic} with boundary conditions \eqref{bc1} and \eqref{bc2} by taking a Fourier transform in the variables $x$ and $y$, but not in $z$. We will denote the Fourier variables associated with $x$ and $y$ by $\xi_{1}$ and $\xi_{2}$, and we will denote $\boldsymbol{\xi} = (\xi_{1}, \xi_{2})$, $|\boldsymbol{\xi}|^{2} = \xi_{1}^{2} + \xi_{2}^{2}$. 
{{}{The Fourier transform equation then reads:}}
 \begin{equation}\label{Fourierstokes}
|\xi|^{4} \widehat{v_{z}}(\boldsymbol{\xi}, z) - 2|\boldsymbol{\xi}|^{2} \frac{\partial^{2}}{\partial z^{2}} \widehat{v_{z}}(\boldsymbol{\xi}, z) + \frac{\partial^{4}}{\partial z^{4}} \widehat{v_{z}}(\boldsymbol{\xi}, z) = 0.
\end{equation}
The general solution to \eqref{Fourierstokes} is
\begin{equation*}
\widehat{v_{z}}(\boldsymbol{\xi}, z) = C_{1}(\boldsymbol{\xi}) e^{|\boldsymbol{\xi}|z} + C_{2}(\boldsymbol{\xi}) ze^{|\boldsymbol{\xi}|z}+ C_{3}(\boldsymbol{\xi}) e^{-|\boldsymbol{\xi}|z} + C_{4}(\boldsymbol{\xi}) z e^{-|\boldsymbol{\xi}|z}.
\end{equation*}
Because $e^{-|\boldsymbol{\xi}|z}$ and $z e^{-|\boldsymbol{\xi}|z}$ are unbounded in the lower half plane where $z < 0$, we exclude these two terms and are left with
\begin{equation}\label{general}
\widehat{v_{z}}(\boldsymbol{\xi}, z) = C_{1}(\boldsymbol{\xi}) e^{|\boldsymbol{\xi}|z} + C_{2}(\boldsymbol{\xi}) ze^{|\boldsymbol{\xi}|z}.
\end{equation}
In Fourier variables, the two boundary conditions \eqref{bc1} and \eqref{bc2} translate to
\begin{equation*}
\widehat{v_{z}}(\boldsymbol{\xi}, 0) = \widehat{u_{t}}(\boldsymbol{\xi}), \qquad \frac{\partial \widehat{v_{z}}}{\partial z}(\boldsymbol{\xi}, 0) = 0,
\end{equation*}
which allow us to solve for the general functions $C_{1}(\boldsymbol{\xi})$ and $C_{2}(\boldsymbol{\xi})$ in \eqref{general}, giving the result
\begin{equation}\label{finalFourier}
\widehat{v_{z}}(\boldsymbol{\xi}, z) = \widehat{u_{t}}(\boldsymbol{\xi}) e^{|\boldsymbol{\xi}| z} - |\xi| \widehat{u_{t}}(\boldsymbol{\xi}) z e^{|\boldsymbol{\xi}| z}.
\end{equation}

We can now compute the right hand side of \eqref{Neumannpressure}. Taking the Fourier transform of \eqref{Neumannpressure} in the $x$ and $y$ variables, and evaluating {{}{the equation}} on $\Gamma$ by using {{}{the kinematic coupling condition}} \eqref{kinematic}, {{}{we get}}
\begin{equation*}
\frac{\partial \widehat{\pi}}{\partial z}(\boldsymbol{\xi}, 0) = -\mu |\xi|^{2} \widehat{u_{t}}(\boldsymbol{\xi}) + \mu \frac{\partial^{2} \widehat{v_{z}}}{\partial z^{{}{2}}}(\boldsymbol{\xi}, 0). 
\end{equation*}
{{}{From}} the explicit formula for $\widehat{v_{z}}(\boldsymbol{\xi}, z)$ in \eqref{finalFourier}, we conclude that
\begin{equation}\label{NeumannFourier}
\frac{\partial \widehat{\pi}}{\partial z}(\boldsymbol{\xi}, 0) = -2\mu |\xi|^{2} \widehat{u_{t}}(\boldsymbol{\xi}).
\end{equation}

{{}{We have now obtained that the pressure $\pi$ is a harmonic function in the lower half space, satisfying
a Neumann boundary condition, posed in Fourier space as \eqref{NeumannFourier}.
To obtain $\pi$ on $\Gamma$ and recover formula
\eqref{pressureDN}, we can now employ the Neumann to Dirichlet operator. }} 

It is well-known that the Dirichlet to Neumann operator for Laplace's equation in the lower half {{}{space}} (with the solution to Laplace's equation having a limit of zero at infinity) is given by $\sqrt{-\Delta}$, see for example \cite{CS}. Therefore, the Neumann to Dirichlet operator for Laplace's equation in the lower half {{}{space}} (with the solution to Laplace's equation having a limit of zero at infinity) is a Fourier multiplier of the form $\frac{1}{|\xi|}$. Since $\pi$ is a harmonic function satisfying the Neumann boundary condition \eqref{NeumannFourier}, by applying the Neumann to Dirichlet operator we get
\begin{equation*}
\widehat{\pi}(\xi) = -2\mu |\xi| \widehat{u_{t}}(\xi) \qquad \text{ on } \Gamma,
\end{equation*}
which establishes the desired formula \eqref{pressureDN}.

\if 1 = 0
To capture the effects of fluid viscosity on the regularization of the 
coupled fluid-structure interaction problem, we ignore the pressure in the fluid 
problem and consider only the Laplace's equation for $\boldsymbol{v}$:
\begin{equation}\label{Laplace}
\triangle \boldsymbol{v} = 0 \quad {\rm in} \quad \Omega.
\end{equation}
In that case, the term $\boldsymbol\sigma \boldsymbol{e}_z \cdot \boldsymbol{e}_z$ is given by:
$$
\boldsymbol\sigma \boldsymbol{e}_z \cdot \boldsymbol{e}_z = \mu \frac{\partial v_z}{\partial {\boldsymbol e}_z }.
$$
%
%The term involving the square root of the negative Laplacian is associated with the coupling between the fluid
%and structure via a Dirichlet-Neumann operator, assuming the non-slip condition at the fluid-structure interface.
Therefore, equation \eqref{dynamic} now becomes:
\begin{equation}\label{equation0}
\partial_{tt} u - \triangle u + u^p  + \mu \frac{\partial v_z}{\partial {\boldsymbol e}_z }= 0.
\end{equation}
The term  involving the normal derivative of the $z$-component of the fluid velocity, $v_z$, can now be expressed in
term of the Dirichlet-to-Neumann operator, since ${\partial v_z}/{\partial {\boldsymbol e}_z }$ appearing in \eqref{equation0}
is Neumann data for the fluid velocity on $\Gamma$, where the fluid velocity  
satisfies the Laplace's equation in $\Omega$, with Dirichlet data \eqref{kinematic} on $\Gamma$.
%
%Since we are interested in the effects of viscous dissipation on ill-posedness of \eqref{model_equation},
%we will be assuming that the variations in the fluid pressure from some constant pressure $p_0$ are negligible, so that we can 
%ignore the term $\nabla p$ in the Stokes equations. 
Therefore, we can express that term as the square root of the negative Laplacian applied to $u_t$
(see, e.g., Caffarelli and Silvestre \cite{CS}).   
\fi

The result in \eqref{pressureDN}, along with \eqref{pressure},  implies the following form of the dynamic coupling condition:
$$
\partial_{tt} u - \triangle u + u^p + 2\mu \sqrt{-\triangle} u_t= 0,
$$
which accounts for the influence of the fluid viscosity within the fluid domain $\Omega$ and its trace on the domain boundary $\Gamma$.
%In the case when the (approximately) constant fluid pressure $p_0 = p_{\rm ext}$, we get:
%$$
%\partial_{tt} u - \triangle u + u^p  + \mu \sqrt{-\triangle} u_t= 0,
%$$
%which is exactly \eqref{model_equation}.
For {{}{simplicity}}, we will set $2\mu = 1$, {{}{and}} study the equation 
$$
\partial_{tt} u - \triangle u + u^p + \sqrt{-\triangle} u_t= 0.
$$
We will refer to equation \eqref{model_equation} as the {\emph{viscous nonlinear wave equation}} (vNLW). }}

We are interested in the Cauchy problem for equation \eqref{model_equation}, where $p > 0 $ is an odd integer, and $\mu > 0$, 
supplemented with 
initial data:
\begin{equation}\label{data}
u(0,\cdot) = f \quad {\rm and} \quad u_t(0,\cdot) = g,
\end{equation}
where 
$(f, g) \in \mathcal{H}^{s}(\mathbb{R}^{2}) = H^{s}(\mathbb{R}^{2}) \times H^{s - 1}(\mathbb{R}^{2})$.
Here $H^s$ denotes the usual (inhomogeneous) Sobolev space.

The analysis of fluid-structure interaction problems involving incompressible, viscous fluids and elastic 
structures started in  the early 2000's with works in which the coupling between the fluid 
and structure was assumed across a fixed fluid-structure interface ({\emph{linear coupling}}) 
as in \cite{Gunzburger,BarGruLasTuff2,BarGruLasTuff,KukavicaTuffahaZiane}, 
and was then extended to problems with {\emph{nonlinear coupling}} in the works 
\cite{BdV1,Lequeurre,FSIforBIO_Lukacova,CDEM,CG,MuhaCanic13,
BorSun3d,LengererRuzicka,CSS1,CSS2,Kuk,ChenShkoller,ChengShkollerCoutand,IgnatovaKukavica,Raymod,ignatova2014well,
BorSunSlip,BorSunNonLinearKoiter,BorSunMultiLayered,Grandmont16}. 
In all these studies, a major underlying  reason for the well-posedness 
is the regularization by the fluid viscosity and the dispersive nature of the wave-like operators in more than one spatial dimension. 
One of the main questions is ``by how much'' does the fluid viscosity regularize the coupled problem? How does the 
viscous regularization ``compete'' with the nonlinearities in the problem?
In the present work, we study the influence of fluid viscosity and nonlinearity on the well-posedness of the Cauchy problem 
for the nonlinear viscous wave equation by studying the following two problems:
\begin{enumerate}
\item[(P1)] For a given exponent $p > 0$, which is an odd integer describing the nonlinearity in the problem, 
is there a critical exponent $s_{cr}$ such that equation \eqref{model_equation}
with initial data $(f, g) \in \mathcal{H}^{s}(\mathbb{R}^{2}) = H^{s}(\mathbb{R}^{2}) \times H^{s - 1}(\mathbb{R}^{2})$,
with {{$0 < s < s_{cr}$}}, is {\sl{ill-posed}} in the sense that the solution mapping
$$
(f,g) \mapsto u,
$$
which takes the initial data $(f, g) \in \mathcal{H}^{s}(\mathbb{R}^{2})$ and maps it to 
a solution $u \in C^0([0,T],H^s(\R^2))\cap C^1([0,T],H^{s-1}(\R^2))$,
fails to be continuous?
\item[(P2)] If well-posedness fails for some critical exponent $s_{cr}$, how ``generic'' is that behavior?
Is there a random perturbation of the initial data that would provide well-posedness, even for 
``supercritical'' initial data, namely for $s < s_{cr}$, and how generic is that random perturbation?
\end{enumerate}

As we shall see, the answer to problem (P1) is yes. Namely, despite the regularization by fluid viscosity, 
there is a critical exponent $s_{cr}$ depending on the nonlinearity $p$, below which
the viscous nonlinear wave equation \eqref{model_equation} is not well-posed for the initial data 
$u \in C^0([0,T],H^s(\R^2))\cap C^1([0,T],H^{s-1}(\R^2))$ such that {{$0 < s < s_{cr}$}}. 
Namely, as $t\to 0$, the energy of the low frequency Fourier modes gets transferred to the high frequencies
so that the $H^s$-norm of the solution becomes arbitrarily large as $t \to 0$, even though the $H^s$-norm
of the initial data is arbitrarily small.

The answer to problem (P2) is that the ill-posedness addressed in problem (P1) is {\sl{not generic}} in a certain sense.
Namely, we show for the quintic nonlinearity $p = 5$, for example,
that if we randomize the initial data using a Wiener randomization, which perturbs initial data in frequency space 
via independent random variables with bounded six moments (associated with $p = 5$), 
then the Cauchy problem for the quintic nonlinear viscous wave equation \eqref{model_equation}
will be well-posed almost surely. 

To obtain answers to problems (P1) and (P2), we start by looking for symmetries of equation \eqref{model_equation}.
Symmetries can provide insight into the questions of if and when well-posedness may be expected.
As in the theory of nonlinear dispersive equations \cite{CCT,CW,LS}, 
we look for scaling symmetries. As we will see later, the scaling symmetry
\begin{equation*}
u(t, x) \to \lambda^{\frac{2}{p - 1}}u(\lambda t, \lambda x)
\end{equation*}
preserves solutions of equation \eqref{model_equation}. Furthermore, it preserves the $\dot{H}^{s}$ norm of the solution when $s = s_{cr}$, where
\begin{equation*}
s_{cr} = \frac{n}{2} - \frac{2}{p - 1} = 1 - \frac{2}{p - 1},
\end{equation*}
and $n$ is the dimension of $\R^n$ ($n = 2$ in the present work). Here, $\dot{H}^{s}(\mathbb{R}^{n})$ denotes the homogeneous Sobolev space, equipped with the norm
\begin{equation*}
||u||_{\dot{H}^{s}(\mathbb{R}^{n})}^{2} = \frac{1}{(2\pi)^{n}}\int_{\mathbb{R}^{n}} |\xi|^{2s} |\widehat{u}(\xi)|^{2}d\xi.
\end{equation*}
Note that this is the same critical exponent as for the defocusing nonlinear wave equation \cite{CCT, LS}:
\begin{equation}\label{NLW}
\partial_{tt}u - \Delta u + u^{p} = 0,
\end{equation}
where the defocusing case corresponds to the choice of the positive sign in front of the nonlinear term $u^p$.
We will see that in terms of energy inequalities and ill-posedness behavior, the viscous nonlinear wave equation \eqref{model_equation}
shares similar properties with
the nonlinear wave equation \eqref{NLW}, but the viscous nonlinear wave equation also exhibits novel behavior 
that arises from dissipation of energy 
due to viscosity.

The presence of a critical exponent is crucial for the analysis of nonlinear partial differential equations, in particular for dispersive equations. Above the critical exponent, dispersive equations usually exhibit {\bf{well-posedness}}, which can be established from fixed point arguments in combination with dispersive estimates for the linear problem \cite{CW,LS}. For example, such a result was established first for the nonlinear Schr\"{o}dinger equation above the critical exponent by Cazenave and Weissler in \cite{CW}, and a well-posedness result for the nonlinear wave equation above the critical exponent was established by Lindblad and Sogge in \cite{LS}. 

In the case of the viscous nonlinear wave equation, we expect to have well-posedness for the
exponents $s$ above the critical exponent $s_{cr}$ using similar approaches as in \cite{CW,LS}, 
since the presence of viscous regularization can only improve solution behavior.
{\emph{Well-posedness}} of the Cauchy problem for equation  \eqref{model_equation} (in the Hadamard sense) is defined by requiring 
that there exists a unique solution $u \in C^0([0,T],H^s(\R^2))\cap C^1([0,T],H^{s-1}(\R^2))$
such that the solution mapping
$$
(f,g) \mapsto u,
$$
which takes the initial data 
$$(f, g) \in \mathcal{H}^{s}(\mathbb{R}^{2})=H^{s} \times H^{s - 1}$$
 and maps it to 
a solution $u \in C^0([0,T],H^s(\R^2))\cap C^1([0,T],H^{s-1}(\R^2))$,
is  continuous.

In this manuscript, however, we are interested in the supercritical ($s < s_{cr}$) behavior of equation \eqref{model_equation}. 
Below the critical exponent which preserves the $\dot{H}^{s_{cr}}$ norm of the scaling symmetry 
$u(t, x) \to \lambda^{\frac{2}{p - 1}}u(\lambda t, \lambda x)$, we would heuristically expect {\bf{ill-posedness}}.

Indeed, such ill-posedness behavior was most famously considered for certain ranges of $s$ of initial data in $\mathcal{H}^{s}$ for the nonlinear wave equation and the nonlinear Schr\"{o}dinger equation by Christ, Colliander, and Tao in \cite{CCT}. 

It is not clear {\sl{a priori}} that the viscous nonlinear wave equation would embody the same property
because of the regularizing effects by the fluid viscosity. 
However, in Sec.~\ref{illposed}, we show that a similar ill-posedness result can be obtained using a similar procedure for the viscous nonlinear wave equation with
initial data in $\mathcal{H}^{s}$ with $0 < s < s_{cr}$. In particular, we will show lack of continuity of the solution map for the viscous nonlinear wave equation when $0 < s < s_{cr}$.

As a remark, we note that for the nonlinear wave equation, because solutions can satisfy finite speed of propagation, it is possible to construct initial data for which there is instantaneous blowup \cite{CCT}. 
%This can be accomplished by using the lack of continuity of the solution map. 
In particular, it can be shown that for a certain choice of initial data, there exist \underline{no} $T > 0$ for which there is a weak solution to the nonlinear wave equation in $C([0, T]; H^{s})$. For more details, see the discussion in \cite{CCT}. 
However, we would not expect instantaneous blowup for the viscous nonlinear wave equation, since 
the solutions no longer obey finite speed of propagation due to the viscous effects in \eqref{model_equation}.
We will show, instead, that the solution map, which associates the solution $u$ to the initial data $u(0)$, is not continuous in the sense that
the $H^s$ norm of the solution can grow without bound as $t\to 0$ even as the initial data $u(0)$ has infinitesimally small $H^s$ norm.
More precisely, we will show that 
if $0 < s < s_{cr}$, then for every $\epsilon > 0$, there exists a solution $u$ of the viscous nonlinear wave equation \eqref{VNLWE} and a positive time $t$ such that
\begin{equation*}
||u(0)||_{H^{s}} < \epsilon, \ \ \ u_{t}(0) = 0, \ \ \ 0 < t < \epsilon, \ \ \ ||u(t)||_{H^{s}} > \epsilon^{-1},
\end{equation*}
for some $u(0) \in \mathcal{S}(\mathbb{R}^{2})$, where $\mathcal{S}(\mathbb{R}^{2})$ denotes the Schwartz class. Thus, the solution map for the equation \eqref{VNLWE} is not continuous at $(0, 0) \in H^{s}(\mathbb{R}^{2}) \times H^{s - 1}(\mathbb{R}^{2})$ for $0 < s < s_{cr}$. 
The main mechanism that allows this behavior is transfer of energy from low to high Fourier modes.
The proof is based on studying the visco-dispersive limit $\nu \to 0$ in:
$$
\partial_{tt} u - \nu^2 \Delta u + \nu \sqrt{-\Delta}\partial_{t}u + u^{p} = 0,
$$
and understanding how the {{unbounded growth of the $H^{s}$ norm in time of the solution}} for the visco-dispersive limit equation $\nu = 0$
translates into the solution behavior of the perturbed equation with $\nu > 0$ small, as $t \to 0$.
By utilizing the various symmetries of the solutions to the viscous nonlinear wave equation,
one can show that the {{unbounded growth of the $H^{s}$ norm in time}} for the visco-dispersive limit $\nu = 0$,
translates into similar growth of the 
 $H^s$ norm of the solution to the perturbed equation with $\nu > 0$ small. {{Using scaling symmetries of the viscous nonlinear wave (vNLW) equation, we can then allow this unbounded $H^{s}$ norm growth to occur at arbitrarily small times $t \to 0$,
 even when the initial data has infinitesimally small $H^s$ norm.}}
 The ill-posedness result is presented in Sec.~\ref{illposed}.
 
 In Sec.~\ref{Strichartz}, we start preparing for the probabilistic well-posedness result, by 
 deriving Strichartz estimates, which will be crucial in proving probabilistic well-posedness for the 
supercritical  nonlinear viscous wave equation with initial data belonging to 
$H^{s}(\mathbb{R}^{2}) \times H^{s - 1}(\mathbb{R}^{2})$, where $s$ is below the critical exponent ($-1/6< s \le 1/2=s_{cr}$).
The Strichartz estimates, which are {\sl{linear}} estimates derived for the linear homogeneous and linear inhomogeneous equations,
will be used to prove probabilistic well-posedness of the {\sl{nonlinear}} viscous wave equation
by considering the nonlinear term as an inhomogeneous (source) term for the corresponding linear problem, 
and then using a fixed point argument to show the existence of a unique solution, which we will then show also depends continuously on the initial data.

Strichartz estimates, which are estimates that control the $L^q_t L^r_x$ norms of solutions to linear dispersive equations in terms of the initial data and source terms,
are crucial for establishing well-posedness results for dispersive equations. 
These estimates were first determined by Strichartz in \cite{S} in the context of a Fourier restriction problem, 
and were found to be equivalent to estimates for the linear wave equation. 
They were extended to a more general abstract context by Keel and Tao in \cite{KT}. 
In Sec.~\ref{Strichartz}, we use the theorem by Keel and Tao in \cite{KT} for general Strichartz estimates to derive 
the corresponding estimates for the linear viscous wave equation (homogeneous and inhomogeneous). 
We show that the presence of viscous regularization influences Strichartz estimates in two ways:
the viscous Strichartz estimates hold for a larger range of admissible exponents $q$ and $r$,
and
the {{homogeneous}} viscous Strichartz estimates hold even in $1$D. 
This is interesting because it is well-known that the linear wave equation in one dimension does not possess such Strichartz estimates.
See Sogge \cite{Sg} for a detailed exposition of Strichartz estimates and generalized Strichartz estimates.

In Sec.~\ref{random} we use the ``global'' $L^q_tL^r_x$ Strichartz estimates and combine them with the ``local''
$C^0([0,T],H^s(\R^n))$ estimates, presented in Sec.~\ref{local_estimates},
to study probabilistic well-posedness for the supercritical nonlinear viscous wave equation. 
We show that under Wiener randomization of initial data, we can get a local probabilistic well-posedness result for the nonlinear 
viscous wave equation 
in $\R^2$ that brings the threshold exponent from $s_{cr} = 1/2$ down to $s > -1/6$.

This result is in the spirit of Burq and Tzvetkov \cite{BT}, who studied probabilistic well-posedness for the supercritical cubic {\sl{wave equation}}.
Ever since Christ, Colliander, and Tao in \cite{CCT} showed an ill-posedness result below the critical exponent for 
the nonlinear wave equation, there has been considerable interest in the behavior of such supercritical wave equations under randomization of the initial data. Burq and Tzvetkov showed in \cite{BT} that under appropriate randomization, 
one has local existence to the supercritical cubic wave equation almost surely for initial data in 
$H^{s}(M)\times H^{s-1}(M)$ appropriately randomized, 
where $s \ge 1/4$, and $M$ is a compact, three-dimensional manifold. The probabilistic randomization decreases the threshold for such local existence from the critical exponent $s_{cr} = 1/2$ for the cubic nonlinear wave equation to $s = 1/4$. 
%The cubic nonlinear wave equation in this context was considered on a compact three-dimensional Riemannian manifold $M$, on which there is hence a discrete spectrum for $-\Delta$. 

The randomization used by Burq and Tzvetkov in \cite{BT} relies heavily on the fact that there is a discrete spectrum for $-\Delta$ on a compact Riemannian manifold $M$. 
The randomization of initial data is then performed by 
randomizing each of the eigenfunction components of the initial data, using a sequence of independent random variables. See \cite{BT} for more details.

Our problem, however, 
is posed on $\mathbb{R}^{2}$, which means that we do not have a discrete spectrum for $-\Delta$, 
and so the randomization from Burq and Tzvetkov in \cite{BT} does not carry over to our case. 
However, we use an analogue of this discrete randomization, known as the Wiener randomization.
Wiener randomization was developed by B\'{e}nyi, Oh, and Pocovnicu in \cite{BOP}.
%It randomizes various frequency domains in $\mathbb{R}^{n}$ on the scale of unit volume cubes, using a partition of unity of translates of a compactly supported function. We provide more details in Section 3. 
Such a randomization was used in recent years to produce existence results for randomized initial data on Euclidean domains 
for the nonlinear wave equation by L\"{u}hrmann and Mendelson in \cite{LM},
 and for the incompressible Euler equations by Wang and Wang in \cite{WW}, for example.
We describe 
the Wiener randomization and its properties in Sec.~\ref{random}, 
and show how it can be used to produce a { probabilistic} local existence result for the 
supercritical nonlinear viscous wave equation.
We prove that for a random perturbation of initial data in $H^s(\R^2)\times H^{s-1}(\R^2)$
based on Wiener randomization, the supercritical quintic nonlinear viscous wave equation is well-posed almost surely,
even for $-1/6<s \le s_{cr} = 1/2$.
In contrast with the cubic {\sl{wave equation}} for which supercritical probabilistic well-posedness holds for the exponents $s \ge 1/4$
in the case of three-dimensional compact manifolds without boundaries, see \cite{BT}, 
we show that the viscous dissipation allows us to 
bring the threshold exponent from $s_{cr} = 1/2$ all the way down to negative Sobolev space exponents $s$ for which $s > -1/6$.

\if 1 = 0
In this paper, we consider a partial differential equation associated with a model fluid structure interaction problem. In this problem, we consider the motion of an infinitely thin elastic plate, represented by the plane $\mathbb{R}^{2}$, under the influence of a fluid beneath the plate in the lower half plane of $\mathbb{R}^{3}$. To describe the motion of the elastic plate, we use a nonlinear wave equation, where the nonlinearity describes the nonlinear relationship between the displacement and the restoring elastic force. To describe the motion of the fluid, we use a simplified model for the fluid under a fully developed flow, which is given by Laplace's equation. In particular, if we define $\mathbf{u} = (u_{x}, u_{y})$ to be the fluid velocity and if we define $v$ to be the displacement of the elastic membrane from equilibrium, we have the following equations.
\begin{equation*}
\Delta \mathbf{u} = 0 \ \ \ \ \text{ in } \mathbb{R}^{2} \times (-\infty, 0)
\end{equation*}
\begin{equation*}
\partial_{tt}v - \Delta v + v^{p} = f(t) \ \ \ \ \text{ on } \mathbb{R}^{2}
\end{equation*}
where $f(t)$ is a inhomogeneous term representing the force exerted on the elastic membrane. 

\medskip

The displacement of the elastic plate $u$ and the fluid velocity $\mathbf{v}$ are coupled in the following way. First, we impose the no-slip condition, which states that at the fluid interface, the normal component of the fluid velocity is equal to the velocity of the elastic plate at that point. In particular, the no-slip condition is
\begin{equation*}
u_{y} = \partial_{t}v \ \ \ \ \text{ on } \mathbb{R}^{2}
\end{equation*} 
In addition, we impose a dynamic coupling condition, which describes the relation between the fluid velocity and the force the fluid exerts on the elastic plate. To describe this force, we use the familiar Dirichlet-Neumann operator. In particular, the dynamic coupling condition is
\begin{equation*}
f(t) = -\sqrt{-\Delta} u_{y}
\end{equation*}
where $\sqrt{-\Delta}$ is the Dirichlet-Neumann operator on $\mathbb{R}^{2}$, as described by Caffarelli and Silvestre in \cite{CS} in a much more general context.  

\medskip

In particular, combining both of the coupling conditions, 
\begin{equation*}
f(t) = -\sqrt{-\Delta}\partial_{t}v
\end{equation*}
Therefore, the fluid-structure interaction problem above can be reduced to a single equation, 
\begin{equation*}
\partial_{tt}v - \Delta v + \sqrt{-\Delta}\partial_{t}v + v^{p} = 0 \ \ \ \ \text{ on } \mathbb{R}^{2}
\end{equation*}
which is the equation we will consider in this paper. Since this is a second order equation, we need two initial conditions
\begin{equation*}
v(0, \cdot) = f(x) \ \ \ \ \ \text{ and } \ \ \ \ \ \partial_{t}v(0, \cdot) = g(x)
\end{equation*}
We will consider initial data $(f, g) \in \mathcal{H}^{s}(\mathbb{R}^{2}) = H^{s}(\mathbb{R}^{2}) \times H^{s - 1}(\mathbb{R}^{2})$. This is the initial value problem that we will consider throughout this paper. We will refer to this equation as the elastic plate equation.

\medskip

A first approach to such an equation is to look for symmetries. In the theory of nonlinear dispersive equations, a common procedure is to look for scaling symmetries. As we will see later, the scaling symmetry
\begin{equation*}
u(t, x) \to \lambda^{\frac{2}{p - 1}}u(\lambda t, \lambda x)
\end{equation*}
preserves solutions of the above partial differential equation. This symmetry preserves the $H^{s_{cr}}$ norm of the solution, for a critical exponent
\begin{equation*}
s_{cr} = 1 - \frac{2}{p - 1}
\end{equation*}
Note that this is the same critical exponent as for the defocusing nonlinear wave equation.
\begin{equation*}
\partial_{tt}v - \Delta v + v^{p} = 0 
\end{equation*}
We will see that in terms of energy inequalities and ill-posedness behavior, the elastic plate equation shares similar properties to the wave equation, but the elastic plate equation also exhibits novel behavior that arises from the dissipation of energy by the influence of the fluid. 

\medskip

The presence of a critical exponent is crucial for the analysis of nonlinear partial differential equations, in particular for dispersive equations. Above the critical exponent, dispersive equations usually exhibit well-posedness, which can be established from fixed point arguments in combination with dispersive estimates for the linear problem. For example, such well-posedness results can be established for various nonlinear partial differential equations. Such a result was established first for the nonlinear Schr\"{o}dinger equation above the critical exponent by Cazenave and Weissler in \cite{CW}, and a well-posedness result for the nonlinear wave equation above the critical exponent was established by Lindblad and Sogge in \cite{LS}. It is likely that a similar result could be established for the elastic plate equation above, though we will not pursue this route in this paper. 

\medskip

Instead, we will consider the subcritical behavior of the equation above. Below the critical exponent which preserves the $H^{s_{cr}}$ norm, we would heuristically expect some ill-posedness result. By well-posedness in the Hadamard sense, we mean existence of a solution, uniqueness of a solution, and continuity of the solution map from $\mathcal{H}^{s} = H^{s} \times H^{s - 1}$ into some appropriate function space. Some formulations in addition require the stronger requirement of uniform continuity of the solution map, but in this case, we will just consider continuity. 

\medskip

Such ill-posedness behavior was most famously considered for certain ranges of $s$ for the initial data in $\mathcal{H}^{s}$ for the nonlinear wave equation and the nonlinear Schr\"{o}dinger equation by Christ, Colliander, and Tao in \cite{CCT}. In Section 1, we show that a similar ill-posedness result can be obtained using a similar procedure for the elastic plate equation for initial data in $\mathcal{H}^{s}$ with $0 < s < s_{cr}$. We in particular will show lack of continuity of the solution map for the elastic plate equation. 

\medskip

As a remark, we note that for nonlinear wave equation solutions that satisfy finite speed of propagation, it is possible to construct initial data for which there is so-called instantaneous blowup using the lack of continuity of the solution map. In particular, for a certain choice of initial data, there exist no $T > 0$ for which there is a weak solution to the nonlinear wave equation in $C([0, T]; H^{s})$. For more details, see the discussion in \cite{CCT}. However, we would not expect instantaneous blowup for the elastic plate equation since the effect of the fluid makes it so that solutions of the elastic plate equation do not obey finite speed of propagation. 

\medskip

In this paper, we also consider dispersive estimates known as Strichartz estimates, which are estimates that control the space and time norms of solutions to linear dispersive equations. These estimates are crucial for establishing well-posedness results for dispersive equations. These estimates were first determined by Strichartz in \cite{S} in the context of a Fourier restriction problem, and were found to be equivalent to estimates for the linear wave equation. These Strichartz estimates were extended to a more general abstract context by Keel and Tao in \cite{KT}. 

\medskip

In Section 2, we use the theorem by Keel and Tao in \cite{KT} for general Strichartz estimates to derive estimates for the elastic plate equation. We will see that while the inhomogeneous elastic plate equation follows the same estimate as the inhomogeneous wave equation, the homogeneous elastic plate equation - in contrast with the homogeneous wave equation - has a wide range of $\sigma$-admissible exponents which will produce Strichartz estimates. (See also Sogge \cite{Sg} for a more detailed exposition of Strichartz estimates and generalized Strichartz estimates.)

\medskip

Specifically, for the nonlinear wave equation for which Christ, Colliander, and Tao showed an ill-posedness result below the critical exponent, there has been a recent interest in the behavior of such supercritical wave equations under randomization of the initial data. As pursued in Burq and Tzvetkov in \cite{BT}, under appropriate randomization, one has local existence to the supercritical cubic wave equation almost surely for initial data in $H^{s}(M)$ appropriately randomized, where $s \ge 1/4$. The probabilistic randomization decreases the threshold for such local existence from the critical exponent $s_{cr} = 1/2$ for the cubic nonlinear wave equation to $s = 1/4$. The cubic nonlinear wave equation in this context was considered on a compact Riemannian manifold $M$, on which there is a discrete spectrum for $-\Delta$. 

\medskip

The randomization used by Burq and Tzvetkov in \cite{BT} relies heavily on the fact that there is a discrete spectrum for $-\Delta$ on a compact Riemannian manifold $M$. One then randomizes each of the eigenfunction components of the initial data using a sequence of independent random variables. See \cite{BT} for more details.

\medskip

However, since our problem is posed on $\mathbb{R}^{2}$ which does not have a discrete spectrum for $-\Delta$, the randomization from Burq and Tzvetkov in \cite{BT} does not carry over to our case. However, an analogue of this discrete randomization, known as the Wiener randomization, was developed by B\'{e}nyi, Oh, and Pocovnicu in \cite{BOP}, which randomizes various frequency domains in $\mathbb{R}^{n}$ on the scale of unit volume cubes, using a partition of unity of translates of a compactly supported function. We provide more details in Section 3. Such a randomization has been used in recent years to produce existence results for randomized initial data on Euclidean domains, for the nonlinear wave equation in \cite{LM} by L\"{u}hrmann and Mendelson and for the incompressible Euler equations in \cite{WW} by Wang and Wang, for example.

\medskip

In Section 3, we will describe the Wiener randomization and its properties, and show that it can be used to produce a local existence result for the elastic plate equation above that brings the threshold from $s_{cr} = 1/2$ down to $s > -1/6$, encompassing even negative Sobolev spaces. We will then interpret this result in a corollary to show that the theorem we prove describes some form of ``probabilistic well-posedness."
\fi

%%%%%%%%%%%%%%%%%%%%%%%
\section{Ill-posedness for viscous nonlinear wave equation}\label{illposed}

We study the Cauchy problem for the viscous nonlinear wave equation
\begin{equation}\label{VNLWE}
\begin{array}{rcl}
\partial_{tt} u - \Delta u + \sqrt{-\Delta}\partial_{t}u + u^{p} &=& 0 \ {\rm on} \ \R^n,\\
u(0,x) = f(x), \quad u_t(0,x) &=& g(x),
\end{array}
\end{equation}
where $p > 1$ is a positive odd integer. {{The case of $n = 2$ corresponds to our given fluid-structure interaction model, but we will use general $n$, as all arguments here hold for general $n$.}}
To study ill-posedness as specified in (P1) of Sec.~\ref{introduction}, 
we begin by investigating the scaling symmetries of this equation and determine the critical exponent $s_{cr}$ that preserves the homogeneous Sobolev $\dot{H}^{s}(\mathbb{R}^{n})$ norm of solutions of 
\eqref{VNLWE} {under this scaling symmetry}. 
We recall that the homogeneous Sobolev space $\dot{H}^{s}(\mathbb{R}^{n})$ is defined as a completion of $C_0^\infty(\R^n)$ in the norm
$$
\|f\|_{\dot{H}^s} := (2\pi)^{-n/2}\| | \xi|^s \hat{f}(\xi) \|_{L^2(\R^n)}.
$$
We start by first noticing that for positive $\lambda$, we have $\sqrt{-\Delta} \left[u(\lambda x)\right]  = \lambda (\sqrt{-\Delta}u)(\lambda x)$. Indeed,
\begin{equation*}
\sqrt{-\Delta} \left[u(\lambda x)\right] 
= \frac{1}{(2\pi)^{2}} \int_{\mathbb{R}^{2}}\lambda |\xi|\widehat{u}(\xi) e^{i\lambda x \cdot \xi} d\xi = \lambda (\sqrt{-\Delta}u)(\lambda x).
\end{equation*}
Therefore, the following scaling map:
\begin{equation}\label{scaling_map}
u(t, x) \to \lambda^{\frac{2}{p - 1}}u(\lambda t, \lambda x)
\end{equation}
preserves solutions to the partial differential equation above. 
A calculation shows that the critical exponent for equation \eqref{VNLWE}
 that preserves the homogeneous Sobolev $\dot{H}^{s}(\mathbb{R}^{n})$ norm of this scaling, is given by:
\begin{equation}\label{s_critical}
s_{cr} = \frac{n}{2} - \frac{2}{p - 1}.
\end{equation}
We note that this is exactly the critical exponent for the defocusing nonlinear wave equation \cite{CCT}
\begin{equation}\label{defocusingNWE}
\partial_{tt}u - \Delta u + u^{p} = 0 \ {\rm on} \ \R^n.
\end{equation}
Christ, Colliander, and Tao have shown in \cite{CCT} that this defocusing nonlinear wave equation 
for odd integers $p > 1$ exhibits ill-posedness for initial data $(f, g) \in H^{s} \times H^{s - 1}$ where $0 < s < s_{cr}$. 
We ask whether equation \eqref{VNLWE}  exhibits similar  behavior. Intuitively, one can show through an energy estimate that even though both the nonlinear wave equation \eqref{NLW} and viscous nonlinear wave equation \eqref{VNLWE} share the same critical exponent, the presence of the viscous term in \eqref{VNLWE} dissipates energy, which might yield better estimates and a different ill-posedness result. 
We will show that this is, in fact, not  the case, although there are some differences. More precisely, we will show that 
the viscous nonlinear wave equation \eqref{VNLWE} has the same ill-posedness property as the nonlinear wave equation \eqref{NLW},
indicating that {\emph{nonlinear effects are dominant over viscous regularization}} 
associated with the fluid-structure coupling. { However, as we will see below, the viscous contribution has the potential to \emph{slow down} the speed of how ``fast" the $H^{s}$ norm of the solution grows.}

A way to show this is to use a variation of the dispersive limit argument in \cite{CCT}. 
%In particular, Christ, Colliander, and Tao \cite{CCT} consider the family of equations
%\begin{equation*}
%\partial_{tt}v - \nu^{2} \Delta v + v^{p} = 0
%\end{equation*}
%where in the ``limit" as $\nu \to 0$, the solution can be written out explicitly. 
More precisely, we  consider 
\begin{equation}\label{dispVNLWE}
\partial_{tt}u - \nu^{2}\Delta u + \nu \sqrt{-\Delta}\partial_{t}u + u^{p} = 0
\end{equation}
for $\nu > 0$, and study the solution in the {\emph{visco-dispersive limit}} when $\nu \to 0$.
In the visco-dispersive limit, one formally gets a limiting equation
\begin{equation}\label{displimit}
\partial_{tt}u + u^{p} = 0,
\end{equation}
whose solution can be written out explicitly, {{and it can be shown that the solution exhibits rapid growth of the $H^{s}$ norm in time}}.
The goal is to then show that for small $\nu$'s, solutions {to \eqref{dispVNLWE} with $\nu > 0$ small} are close to the solution with $\nu = 0$.
Note that equation \eqref{dispVNLWE} can be considered as a ``perturbation'' of the visco-dispersive limit equation
\eqref{displimit}.

To show that for small $\nu$'s, solutions of \eqref{dispVNLWE} are close to the solution with $\nu = 0$, let $\phi_{0}$ be any smooth compactly supported function on $\mathbb{R}^{n}$. 
The initial value problem we consider is
\begin{align}\label{IVPdispVNLWE}
\partial_{tt}u & - \nu^{2} \Delta u + \nu \sqrt{-\Delta}\partial_{t}u + u^{p} = 0 \text{ on } \mathbb{R}^{n},\\
&u(0, x) = \phi_{0}(x), \ \ \ \ \ \partial_{t}u(0, x) = 0,
\nonumber
\end{align}
and the visco-dispersive limit of this initial value problem is
\begin{align}\label{IVPdisplimit}
&\partial_{tt}u + u^{p} = 0, \\
u(0, x) = \ &\phi_{0}(x), \ \ \ \ \ \partial_{t}u(0, x) = 0.
\nonumber
\end{align}
The solution of \eqref{IVPdisplimit} is given by
\begin{equation}\label{displimitexpsol}
\phi^{(0)}(t, x) = \phi_{0}(x)V\left(t|\phi_{0}(x)|^{\frac{p - 1}{2}}\right) = \phi_{0}(x)V\left(t\left(\phi_{0}(x)\right)^{\frac{p - 1}{2}}\right),
\end{equation}
where $V$ is the smooth periodic solution to
\begin{equation*}
V'' + V^{p} = 0, \ \ \ \ \ V(0) = 1, \ \ \ \ \ V'(0) = 0.
\end{equation*}
Note that $V$ is even, which is why we can remove the absolute values in \eqref{displimitexpsol}.
 
Crucial for the proof is the following scaling property of the visco-dispersive limit equation \eqref{IVPdispVNLWE}: {{if $u(t, x)$ is a solution to \eqref{IVPdispVNLWE}, then
the entire one-parameter family of functions 
\begin{equation*}
\lambda^{-\frac{2}{p - 1}}u(\lambda^{-1}t, \lambda^{-1}\nu x), \quad \lambda > 0,
\end{equation*}
obtained via the scaling map \eqref{scaling_map}, is a solution to the viscous nonlinear wave equation.}}

We want to argue that for small values of $\nu$, solutions to both initial value problems \eqref{IVPdispVNLWE} and \eqref{IVPdisplimit} are close for {a bounded set of} times, {which increases as $\nu \to 0$}. We make this statement precise in the following proposition. 

{\bf Note on notation.} In what follows, we use $k$ for the exponent in $H^{k}$ whenever we want to emphasize that the Sobolev exponent is an integer, while we use $s$ for the Sobolev exponent in $H^{s}$ when the Sobolev exponent can be a general real number, possibly fractional.

\begin{proposition}\label{closeness}
Let $p > 1$ be a positive odd integer, and let $k \ge n + 1$ be an integer.
Suppose  $\phi_{0}$ is a compactly supported smooth function, and let $\phi^{(0)}$ be the solution to \eqref{IVPdisplimit}. Given any $\delta > 0$, there exist $C, c > 0$ depending on $p, k, \delta,$ {{and $\phi_{0} \in C_{0}^{\infty}(\mathbb{R}^{n})$}}, such that for all $0 < \nu \le c$, there exists a solution $\phi(t, x)$ of \eqref{IVPdispVNLWE} such that 
\begin{equation*}
||\phi(t) - \phi^{(0)}(t)||_{H^{k}(\mathbb{R}^{n})} + ||\phi_{t}(t) - \phi_{t}^{(0)}(t)||_{H^{k}(\mathbb{R}^{n})} \le C|\nu|^{1 - \delta}, \quad {\rm{for \ all}} \ 0 \le t \le c|\text{log}\nu|^{c}.
\end{equation*}
\end{proposition}

\begin{proof} The proof is based on energy methods. %See also Colliander, Christ, and Tao for the Schr\"{o}dinger equation \cite{CCT}.
We begin by defining
\begin{equation*}
w(t, x) = \phi(t, x) - \phi^{(0)}(t, x),
\end{equation*}
where $\phi(t, x)$ is the solution to \eqref{IVPdispVNLWE} and $\phi^{(0)}(t, x)$ is the solution to \eqref{IVPdisplimit}. 
Then, $w$ satisfies the following initial value problem,
where we use $G$ to denote the nonlinearity $G(z) = z^{p}$:
\begin{align}\label{IVPproof}
\partial_{tt}w - \nu^{2}\Delta w + \nu \sqrt{-\Delta} \partial_{t}w = \nu^{2}\Delta \phi^{(0)} &- G(\phi^{(0)} + w) + G(\phi^{(0)}) - \nu\sqrt{-\Delta}\partial_{t}\phi^{(0)},\\
w(0, x) &= 0, \ \ \ \ \ \partial_{t}w(0, x) = 0.
\nonumber
\end{align}
One can use energy estimates to show that 
this initial value problem has a solution in $H^{k + 1} \times H^{k}$, { as long as the $H^{k + 1} \times H^{k}$ norm of the solution is bounded}.
The main idea to prove this is to first obtain the existence of a unique
{\sl{local}} solution, and then extend the solution locally whenever the energy norm is bounded.
 The proof of this is by a Picard iteration argument. 
For a full proof of the existence of a solution to \eqref{IVPproof} as long as the $H^{k + 1} \times H^{k}$ norm of the solution is bounded, we refer the reader to the Appendix. 

We will use the energy method to estimate the size of $w$. We note that we can derive an energy estimate for the inhomogeneous linear problem
\begin{equation}\label{energyestderivation}
u_{tt} - \nu^{2}\Delta u + \nu\sqrt{-\Delta} \partial_{t}u = F(t, x)
\end{equation}
as follows. Define the {{$\nu$-wave energy}} of a solution $u$ (depending on $\nu$) by
\begin{equation*}
E_{\nu}(u(t)) := \int \frac{1}{2} |u_{t}(t, x)|^{2} + \frac{\nu^{2}}{2}|\nabla u(t, x)|^{2} dx.
\end{equation*}
By multiplying the equation \eqref{energyestderivation} by $u_{t}$ and integrating in space, we get
\begin{equation*}
\int \partial_{tt}u \cdot \partial_{t}u - \nu^{2} \Delta u (\partial_{t} u) dx = \int F(t, x) \partial_{t}u(x) dx - \nu \int \sqrt{-\Delta}\partial_{t}u \cdot \partial_{t}u dx,
\end{equation*}
or equivalently (if $u$ decays rapidly at infinity)
\begin{equation*}
\frac{d}{dt} (E_{\nu}(u(t))) = \int F(t, x) \partial_{t}u(x) dx - \nu \int \sqrt{-\Delta}\partial_{t}u \cdot \partial_{t}u dx.
\end{equation*}
Using Plancherel's theorem,
\begin{align*}
\frac{d}{dt} (E_{\nu}(u(t))) &= \int F(t, x) \partial_{t}u(x) dx - \nu ||u_{t}||^{2}_{\dot{H}^{1/2}} \le \int F(t, x) \partial_{t}u(x) dx \\
&\le ||F(t, \cdot)||_{L^{2}} \cdot \sqrt{2}(E_{\nu}(u(t)))^{1/2} \le ||F(t, \cdot)||_{L^{2}} \cdot 2(E_{\nu}(u(t)))^{1/2},
\end{align*}
where we applied the Cauchy-Schwarz inequality. Therefore,
\begin{equation*}
\frac{d}{dt}\left(E_{\nu}^{1/2}(w(t))\right) \le ||F(t, \cdot)||_{L^{2}},
\end{equation*}
which gives the desired energy inequality. Using the fact that derivatives commute with $(\partial_{tt} - \nu^{2}\Delta + \nu\sqrt{-\Delta} \partial_{t})$ (since $\sqrt{-\Delta}$ is a Fourier multiplier and hence commutes with ordinary derivatives which are also Fourier multipliers), we can get an estimate on the derivatives too. In particular, if we define
\begin{equation*}
E_{\nu, k}(w(t)) := \sum_{|\alpha| \le k} E_{\nu}(\partial^{\alpha}_{x}w(t)),
\end{equation*}
we have the energy inequality
\begin{equation*}
\frac{d}{dt}\left(E_{\nu, k}^{1/2}(w(t))\right) \le C||F(t, \cdot)||_{H^{k}},
\end{equation*}
{ for a constant $C$ depending only on $k$}. Applying this energy inequality to \eqref{IVPproof} gives
\begin{equation}\label{applyenergyineq}
\frac{d}{dt}\left(E_{\nu, k}^{1/2}(w(t))\right) \le C\left(||\nu^{2}\Delta \phi^{(0)}||_{H^{k}} + ||\nu \sqrt{-\Delta} \partial_{t}\phi^{(0)}||_{H^{k}} + ||G(\phi^{(0)} + w)(t) - G(\phi^{(0)})(t)||_{H^{k}}\right).
\end{equation}
In addition, we note that we can get estimates on spatial derivatives as follows:
\begin{equation}\label{spatialest}
||w(t)||_{H^{k}} \le \int_{0}^{t} ||w_{t}(t')||_{H^{k}} dt' \le C\int_{0}^{t}E_{\nu, k}^{1/2}(w(t'))dt' \le Cte(t),
\end{equation}
where
\begin{equation*}
e(t) := \text{sup}_{0 \le t' \le t} E_{\nu, k}^{1/2}(w(t')).
\end{equation*}
Note that $\frac{d}{dt}e(t) \le \text{max}\left(\frac{d}{dt}\left(E_{\nu, k}^{1/2}(w(t))\right), 0\right)$, and so, since the right hand side of \eqref{applyenergyineq} is nonnegative, we have
\begin{equation}\label{mainenergyineq}
\frac{d}{dt}e(t) \le C\left(||\nu^{2}\Delta \phi^{(0)}||_{H^{k}} + ||\nu \sqrt{-\Delta} \partial_{t}\phi^{(0)}||_{H^{k}} + ||G(\phi^{(0)} + w)(t) - G(\phi^{(0)})(t)||_{H^{k}}\right).
\end{equation}
Recall the form of 
\begin{equation*}
\phi^{(0)}(t, x) = \phi_{0}(x)V\left(t\left(\phi_{0}(x)\right)^{\frac{p - 1}{2}}\right),
\end{equation*}
and note that $|\phi_{0}(x)|^{\frac{p - 1}{2}}$ is smooth since $p > 1$ is odd, so $\frac{p - 1}{2}$ is a positive integer. We can deduce (recalling that $\phi_{0}$ is fixed) that 
\begin{equation}\label{energyest1}
||\nu^{2}\Delta \phi^{(0)}||_{H^{k}} \le C\nu^{2}(1 + |t|)^{k + 2} \le C\nu^{2}(1 + |t|)^{C}.
\end{equation}
Using H\"{o}lder and Sobolev inequalities as in \cite{CCT} at the bottom of pg.~11, one obtains
\begin{equation}\label{CCTest}
||G(\phi^{(0)} + w)(t) - G(\phi^{(0)})(t)||_{H^{k}} \le C(1 + |t|)^{C}||w(t)||_{H^{k}}(1 + ||w(t)||_{H^{k}})^{p - 1}.
\end{equation}
Estimate \eqref{spatialest} now implies
\begin{equation}\label{energyest2}
||G(\phi^{(0)} + w)(t) - G(\phi^{(0)})(t)||_{H^{k}} \le C(1 + |t|)^{C}(e(t) + e(t)^{p}).
\end{equation}
Finally, we have to estimate $||\nu \sqrt{-\Delta} \partial_{t}\phi^{(0)}||_{H^{k}}$. Note that
\begin{align}\label{energyest3}
||\nu \sqrt{-\Delta} \partial_{t}\phi^{(0)}||_{H^{k}} &= \nu \frac{1}{(2\pi)^{2}} \int_{\mathbb{R}^{2}} (1 + |\xi|^{2})^{k} |\xi|^{2} |\partial_{t}\phi^{(0)}(t, \xi)|^{2} d\xi 
\nonumber
\\
&\le \nu \frac{1}{(2\pi)^{2}} \int_{\mathbb{R}^{2}} (1 + |\xi|^{2})^{k + 1} |\partial_{t}\phi^{(0)}(t, \xi)|^{2}d\xi 
\nonumber
\\
&\le \nu ||\partial_{t}\phi^{(0)}||_{H^{k + 1}} \le C\nu(1 + |t|)^{k + 1} \le C\nu(1 + |t|)^{C},
\end{align}
\if 1 = 0

%NOTE: I don't think we need this. There is a simpler %way to do this in Fourier space.

We do this using a simple interpolation lemma, that will convert the problem of analyzing the operator $\sqrt{-\Delta}$ into a problem posed entirely in terms of spatial derivatives. 

\begin{lemma}[Interpolation lemma]
Consider some function $F(t, x)$. If $||F(t, \cdot)||_{H^{k}} \le a(t)$ and $||(\Delta F)(t, \cdot)||_{H^{k}} \le b(t)$ for functions $a(t)$ and $b(t)$, then
\begin{equation*}
||(\sqrt{-\Delta}F)(t, \cdot)||_{H^{k}} \le (a(t)b(t))^{1/2}
\end{equation*}
\end{lemma}

\begin{proof}
We have that
\begin{multline*}
||(\sqrt{-\Delta}F)(t, \cdot)||^{2}_{H^{k}} = \int (1 + |\xi|^{2})^{k}|\xi|^{2}|\widehat{F}(t, \xi)|^{2} d\xi \\
\le \left(\int (1 + |\xi|^{2})^{k}|\xi|^{4}|\widehat{F}(t, \xi)|^{2} d\xi\right)^{1/2} \left(\int (1 + |\xi|^{2})^{k}|\widehat{F}(t, \xi)|^{2} d\xi\right)^{1/2} \\
= ||(\Delta F)(t, \cdot)||_{H^{k}} ||F(t, \cdot)||_{H^{k}} \le a(t)b(t)
\end{multline*}
where in the first inequality, we used Cauchy-Schwarz, splitting the product into 
\begin{equation*}
(1 + |\xi|^{2})^{k}|\xi|^{2}|\widehat{F}(t, \xi)|^{2} = \left[(1 + |\xi|^{2})^{k/2}|\xi|^{2}|\widehat{F}(t, \xi)|\right] \cdot \left[(1 + |\xi|^{2})^{k/2} |\widehat{F}(t, \xi)| \right]
\end{equation*}
Taking square roots gives the desired result. 
\end{proof}

Using this lemma, we can estimate $||\nu \sqrt{-\Delta} \partial_{t}\phi^{(0)}||_{H^{k}}$ where we recall that
\begin{equation*}
\phi^{(0)}(t, x) = \phi_{0}(x)V\left(t\left(\phi_{0}(x)\right)^{\frac{p - 1}{2}}\right)
\end{equation*}
where $V$ is smooth and periodic (since $p$ is odd). \textit{In particular, $V$ and all of its derivatives are uniformly bounded}. Note that
\begin{equation*}
\partial_{t}\phi^{(0)}(t, x) = \left(\phi_{0}(x)\right)^{\frac{p + 1}{2}} V'\left(t\left(\phi_{0}(x)\right)^{\frac{p - 1}{2}}\right)
\end{equation*}
and $\partial_{t}\phi^{(0)}(t, x)$ is smooth in both $t$ and $x$, since $p$ is odd. Applying Lemma 1.1, we have
\begin{equation*}
||\nu \sqrt{-\Delta} \partial_{t}\phi^{(0)}||_{H^{k}} \le \nu ||\Delta(\partial_{t}\phi^{(0)})||_{H^{k}}^{1/2}||\partial_{t}\phi^{(0)}||_{H^{k}}^{1/2}
\end{equation*}
To estimate the right hand side, we note that by the Leibnitz rule,
\begin{equation*}
||\partial_{t}\phi^{(0)}(t, x)||_{H^{k}} \le C(1 + |t|)^{k}
\end{equation*}
\begin{equation*}
||\Delta(\partial_{t}\phi^{(0)}(t, x)||_{H^{k}} \le C(1 + |t|)^{k + 2}
\end{equation*}
since each term that we get from the Leibnitz rule computation of an $N$th order spatial derivative of $\partial_{t}\phi^{(0)}$ is of the form $t^{j}f(x)V^{(j)}\left(t\left(\phi_{0}(x)\right)^{\frac{p - 1}{2}}\right)$, where $f$ is a compactly supported function independent of $t$ (that differs depending on the term in the Leibnitz rule expansion), and $0 \le j \le N$. We can then estimate the $L^{2}$ norm of terms of this type by noting that all derivatives of $V$ are uniformly bounded, and by noting that the $L^{2}$ norm of $f(x)$ is bounded and independent of $t$, since $f(x)$ is independent of $t$. Substituting everything into the right hand side of the previous inequality,
\begin{equation}
||\nu \sqrt{-\Delta} \partial_{t}\phi^{(0)}||_{H^{k}} \le C\nu(1 + |t|)^{k + 1} \le C\nu(1 + |t|)^{C}
\end{equation}

\fi
and so by using \eqref{mainenergyineq}, \eqref{energyest1}, \eqref{energyest2}, and \eqref{energyest3}, we get
\begin{equation*}
\frac{d}{dt}e(t) \le C(1 + |t|)^{C}(\nu^{2} + \nu + e(t) + e(t)^{p}).
\end{equation*}
Taking $c$ (and hence $\nu$, since $0 \le \nu < c$) sufficiently small, and using an a priori estimate $e(t) \le 1$ (which we will recover later by bootstrap, and which is valid by continuity for small $t$ since $e(t) = 0$), we get
\begin{equation*}
\frac{d}{dt}e(t) \le C(1 + |t|)^{C}(\nu + e(t)).
\end{equation*}
Gronwall's inequality gives that
\begin{equation*}
e(t) \le C\nu \cdot \text{exp}(C(1 + |t|)^{C}).
\end{equation*}
Now we choose $c$ sufficiently small such that for all $0 \le t < c|\text{log}\nu|^{c}$ and $0 < \nu \le c$, we have 
\begin{equation*}
C\nu \cdot \text{exp}(C(1 + |t|)^{C}) \le C'\nu^{1 - \delta},
\end{equation*}
\begin{equation*}
C\nu \cdot \text{exp}(C(1 + |t|)^{C}) \le 1/2.
\end{equation*}
To see that this is possible, note that we can choose $c$ sufficiently small so that $cC << 1$, $c << 1$, and $Cc^{C} << \frac{\delta}{2}$. Then, for all $0 \le t < c|\text{log}\nu|^{c}$ and $0 < \nu \le c$, 
\begin{align*}
C\nu \cdot \text{exp}(C(1 + |t|)^{C}) &\le C\nu \cdot \text{exp}(C + C|t|^{C}) \le Ce^{C}\nu \cdot \text{exp}(Cc^{C}|\text{log}\nu|^{cC}) 
\\
&\le Ce^{C}\nu \cdot \exp(Cc^{C}|\text{log}\nu|) = Ce^{C}\nu \cdot \exp(-Cc^{C}\text{log}\nu) = Ce^{C}\nu \cdot \nu^{\left(-Cc^{C}\right)} 
\\
&\le Ce^{C}\nu^{1 - \frac{\delta}{2}} \le Ce^{C}\nu^{1 - \delta} = C'\nu^{1 - \delta}.
\end{align*}
Note that $C'$ is independent of the sufficiently small $c$ we choose. Then, by making $c$ even smaller if necessary, we can also get $C'\nu^{1 - \delta} << 1/2$ for all $0 < \nu \le c$ to get the second inequality above. Then a bootstrap continuity argument (based on the second inequality above, which implies $e(t) \le 1/2$) can be used to justify the a priori assumption that $e(t) \le 1$ for the $t$ we are considering in $0 \le t < c|\text{log}\nu|^{c}$, $0 < \nu \le c$. Combining the definition of $e(t)$ and estimate \eqref{spatialest} gives the desired result. (For estimate \eqref{spatialest}, we note that with the choices above, $te(t) \le C''\nu^{1 - \delta}$ still for all $0 \le t < c|\text{log}\nu|^{c}$, since $|\text{log}\nu|^{c} \le |\text{log}\nu| \le C_{r}\nu^{-r}$ for any $r > 0$ and since we showed earlier that with our choice of $C$, we actually had $e(t) \le C'\nu^{1 - \frac{\delta}{2}}$. In particular, we can set $r = \delta/2$.)
\end{proof}

\begin{remark}

Christ, Colliander, and Tao showed an analogous result for the nonlinear wave equation \eqref{NLW}, but the resulting exponent in the analogue of the lemma above is $\nu^{2 - \delta}$. {{This shows that the dissipative effect of the $\nu \sqrt{-\Delta}\partial_{t}u$ term makes the solution of the perturbed viscous nonlinear wave equation \eqref{IVPdispVNLWE} ``less close" than the solution of the analogously perturbed initial value problem for the nonlinear wave equation \eqref{NLW} (which is just \eqref{IVPdispVNLWE} without the $\nu \sqrt{-\Delta}\partial_{t}u$ term) to the dispersive limit solution. Thus, there is the potential for the solution of the perturbed initial value problem \eqref{IVPdispVNLWE} to have $H^{s}$ norm growing less fast than for the corresponding perturbed initial value problem for the nonlinear wave equation.}}

\end{remark}

The important feature of the above lemma is that for $\delta > 0$ sufficiently small, $\nu^{1 - \delta}$ still goes to $0$ as $\nu$ goes to $0$. Therefore, the proofs of {ill-posedness} in Christ, Colliander, and Tao \cite{CCT} still apply to this equation. In particular, we have the following result, {{which holds for general $n$, although $n = 2$ will correspond to the specific case for our given fluid-structure interaction model.}}

\begin{theorem}\label{ill_posedness}
Let $p > 1$ be a positive odd integer. If $0 < s < s_{cr}$, where $s_{cr}$ is given by \eqref{s_critical},
then for every $\epsilon > 0$, there exists a solution $u$ of the viscous nonlinear wave equation \eqref{VNLWE} and a positive time $t$ such that
\begin{equation*}
||u(0)||_{H^{s}} < \epsilon, \ \ \ u_{t}(0) = 0, \ \ \ 0 < t < \epsilon, \ \ \ ||u(t)||_{H^{s}} > \epsilon^{-1},
\end{equation*}
for some $u(0) \in \mathcal{S}(\mathbb{R}^{n})$, where $ \mathcal{S}(\mathbb{R}^{n})$ is the Schwartz class. 
Thus, the solution map for the equation \eqref{VNLWE} is not continuous at 
$(u(0),u_t(0))=(0, 0) \in H^{s}(\mathbb{R}^{n}) \times H^{s - 1}(\mathbb{R}^{n})$ for $0 < s < s_{cr}$. 
\end{theorem}

%The proof of this theorem relies on the same ideas as the proof of Theorem 4 in Christ, Colliander, and Tao's paper \cite{CCT}.
For completeness, we present the main ideas of the proof here, and refer the reader to \cite{CCT} for more details.
{{The main idea is to utilize the various symmetries of the vNLW equation \eqref{VNLWE} 
to create a family of solutions depending on $\nu$, that ``translates'' $H^{s}$ growth of the solution to the visco-dispersive limit equation ($\nu = 0$)
into the unbounded growth of the $H^s$ norm of solutions to the vNLW equation \eqref{VNLWE} at progressively smaller times $t \to 0$.}}

In particular, we have two symmetries:
\begin{itemize}
\item \emph{The visco-dispersive scaling symmetry}: 
If $u(t, x)$ solves $\partial_{tt}u - \nu^{2} \Delta u + \nu \sqrt{-\Delta} \partial_{t} u + u^{p} = 0$, then $u(t, \nu x)$ solves the original vNLW equation \eqref{VNLWE}; and
\item \emph{The vNLW scaling symmetry}:  
If $u(t, x)$ solves the original vNLW equation \eqref{VNLWE}, then so does $\lambda^{\frac{2}{p - 1}}u(\lambda t, \lambda x)$.
\end{itemize}
Using these symmetries, one can construct a family of solutions of the vNLW equation as follows.
Fix an arbitrary compactly supported function $\phi_{0} \in C_{0}^{\infty}(\mathbb{R}^{n})$, 
and let $\phi^{(\nu)}(t, x)$, $\nu > 0$, denote the solution to the initial value problem:
\begin{equation*}
\partial_{tt}u - \nu^{2} \Delta u + \nu \sqrt{-\Delta} \partial_{t} u + u^{p} = 0,
\end{equation*}
\begin{equation*}
u(0, x) = \phi_{0}(x), \ \ \ \ \ \partial_{t}u(0, x) = 0,
\end{equation*}
with $(\phi_{0},0)$ as the initial data.
Using the symmetries above, and the solution $\phi^{(\nu)}(t, x)$, we can construct an entire family of solutions 
\begin{equation}\label{unulambda}
u^{(\nu, \lambda)}(t, x) = \lambda^{-\frac{2}{p - 1}} \phi^{(\nu)}(\lambda^{-1}t, \nu\lambda^{-1}x)
\end{equation}
to the original vNLW equation \eqref{VNLWE} 
corresponding to the following family of
initial  data for displacement:
\begin{equation*}
u^{(\nu, \lambda)}(0, x) = \lambda^{-\frac{2}{p - 1}} \phi_{0}(\nu\lambda^{-1}x).
\end{equation*}
One can show, following the bounds in \cite{CCT}, that for {{}{$0 < s < s_{cr} := \frac{n}{2} - \frac{2}{p - 1}$,}} the $H^{s}$ norm of the initial displacement 
$||u^{(\nu, \lambda)}(0, \cdot)||_{H^{s}(\mathbb{R}^{n})}$ is bounded by
\begin{equation*}
{{}{||u^{(\nu, \lambda)}(0, \cdot)||_{H^{s}(\mathbb{R}^{n})} \le C\lambda^{s_{cr} - s}\nu^{s - \frac{n}{2}},}}
\end{equation*}
whenever $0 < \lambda \le \nu$, for a constant $C$ independent of $0 < \lambda \le \nu$. 
So given $\epsilon > 0$, we can choose $\lambda$ and $\nu$ so that 
\begin{equation}\label{lambdanurelation}
{{}{\epsilon =C \lambda^{s_{cr} - s} \nu^{s - \frac{n}{2}}=C \frac{\lambda^{s_{cr} - s}}{\nu^{\frac{n}{2}-s}}, \quad \text{ where } s_{cr} - s > 0 \text{ and } \frac{n}{2}-s >0, \text{ since } 0 < s < s_{cr} < n/2.}}
\end{equation}
This gives a whole family of possibilities for $\lambda$ and $\nu$, where $0 < \lambda \le \nu$ for all $\nu > 0$ sufficiently small, {{since $\frac{n/2 - s}{s_{cr} - s} > 1$}}. In particular, for any choice of $\lambda, \nu$ with $\nu$ sufficiently small satisfying \eqref{lambdanurelation}, we will get that the norm of the initial displacement corresponding to the solution $u^{(\nu, \lambda)}$ in $H^{s}$ is indeed less than $\epsilon$.

To show that the $H^{s}$ norm of the solution $u^{(\nu, \lambda)}(t,x)$ is greater than $\epsilon^{-1}$ for {{some}} $0 < t < \epsilon$, 
recall that we had an explicit form for the solution of the visco-dispersive limit initial value problem:
\begin{equation*}
u_{tt} + u^{p} = 0,
\end{equation*}
\begin{equation*}
u(0, x) = \phi_{0}(x), \ \ \ \ \ \partial_{t}u(0, x) = 0,
\end{equation*}
which was given by $\phi^{(0)}(t, x) = \phi_{0}(x)V\left(t(\phi_{0}(x))^{\frac{p - 1}{2}})\right)$ 
for some smooth, even, periodic function $V$. 
Analyzing this $\phi^{(0)}$, one can see that $||\phi^{(0)}(t, \cdot)||_{H^{k}(\mathbb{R}^{n})} \sim t^{k}$ for all nonnegative integers $k \ge 0$ and sufficiently large $t$, 
and hence the result also holds for all $H^{s}$ for $s \ge 0$ not necessarily an integer, by interpolation. 
By Proposition~\ref{closeness}, we see that 
 for all $\nu > 0$ sufficiently small, $||\phi^{(\nu)}(t, \cdot)||_{H^{s}(\mathbb{R}^{n})}$ has the same behavior:
\begin{equation}\label{growth}
||\phi^{(\nu)}(t, \cdot)||_{H^{s}(\mathbb{R}^{n})} \sim t^{s},
\end{equation}
for $s \ge 0$ and for times $1 << t \le c|\text{log}\nu|^{c}$ (which can become increasingly large {{as $\nu \to 0$})}. 
Thus, the function $\phi^{(\nu)}$ transfers its energy to increasingly higher frequencies as time progresses.

We now want to show that this translates into increasingly high $H^s$ norm of $u^{(\nu, \lambda)}(t, x)$, for {{some time between $0$ and $\epsilon$}}.
Indeed, by recalling that $\phi^{(\nu)}$ appears in the definition of $u^{(\nu, \lambda)}(t, x)$ in the following way, see \eqref{unulambda}:
\begin{equation*}\label{unulambda}
u^{(\nu, \lambda)}(t, x) = \lambda^{-\frac{2}{p - 1}} \phi^{(\nu)}(\lambda^{-1}t, \nu\lambda^{-1}x),
\end{equation*}
and by using Fourier transform to examine the $H^s$ norm of $u^{(\nu, \lambda)}$ for all $\nu$ and $\lambda$ satisfying
\eqref{lambdanurelation}, one obtains, following the steps in the proof of Theorem 2, pg. 17, Sec.~4 of \cite{CCT}, 
the following estimate
$$
\| u^{(\nu, \lambda)}(\lambda t) \|_{H^s} \ge c' \epsilon t^s,
$$
which holds for a constant $c' > 0$ independent of $\lambda$ and $\nu$. 
By choosing $t$ large enough (depending on $\epsilon$) and $\lambda$ small enough (depending on $t$ and $\epsilon$),
one gets the desired estimate. 
Full details can be found in \cite{CCT}.

%%%%%%%%%%%%%%%
\if 1 = 0
We will now use the behavior of $\phi^{(\nu)}$ and the scalings of time  to conclude that for arbitrarily small time 
$$ 0 < t < \epsilon = C \frac{\lambda^{s_{c} - s}}{\nu^{\frac{n}{2}-s}}$$
the $H^s$ norm of $u^{(\nu, \lambda)}$ will become arbitrarily large as $\nu \to 0$.

Indeed, as we let $\nu \to 0$, we see that $\lambda \to 0$ so that $\epsilon$ is kept constant.
This implies that the time variable $t/\lambda$ in $\phi^{(\nu)}$ is increasing to infinity for each fixed $t \in (0,\epsilon)$,
since 
$$\frac{t}{\lambda} \in (0,\frac{\epsilon}{\lambda}) = (0, C \frac{\lambda^{s_c - s - 1}}{\nu^{\frac{n}{2} - s}})= (0,  \frac{\lambda^{- \frac{2}{p - 1}}}{\nu^{\frac{n}{2} - s}})\to (0,\infty),$$
where we used the expression for $s_{cr}$ given in \eqref{s_critical}.
This implies that the $H^s$ norm
$||\phi^{(\nu)}(t, \cdot)||_{H^{s}(\mathbb{R}^{n})}$ grows to infinity with the rate of $(t/\lambda)^s$, from \eqref{growth}.
Thus, for each given $\epsilon > 0$, the $H^s$ norm
$||\phi^{(\nu)}(t, \cdot)||_{H^{s}(\mathbb{R}^{n})}$ goes to infinity as $\nu\to 0$, for each small $t$,  $0<t < \epsilon$, which is what we wanted to show.
\fi
%%%%%%%%%%%%%%
%
%By noting that $\phi^{(\nu)}$ appears in the definition of $u^{(\nu, \lambda)}(t, x)$ in \eqref{unulambda}, we can use the immediately preceding estimate to make the $H^{s}$ norm of $\phi^{(\nu)}(t, x)$ at some future time $t$ (on the order of $t^{s}$) arbitrarily large by taking $t$ sufficiently \textit{large}. However, we want this $H^{s}$ amplification to happen at an arbitrarily \textit{small} time $0 < t < \epsilon$, but by using the symmetry $u(t, x) \to \lambda^{-\frac{2}{p - 1}}u(\lambda^{-1} t, \lambda^{-1} x)$, we can achieve this by taking $\lambda$ (and hence $\nu$) sufficiently small to scale down the time of $H^{s}$ norm amplification. 

While the proof in \cite{CCT} uses times that can be both negative and positive (since the wave equation is reversible in time), the argument from \cite{CCT} carries over to 
the vNLW equation case, where $t > 0$.

%TODO: Should we summarize the proof here?

In conclusion, we have shown by Theorem~\ref{ill_posedness},
that the viscous nonlinear wave equation is ill-posed for $0 < s < s_{cr}=n/2 - 2/(p-1)$.
In the second half of the manuscript we will show that this ill-posedness associated with the 
lack of continuity in the solution map, is in some sense a non-generic phenomenon,
and that using probabilistic arguments we can still get {{``probabilistic"}} well-posedness even below the critical exponent. 

Crucial for this argument will be Strichartz estimates, which we address next.

\section{Strichartz estimates for the linear viscous wave equation}\label{Strichartz}

In this section we show that the linear viscous wave operator has strong decay 
properties that admit a large collection of Strichartz estimates. 
Before beginning this analysis, following the abstract Strichartz estimates 
for the linear wave equation and the Schr\"{o}dinger equation of Keel and Tao \cite{KT}, 
we first consider the Fourier representation of the solution to the initial value problem for the linear viscous wave equation. We emphasize that Strichartz estimates are estimates for the \textit{linear} equation, 
which we will later use to prove probabilistic well-posedness results for the nonlinear problem studied in Sec.~\ref{random}. 

\subsection{Fourier representation of solution to the homogeneous and the inhomogeneous linear viscous wave equation}
\vskip 0.1in
\noindent
{\bf{The homogeneous linear viscous wave equation.}}
Consider the Cauchy problem for the
linear homogeneous viscous wave equation with initial conditions, $f, g \in \mathcal{S}(\mathbb{R}^{n})$:
\begin{equation*}
\partial_{tt}u - \Delta u + \sqrt{-\Delta}\partial_{t}u = 0,
\end{equation*}
\begin{equation*}
u(0, x) = f(x), \ \ \ \ \ \partial_{t}u(0, x) = g(x).
\end{equation*}
Taking spatial Fourier transforms (assuming that $u$ decays rapidly at infinity), we get
\begin{equation*}
\partial_{tt}\widehat{u}(t, \xi) + |\xi|\partial_{t}\widehat{u}(t, \xi) + |\xi|^{2}\widehat{u}(t, \xi) = 0,
\end{equation*}
\begin{equation*}
\widehat{u}(0, \xi) = \widehat{f}(\xi), \ \ \ \ \ \partial_{t}\widehat{u}(0, \xi) = \widehat{g}(\xi).
\end{equation*}
One can easily see that the solution is given by:
\begin{equation}\label{FTv}
\widehat{u}(t, \xi) = \widehat{f}(\xi)e^{-\frac{|\xi|}{2}t}\left(\text{cos}\left(\frac{\sqrt{3}}{2}|\xi|t\right) + \frac{1}{\sqrt{3}}\text{sin}\left(\frac{\sqrt{3}}{2}|\xi|t\right)\right) + \widehat{g}(\xi)e^{-\frac{|\xi|}{2}t}\frac{\text{sin}\left(\frac{\sqrt{3}}{2}|\xi|t\right)}{\frac{\sqrt{3}}{2}|\xi|}.
\end{equation}
This formula makes sense when $f, g \in \mathcal{S}(\mathbb{R}^{n})$, but it can also be extended to the case when $f, g \in L^{2}(\mathbb{R}^{n})$ by density arguments. 

An important thing to notice is that, in contrast with the wave equation, solution \eqref{FTv} has a 
damping term of the form $e^{-\frac{|\xi|}{2}t}$ associated with the viscous effects.

\vskip 0.2in
\noindent
{\bf{The inhomogeneous linear viscous wave equation.}}
Consider  the Cauchy problem for the inhomogeneous linear viscous wave equation:
\begin{align}
&(\partial_{tt} - \Delta + \sqrt{-\Delta}\partial_{t})u(t, x) = F(t, x),
\label{inhomo}\\
&\ \ u(0, x) = f(x) \ \ \ \ \ \partial_{t}u(0, x) = g(x)
\nonumber
\end{align}
To get the representation formula for the inhomogeneous viscous wave equation, we use Duhamel's principle. 
%We first verify that the usual Duhamel formulation is applicable. 
We follow the notation in Sogge \cite{Sg}: for $\tau > 0$, let $v(\tau; t, x)$ be the solution to the Cauchy problem:
\begin{equation*}
(\partial_{tt} - \Delta + \sqrt{-\Delta}\partial_{t})u = 0,
\end{equation*}
\begin{equation*}
u(\tau; 0, x) = 0, \ \ \ \ \ \partial_{t}u(\tau; 0, x) = F(\tau, x).
\end{equation*}
One can easily show that the solution to the inhomogeneous problem with source term $F(t, x)$ and zero initial data is then given by
\begin{equation*}
u(t, x) = \int_{0}^{t} u(\tau; t - \tau, x) d\tau.
\end{equation*}
%
%%%%%%%%%%%%%% DO NOT ERASE THE TEXT BELOW%%%%%%%%%%%
\if 1 = 0
We verify this in the usual way. As desired, $v(0, x) = 0$.
\begin{equation*}
\partial_{t}v(t, x) = v(t; 0, x) + \int_{0}^{t} \partial_{t}v(\tau; t - \tau, x) ds = \int_{0}^{t} \partial_{t}v(\tau; t - \tau, x) ds
\end{equation*}
So indeed $\partial_{t}v(0, x) = 0$. Taking another time derivative,
\begin{equation*}
\partial_{tt}v(t, x) = \partial_{t}v(t; 0, x) + \int_{0}^{t} \partial_{tt}v(\tau; t - \tau, x) d\tau = F(t, x) + \int_{0}^{t} \partial_{tt}v(\tau; t - \tau, x) d\tau
\end{equation*}
\begin{equation*}
\Delta v(t, x) = \int_{0}^{t} \Delta v(\tau; t - \tau, x) d\tau
\end{equation*}
\begin{equation*}
\sqrt{-\Delta} \partial_{t}v(t, x) = \int_{0}^{t} \sqrt{-\Delta} \partial_{t}v(\tau; t - \tau, x) d\tau
\end{equation*}
So, as desired,
\begin{equation*}
(\partial_{tt} - \Delta + \sqrt{-\Delta}\partial_{t})v = F(t, x) + \int_{0}^{t} (\partial_{tt} - \Delta + \sqrt{-\Delta}\partial_{t})v(\tau; t - \tau, x) d\tau = F(t, x)
\end{equation*}
\fi
%%%%%%%%%%%%%% END DO NOT ERASE %%%%%%%%%%%%%%%%%
%
Therefore, for the inhomogeneous Cauchy problem \eqref{inhomo}, we get the 
final formula for the Fourier representation of its solution:
\begin{eqnarray}\label{Fouriersoln1}
\widehat{u}(t, \xi) &= \widehat{f}(\xi)e^{-\frac{|\xi|}{2}t}\left(\text{cos}\left(\frac{\sqrt{3}}{2}|\xi|t\right) + \frac{1}{\sqrt{3}}\text{sin}\left(\frac{\sqrt{3}}{2}|\xi|t\right)\right) + \widehat{g}(\xi)e^{-\frac{|\xi|}{2}t}\frac{\text{sin}\left(\frac{\sqrt{3}}{2}|\xi|t\right)}{\frac{\sqrt{3}}{2}|\xi|} 
\nonumber
\\
&+ \displaystyle{\int_{0}^{t} \widehat{F}(\tau, \xi)e^{-\frac{|\xi|}{2}(t - \tau)}\frac{\text{sin}\left(\frac{\sqrt{3}}{2}|\xi|(t - \tau)\right)}{\frac{\sqrt{3}}{2}|\xi|} d\tau.}
\qquad\qquad\qquad
\end{eqnarray}
%Again, under suitable conditions on $f$, $g$, and $F$, this formula makes sense. 
We will write this formula alternatively as:
\begin{eqnarray}\label{Fouriersoln2}
u(t, \cdot) &= e^{-\frac{\sqrt{-\Delta}}{2}t}\left(\text{cos}\left(\frac{\sqrt{3}}{2}\sqrt{-\Delta}t\right) + \frac{1}{\sqrt{3}}\text{sin}\left(\frac{\sqrt{3}}{2}\sqrt{-\Delta}t\right)\right)f + e^{-\frac{\sqrt{-\Delta}}{2}t}\frac{\text{sin}\left(\frac{\sqrt{3}}{2}\sqrt{-\Delta}t\right)}{\frac{\sqrt{3}}{2}\sqrt{-\Delta}}g 
\nonumber
\\
&+ \displaystyle{\int_{0}^{t} e^{-\frac{\sqrt{-\Delta}}{2}(t - \tau)}\frac{\text{sin}\left(\frac{\sqrt{3}}{2}\sqrt{-\Delta}(t - \tau)\right)}{\frac{\sqrt{3}}{2}\sqrt{-\Delta}}F(\tau, \cdot) d\tau.}\qquad\qquad\qquad
\end{eqnarray}

\subsection{Statement of Strichartz estimates}

Next, we present Strichartz estimates, which will be useful when studying the nonlinear viscous wave equation using fixed point arguments. 
The estimates provide information about how the $L^q_tL^r_x$ norm of the solution to the linear problem is 
controlled in terms of data, for both the homogeneous case, and the inhomogeneous case.
While the homogeneous estimates follow using techniques similar to those used in the case of the linear wave equation
and the linear Schor\"{o}dinger equation, the inhomogeneous estimates will require different approaches because the
associated evolution operator is now self-adjoint due to the viscous contribution, as we explain below.
In both cases, the estimates will be given in terms of the $L^q_tL^r_x$ norms of the solution,
obtained via an evolution operator $U(t)$ associated with each problem separately. 
Crucial for the proof is the following important abstract result about Strichartz estimates due to Keel and Tao \cite{KT},
which uses the following definition of $\sigma$-admissible exponents $(q, r)$
 (see Sogge \cite{Sg} or Keel and Tao \cite{KT}):

\begin{definition}Let $\sigma > 0$.
The exponent pair $(q, r)$ is said to be $\sigma$-admissible if $q, r \ge 2$, $(q, r, \sigma) \ne (2, \infty, 1)$ and 
\begin{equation*}
\frac{2}{q} + \frac{2\sigma}{r} \le \sigma.
\end{equation*}
\end{definition}

\begin{theorem}[General Strichartz estimates, Keel and Tao \cite{KT}]\label{KTStr}
Let $U(t), t \in \mathbb{R}$ be a one-parameter family of operators
$$
U(t): L^2(\R^n) \to L^2(\R^n),
$$ 
such that the following two estimates hold:
\begin{enumerate}
\item The energy estimate holding uniformly in $t$:
\begin{equation}\label{L2}
||U(t)f||_{L^{2}_{x}} \le C||f||_{L_x^{2}}, \ \ \ \ \ f \in L^{2}(\mathbb{R}^{n}), 
\end{equation}
\item The truncated dispersive decay estimate holding  for some $\sigma > 0$, uniformly in $\tau$ and $t$:
\begin{equation}\label{dispersive}
||U(\tau)U^{*}(t)f||_{L^{\infty}_{x}} \le C(1 + |t - \tau|)^{-\sigma} ||f||_{L_x^{1}},
\end{equation}
where $U^*(t)$ is the adjoint operator. 
\end{enumerate}
Then, for all $\sigma$-admissible pairs $(q, r)$ and $(\tilde{q}, \tilde{r})$, the following estimates hold:
\begin{equation}\label{KeelTao1}
||U(t)f||_{L^{q}_{t}L^{r}_{x}} \le C||f||_{L_x^{2}},
\end{equation}
\begin{equation}\label{KeelTao2}
\left|\left|\int_{-\infty}^{\infty} U^{*}(\tilde\tau)F(\tilde\tau, \cdot) d\tilde\tau\right|\right|_{L^{2}_{x}} \le C||F||_{L^{q'}_{t}L^{r'}_{x}},
\end{equation}
\begin{equation}\label{KeelTao3}
\left|\left|\int_{-\infty}^{t} U(t)U^{*}(\tilde\tau)F(\tilde\tau, \cdot)d\tilde\tau\right|\right|_{L^{q}_{t}L^{r}_{x}} \le 
C||F||_{L^{\tilde{q}'}_{t}L^{\tilde{r}'}_{x}},
\end{equation}
where $\tilde{q}'$ and $\tilde{r}'$ are H\"{o}lder conjugates of $\tilde{q}$ and $\tilde{r}$, respectively.
\end{theorem}
We will use estimate \eqref{KeelTao1} of this theorem to get appropriate estimates 
on the solution of the homogeneous linear viscous wave equation, 
defined via certain evolutions operators $U(t)$ that we define below, applied to the given data. Estimates \eqref{KeelTao2} and \eqref{KeelTao3} are 
usually used to estimate the Duhamel contribution of the inhomogeneous term for the 
linear wave equation and the Schr\"{o}dinger equation. 
In our case, however, since the contribution from the viscous regularization is self-adjoint, 
we will have to resort to different approaches to estimate the Duhamel contribution
in the inhomogeneous case, as we explain below.

Our main results on Strichartz estimates are the following.

\begin{theorem}[{\bf{Strichartz estimates for homogeneous linear viscous wave equation}}]\label{strongdecay}
Let $u$ be a solution to
the Cauchy problem
\begin{equation}\label{Cauchy1}
(\partial_{tt} - \Delta + \sqrt{-\Delta}\partial_{t})u = 0,
\end{equation}
\begin{equation*}
u(0, \cdot) = f, \ \ \ \ \ \partial_{t}u(0, \cdot) = g.
\end{equation*}
Then, for any $\sigma > 0$ and any $\sigma$-admissible pair $(q, r)$, $r < \infty$, 
there exists a constant $C_\sigma >0$ depending only on $\sigma$, such that for every time $0 < T < \infty$, 
the following estimate holds:
\begin{equation}\label{Strichartz_homo}
||u||_{L_{t}^{q}([0, T]; L_{x}^{r}(\mathbb{R}^{n}))} + ||u(T, \cdot)||_{\dot{H}^{s}(\mathbb{R}^{n})} + ||\partial_{t}u(T, \cdot)||_{\dot{H}^{s - 1}(\mathbb{R}^{n})} \le C_{\sigma}||f||_{\dot{H}^{s}(\mathbb{R}^{n})} + C_{\sigma}||g||_{\dot{H}^{s - 1}(\mathbb{R}^{n})},
\end{equation}
provided that the gap condition 
\begin{equation}\label{gap1}
\frac{1}{q} + \frac{n}{r} = \frac{n}{2} - s
\end{equation}
holds.
\end{theorem}

\begin{theorem}[{\bf{Strichartz estimates for inhomogeneous linear viscous wave equation}}]\label{inhStr}
Let $n \ge 2$ and let $(q, r)$ and $(\tilde{q}, \tilde{r})$ be any two pairs with $\tilde{q}, \tilde{r} \ge 2$, $1 < \tilde{q}' < q < \infty$, $1 < \tilde{r}' < r \le \infty$, where $\tilde{q}'$ and $\tilde{r}'$ are H\"{o}lder conjugates of $\tilde{q}$ and $\tilde{r}$, respectively.
Let $u$ be a solution to the Cauchy problem
\begin{equation*}
(\partial_{tt} - \Delta + \sqrt{-\Delta}\partial_{t})u = F,
\end{equation*}
\begin{equation*}
u(0, \cdot) = 0, \ \ \ \ \ \partial_{t}u(0, \cdot) = 0.
\end{equation*}
Then, there exists a constant $C_{q,\tilde{q},r,\tilde{r}} > 0$ depending on $q, \tilde{q}, r, \tilde{r}$, such that for every time $0 < T < \infty$, the following estimate holds:
\begin{equation}\label{S_estimate_inhomo}
||u||_{L_{t}^{q}([0, T]; L_{x}^{r}(\mathbb{R}^{n}))} + ||u(T, \cdot)||_{\dot{H}^{s}(\mathbb{R}^{n})} + ||\partial_{t}u(T, \cdot)||_{\dot{H}^{s - 1}(\mathbb{R}^{n})} \\
\le C_{q, \tilde{q}, r, \tilde{r}}||F||_{L_{t}^{\tilde{q}'}([0, T]; L_{x}^{\tilde{r}'}(\mathbb{R}^{n}))},
\end{equation}
provided that the gap condition
\begin{equation}\label{gap2}
\frac{1}{q} + \frac{n}{r} = \frac{n}{2} - s = \frac{1}{\tilde{q}'} + \frac{n}{\tilde{r}'} - 2
\end{equation}
holds. 
\end{theorem}

\begin{remark}[{\bf{Dimension $n = 1$}}]
Note that the homogeneous Strichartz estimates in Theorem \ref{strongdecay} hold for any dimension $n$, including $n = 1$. This is interesting because it is well-known that the linear wave equation in one dimension does not possess such Strichartz estimates. However, the viscous wave equation in one dimension has homogeneous Strichartz estimates due to the dissipating effects of the viscosity. 
\end{remark}

\begin{remark}[{\bf{The gap condition}}]
The gap condition is a natural condition to impose in both cases, as it is the exact condition needed for the inequalities above to respect the scaling symmetry of solutions in time and space. This makes the inequality scale invariant, and this property will later play an important role, especially in the proof of Theorems \ref{strongdecay} and \ref{inhStr}.
\end{remark}

\begin{remark}[{\bf{Admissible exponents $(q, r)$}}]
The Strichartz estimates for the viscous wave equation are better than those of the wave equation in the sense that 
both the homogeneous and inhomogeneous estimates hold for
a larger class of admissible exponents. 
Again, this is due to the dissipative effects of viscosity. 
\end{remark}

Classical Strichartz estimates for the linear Schr\"{o}dinger equation and the linear wave equation usually follow immediately from the abstract Strichartz estimates by Keel and Tao \cite{KT} in Theorem \ref{KTStr} above. In the case of the viscous wave equation, this will indeed be true for the homogeneous estimates, but it will not be true for the inhomogeneous estimates, because the dissipative portion of the evolution operator is self-adjoint. This makes the proof of the inhomogeneous Strichartz estimates considerably more subtle, and the proof will employ techniques from harmonic analysis that are markedly different from the techniques used to prove the homogeneous estimates. 

In the proof of Theorems \ref{strongdecay} and \ref{inhStr}, we will use the following well-known Littlewood-Paley theorem, 
which will help us reduce the problem to proving the estimates for the components in the Littlewood-Paley decomposition.

\begin{lemma}[Littlewood-Paley lemma]\label{LPlemma}
Let $\beta \in C_{0}^{\infty}(\mathbb{R}_{+})$, $0 \le \beta \le 1$, with support in $[1/2, 2]$ give the Littlewood-Paley decomposition
\begin{equation*}
\sum_{j = -\infty}^{\infty} \beta\left(\frac{\xi}{2^{j}}\right) = 1 \text{ for all } \xi > 0.
\end{equation*}
Define the Littlewood-Paley operators:
\begin{equation*}
G_{j}(t, x) = \frac{1}{(2\pi)^{n}} \int_{\mathbb{R}^{n}} e^{ix \cdot \xi} \beta\left(\frac{|\xi|}{2^{j}}\right)\widehat{G}(t, \xi) d\xi,
\end{equation*}
which send $G$ to its Littlewood-Paley decomposition $\{G_{j}\}_{j = -\infty}^{\infty}$. 
Then, the following estimates hold:
\begin{itemize}
\item If $2 \le r < \infty$ and $2 \le q \le \infty$,
\begin{equation}\label{LP1}
||G||^{2}_{L^{q}_{t}L^{r}_{x}} \le C \sum_{j = -\infty}^{\infty} ||G_{j}||^{2}_{L_{t}^{q}L_{x}^{r}}.
\end{equation}
\item If $1 < r \le 2$ and $1 \le q \le 2$, 
\begin{equation}\label{LP2}
\sum_{j = -\infty}^{\infty} ||G_{j}||^{2}_{L^{q}_{t}L^{r}_{x}} \le C||G||^{2}_{L^{q}_{t}L^{r}_{x}}.
\end{equation}
\end{itemize}
\end{lemma}

To prove Theorems \ref{strongdecay} and \ref{inhStr} we introduce the 
Littlewood-Paley operators $U^{(j)}(t)$ and $V^{(j)}(t)$ 
that account for the contribution of the initial data $f$ and $g$, separately, to the solution 
of \eqref{Cauchy1}. We recall \eqref{FTv}:
\begin{equation*}
\widehat{u}(t, \xi) = e^{-\frac{|\xi|}{2}t}
\left(\text{cos}\left(\frac{\sqrt{3}}{2}|\xi|t\right) + \frac{1}{\sqrt{3}}\text{sin}\left(\frac{\sqrt{3}}{2}|\xi|t\right)\right) \widehat{f}(\xi)
+ e^{-\frac{|\xi|}{2}t}\frac{\text{sin}\left(\frac{\sqrt{3}}{2}|\xi|t\right)}{\frac{\sqrt{3}}{2}|\xi|}\widehat{g}(\xi),
\end{equation*}
and introduce:
%
%
%that we will need and the necessary result that we need to apply Keel and Tao's abstract Strichartz estimates \cite{KT} regarding these %operators. In particular, we will consider the operators
\begin{equation*}
U^{(j)}(t)f(x) = \chi_{[0, T]}(t)\int_{\mathbb{R}^{n}} e^{ix \cdot \xi} e^{-\frac{|\xi|}{2}t}
\left(\text{cos}\left(\frac{\sqrt{3}}{2}|\xi|t\right) + \frac{1}{\sqrt{3}}\text{sin}\left(\frac{\sqrt{3}}{2}|\xi|t\right)\right) 
 \widehat{f}(\xi) \beta\left(\frac{|\xi|}{2^{j}}\right) d\xi,
\end{equation*}
and
\begin{equation*}
V^{(j)}(t)g(x) = \chi_{[0, T]}(t)\int_{\mathbb{R}^{n}} e^{ix \cdot \xi} e^{-\frac{|\xi|}{2}t}\frac{\text{sin}\left(\frac{\sqrt{3}}{2}|\xi|t\right)}{\frac{\sqrt{3}}{2}|\xi|}\widehat{g}(\xi) \beta\left(\frac{|\xi|}{2^{j}}\right) d\xi.
\end{equation*}
As we shall see later, only the operators $U^{(j)}(t)$ and $V^{(j)}(t)$ with $-2\le j \le 2$ will be relevant for
the proof.  Operators $U^{(j)}(t)$ and $V^{(j)}(t)$, $-2 \le j \le 2$,
satisfy the following estimates, which follow from Keel and Tao  \cite{KT}:

\begin{lemma}[Estimates on $U^{(j)}(t)f(x)$ and $V^{(j)}(t)g(x)$]\label{operatorest}
Given $\sigma > 0$, there exists a constant $C_{\sigma}$ independent of $T > 0$ such that for all $\sigma$-admissible pairs $(q, r)$, $(\tilde{q}, \tilde{r})$ and for all integers $-2 \le j \le 2$,
\begin{equation*}
||U^{(j)}_{\pm}(t)f||_{L^{q}_{t}L^{r}_{x}} \le C_{\sigma}||f||_{L^{2}},
\end{equation*}
\begin{equation*}
\left|\left|\int_{-\infty}^{\infty} (U^{(j)}_{\pm})^{*}(\tau)F(\tau, \cdot) d\tau\right|\right|_{L^{2}_{x}} \le C_{\sigma}||F||_{L^{q'}_{t}L^{r'}_{x}},
\end{equation*}
\begin{equation*}
\left|\left|\int_{-\infty}^{t} U^{(j)}_{\pm}(t)(U^{(j)}_{\pm})^{*}(\tau)F(\tau, \cdot)d\tau\right|\right|_{L^{q}_{t}L^{r}_{x}} \le C_{\sigma}||F||_{L^{\tilde{q}'}_{t}L^{\tilde{r}'}_{x}},
\end{equation*}
\begin{equation*}
||V^{(j)}(t)g||_{L^{q}_{t}L^{r}_{x}} \le C_{\sigma}||g||_{L^{2}},
\end{equation*}
\begin{equation*}
\left|\left|\int_{-\infty}^{\infty} (V^{(j)})^{*}(\tau)F(\tau, \cdot) d\tau\right|\right|_{L^{2}_{x}} \le C_{\sigma}||F||_{L^{q'}_{t}L^{r'}_{x}},
\end{equation*}
\begin{equation*}
\left|\left|\int_{-\infty}^{t} V^{(j)}(t)(V^{(j)})^{*}(\tau)F(\tau, \cdot)d\tau \right|\right|_{L^{q}_{t}L^{r}_{x}} \le C_{\sigma}||F||_{L^{\tilde{q}'}_{t}L^{\tilde{r}'}_{x}}.
\end{equation*}
\end{lemma}

\begin{proof}
We verify the necessary conditions in Theorem \ref{KTStr} to get the desired estimates. 
We start with inequality \eqref{L2}. It is clear from Plancherel's theorem that
\begin{equation*}
||U^{(j)}(t)f||_{L^{2}_{x}} \le C||f||_{L^{2}},
\end{equation*}
where $C$ is independent of $t$ and $j$, since $\left|e^{-\frac{|\xi|}{2}t}
\left(\text{cos}\left(\frac{\sqrt{3}}{2}|\xi|t\right) + \frac{1}{\sqrt{3}}\text{sin}\left(\frac{\sqrt{3}}{2}|\xi|t\right)\right) 
 \beta\left(\frac{|\xi|}{2^{j}}\right)\right| \le 2$. 

Similarly, by Plancherel's formula, we can deduce that
\begin{equation*}
||V^{(j)}(t)g||_{L^{2}_{x}} \le C||g||_{L^{2}}
\end{equation*}
for some $C$ that is independent of $t$ and $j$. To see this, we simply note that there exists a constant $C$ such that
\begin{equation}\label{Vest}
\left|e^{-\frac{|\xi|}{2}t}\frac{\text{sin}\left(\frac{\sqrt{3}}{2}|\xi|t\right)}{\frac{\sqrt{3}}{2}|\xi|}\beta\left(\frac{|\xi|}{2^{j}}\right)\right| \le C
\end{equation}
for all $t > 0$ and for all $\xi \in \mathbb{R}^{n}$. {{This is due to the support properties of $\beta\left(\frac{|\xi|}{2^{j}}\right)$ and the fact that $-2 \le j \le 2$, so that the quantity on the left hand side of \eqref{Vest} is potentially nonzero only for $1/8 \le |\xi| \le 8$}}.

To verify \eqref{dispersive}, we fix an arbitrary $\sigma > 0$ and verify the estimate
\begin{equation*}
||U_{\pm}^{(j)}(t)(U^{(j)}_{\pm})^{*}(\tau)f||_{L^{\infty}_{x}} \le C_{\sigma}(1 + |t - \tau|)^{-\sigma}||f||_{L^{1}_{x}}
\end{equation*}
for some constant $C_{\sigma}$. It suffices to prove this inequality for positive {\sl{integers}} $\sigma$. 
{{We calculate
\begin{equation*}
U^{(j)}(t)(U^{(j)})^{*}(\tau)f(x) = \chi_{[0, T]}(t)\chi_{[0, T]}(\tau) \int_{\mathbb{R}^{n}} e^{-\frac{|\xi|}{2}(t + \tau)}e^{ix \cdot \xi} a(\xi, t) a(\xi, \tau) \widehat{f}(\xi) \beta^{2}\left(\frac{|\xi|}{2^{j}}\right) d\xi,
\end{equation*}
where
\begin{equation*}
a(\xi, t) := \cos\left(\frac{\sqrt{3}}{2}|\xi|t\right) + \frac{1}{\sqrt{3}}\sin\left(\frac{\sqrt{3}}{2}|\xi|t\right).
\end{equation*}

%
%%%%%%%%%%%%%%%%%%%
\if 1 = 0

NOTE: I don't think we need to separate out $|t - s|$ small and $|t - s| \ge 1$
Note that since $0 \le \beta(\xi) \le 1$, 
\begin{equation*}
|U_{\pm}^{(j)}(s)(U^{(j)}_{\pm})^{*}(t)f(x)| \le \int_{\mathbb{R}^{n}} |\widehat{f}(\xi)| d\xi = ||f||_{L^{1}}
\end{equation*}
which controls the $L^{\infty}$ norm when $|t - s|$ is small. When $|t - s| \ge 1$, we want to establish a bound of the form
\begin{equation*}
||U_{\pm}^{(j)}(s)(U^{(j)}_{\pm})^{*}(t)f||_{L_{x}^{\infty}} \le C_{\sigma}(1 + |t - s|)^{-\sigma}||f||_{L^{1}}
\end{equation*}
where
\begin{equation*}
U_{\pm}^{(j)}(s)(U^{(j)}_{\pm})^{*}(t)f(x) = \chi_{[0, T]}(t)\chi_{[0, T]}(s) \int_{\mathbb{R}^{n}} e^{-\frac{|\xi|}{2}(s + t)}e^{ix \cdot \xi \pm i \frac{\sqrt{3}}{2} (s - t)|\xi|} \widehat{f}(\xi) \beta^{2}\left(\frac{|\xi|}{2^{j}}\right) d\xi
\end{equation*}
\fi 
%%%%%%%%%%%%%%%%%%%%
%
Assume $t$ and $\tau$ are such that $0 \le \tau, t \le T$ where $T > 0$ is arbitrary. Fix an integer $k > n/2$. Then,
\begin{equation*}
||U^{(j)}(t)(U^{(j)})^{*}(\tau)f||^{2}_{H^{k}_{x}} = \frac{1}{(2\pi)^{n}} \int_{\mathbb{R}^{n}} (1 + |\xi|^{2})^{k} \left|e^{-\frac{|\xi|}{2}(t + \tau)}a(\xi, t)a(\xi, \tau)\widehat{f}(\xi)\beta^{2}\left(\frac{|\xi|}{2^{j}}\right)\right|^{2} d\xi.
\end{equation*}
Note that for $-2 \le j \le 2$, $\beta(|\xi|/2^{j})$ is supported in $1/8 \le |\xi| \le 8$ and $0 \le \beta(\xi) \le 1$. Therefore, since $|a(\xi, t)| \le 2$, 
\begin{multline*}
||U^{(j)}(t)(U^{(j)})^{*}(\tau)f||^{2}_{H^{k}_{x}} \le \frac{1}{(2\pi)^{n}}\int_{1/8 \le |\xi| \le 8} 4 \cdot 65^{k}e^{-\frac{1}{8}(t + \tau)} |\widehat{f}(\xi)|^{2} d\xi \\
= C_{k} e^{-\frac{1}{8}(t + \tau)} \int_{1/8 \le |\xi| \le 8} |\widehat{f}(\xi)|^{2}d\xi \le C_{k}'e^{-\frac{1}{8}(t + \tau)}||\widehat{f}||_{L^{\infty}_{\xi}}^{2} \le C_{k}'e^{-\frac{1}{8}(t + \tau)}||f||_{L^{1}_{x}}^{2},
\end{multline*}
implying
\begin{equation*}
||U^{(j)}(t)(U^{(j)})^{*}(\tau)f||_{H^{k}_{x}} \le C_{k}e^{-\frac{1}{16}(t + \tau)}||f||_{L_{x}^{1}} \le C_{k}e^{-\frac{1}{16}|t - \tau|}||f||_{L_x^{1}},
\end{equation*}
for $0 \le \tau, t \le T$, since  $|t - \tau| \le t + \tau$. Because $e^{-\frac{1}{16}|t - \tau|}$ decays exponentially, it decays faster than $(1 + |t - \tau|)^{-\sigma}$ for any positive integer $\sigma$. In particular,
\begin{equation*}
||U^{(j)}_{\pm}(t)(U^{(j)}_{\pm})^{*}(\tau)f||_{H^{k}_{x}} \le C_{\sigma}(1 + |t - \tau|)^{-\sigma}||f||_{L^{1}_{x}},
\end{equation*}
for $0 \le t, \tau \le T$ and for all $-2 \le j \le 2$. For all other $t, \tau$, the left hand side is zero by the characteristic functions $\chi_{[0, T]}$. 
Since $k > n/2$, $H^{k}(\R^n)$ embeds into $L^{\infty}(\R^n)$, and so we have:
\begin{equation*}
||U^{(j)}_{\pm}(t)(U^{(j)}_{\pm})^{*}(\tau)f||_{L^{\infty}_{x}} \le C||U^{(j)}_{\pm}(t)(U^{(j)}_{\pm})^{*}(\tau)f||_{H^{k}_{x}} \le C_{\sigma}(1 + |t - \tau|)^{-\sigma}||f||_{L^{1}_{x}},
\end{equation*}
which shows that assumption \eqref{dispersive} holds for the operator $U^{(j)}$, $-2 \le j \le 2$.}}
%%%%%%%%%%
%(The first inequality is from the fact that $H^{s}$ embeds into $L^{\infty}$ for $s > n/2$, by using the inverse Fourier transform and the Cauchy-Schwartz inequality, and the fact that $\int_{\mathbb{R}^{n}} \frac{1}{(1 + |\xi|^{2})^{s}} d\xi < \infty$ for $s > n/2$.)
%%%%%%%%%

To show that the operator $V^{(j)}$ with $-2 \le j \le 2$ satisfies assumption \eqref{dispersive}, we proceed in a similar way. We calculate
$$
V^{(j)}(t)(V^{(j)})^{*}(\tau)g(x) 
= \chi_{[0, T]}(t)\chi_{[0, T]}(\tau) \int_{\mathbb{R}^{n}} e^{-\frac{|\xi|}{2}(t + \tau)}e^{ix \cdot \xi}\frac{\text{sin}\left(\frac{\sqrt{3}}{2}|\xi|t\right)}{\frac{\sqrt{3}}{2}|\xi|}\cdot \frac{\text{sin}\left(\frac{\sqrt{3}}{2}|\xi|\tau\right)}{\frac{\sqrt{3}}{2}|\xi|} \widehat{g}(\xi) \beta^{2}\left(\frac{|\xi|}{2^{j}}\right) d\xi.
$$
Again, we only need to consider $0 \le \tau, t \le T$ where $T > 0$ is arbitrary. Fix an integer $k > n/2$, and note that
\begin{align*}
||V^{(j)}(t)(V^{(j)})^{*}(\tau)g||^{2}_{H^{k}_{x}} 
&\le \frac{1}{(2\pi)^{n}} \int_{\mathbb{R}^{n}} (1 + |\xi|^{2})^{k} \left|e^{-\frac{|\xi|}{2}(t + \tau)} \frac{\text{sin}\left(\frac{\sqrt{3}}{2}|\xi|t\right)}{\frac{\sqrt{3}}{2}|\xi|}\cdot \frac{\text{sin}\left(\frac{\sqrt{3}}{2}|\xi|\tau\right)}{\frac{\sqrt{3}}{2}|\xi|} \widehat{g}(\xi) \beta^{2}\left(\frac{|\xi|}{2^{j}}\right)\right|^{2} d\xi 
\\
&\le C \int_{1/8 \le |\xi| \le 8} (1 + |\xi|^{2})^{k} \left|e^{-\frac{|\xi|}{2}(t + \tau)} \widehat{g}(\xi)\beta^{2}\left(\frac{|\xi|}{2^{j}}\right)\right|^{2} d\xi,
\end{align*}
where we used
$$
\left|\frac{\text{sin}\left(\frac{\sqrt{3}}{2}|\xi|\tau\right)}{\frac{\sqrt{3}}{2}|\xi|}\beta\left(\frac{|\xi|}{2^{j}}\right)\right| \le C
$$
{ uniformly on the support of $\beta\left(\frac{|\xi|}{2^{j}}\right)$ for $-2 \le j \le 2$}.
%and similarly for $t$ in place of $\tau$ for all $\tau, t > 0$ and for all $-2 \le j \le 2$ by the support properties of $\beta\left(\frac{|\xi|}{2^{j}}\right)$. 
Since $-2 \le j \le 2$, we get
\begin{align*}
||V^{(j)}(t)(V^{(j)})^{*}(\tau)g||^{2}_{H^{k}_{x}}& \le C\int_{1/8 \le |\xi| \le 8}(1 + |\xi|^{2})^{k}e^{-|\xi|(t + \tau)}|\widehat{g}(\xi)|^{2} d\xi \\
&\le C\int_{1/8 \le |\xi| \le 8} 65^{k}e^{-\frac{1}{8}(t + \tau)} |\widehat{g}(\xi)|^{2} d\xi \le C_{k}e^{-\frac{1}{8}(t + \tau)}||\widehat{g}||_{L^{\infty}_{\xi}}^{2} \le C_{k}e^{-\frac{1}{8}(t + \tau)}||g||_{L^{1}_{x}}^{2}.
\end{align*}
Therefore, recalling that $k > n/2$, for all $-2 \le j \le 2$ we have
\begin{align*}
||V^{(j)}(t)(V^{(j)})^{*}(\tau)g||_{L^{\infty}_{x}} 
&\le C||V^{(j)}(t)(V^{(j)})^{*}(\tau)g||_{H^{k}_{x}} \le C_{k}e^{-\frac{1}{16}(t + \tau)} ||g||_{L^{1}_{x}} 
\\
&\le C_{k}e^{-\frac{1}{16}|t - \tau|} ||g||_{L^{1}_{x}} \le C_{\sigma}(1 + |t - \tau|)^{-\sigma}||g||_{L^{1}_{x}},
\end{align*}
where $|t - \tau| \le t + \tau$ since we are considering $0 \le \tau, t \le T$.
The estimates from the statement of Lemma~\ref{operatorest} now follow directly from Keel and Tao \cite{KT}.
\end{proof}

\subsubsection{Proof of Theorem \ref{strongdecay}: Strichartz estimates for the homogeneous problem}
To prove the Strichartz estimates for the homogeneous linear viscous wave equation stated in Theorem \ref{strongdecay}
we define $u_{j}(t, x)$, $f_{j}(x)$, $g_{j}(x)$ for $j \in \mathbb{Z}$ by
\begin{equation}\label{uj}
u_{j}(t, x) = \frac{1}{(2\pi)^{n}}\int_{\mathbb{R}^{n}} e^{ix \cdot \xi} \beta\left(\frac{|\xi|}{2^{j}}\right)\widehat{u}(t, \xi) d\xi,
\end{equation}
\begin{equation}\label{fj}
f_{j}(x) = \frac{1}{(2\pi)^{n}} \int_{\mathbb{R}^{n}} e^{ix\cdot\xi} \beta\left(\frac{|\xi|}{2^{j}}\right)\widehat{f}(\xi) d\xi,
 \ \ \ \ \ \ \ \ g_{j}(x) = \frac{1}{(2\pi)^{n}} \int_{\mathbb{R}^{n}} e^{ix\cdot\xi} \beta\left(\frac{|\xi|}{2^{j}}\right)\widehat{g}(\xi) d\xi.
\end{equation}
It is easy to see that $u_j$ solves the corresponding linear viscous wave equation with initial data $f_{j}$, $g_{j}$.
%By taking Fourier transforms and multiplying through by $\beta(|\xi|/2^{j})$ (and noting that this can be pulled out of any partial $t$ derivatives), we see that $u_{j}$ solves the corresponding linear viscous wave equation with initial data $f_{j}$, $g_{j}$. 
%
Notice that restricted on the time interval $[0,T]$, $u_j$  can be written in terms of $U^{(j)}(t)$ and $V^{(j)}(t)$
as follows:
\begin{equation}\label{ujform}
u_j (t,x) = \frac{1}{(2\pi)^n} \left(U^{(j)}(t) f(x) + V^{(j)}(t) g(x)\right) = \frac{1}{(2\pi)^{n}} \sum_{k = j-2}^{j + 2} U^{(k)}(t)f_{j}(x) + V^{(k)}(t) g_{j}(x)
,  \ t\in [0,T],
\end{equation}
{{where the second equality follows from the fact that $\text{supp}(\beta) \subset [1/2, 2]$}}. Obtaining estimates for $u_j$ will be based on using the results from Lemma~\ref{operatorest},
and the Littlewood-Paley decomposition estimates from Lemma~\ref{LPlemma}.
{{More precisely, we first show in Step 1 below that given a solution $u$ for any initial data $(f, g)$, it suffices to obtain the corresponding estimate for each $u_j$, uniform in $j \in \mathbb{Z}$, for the initial data 
$f_j, g_{j}$
defined via \eqref{fj}. In Step 2, we then simplify this even further by showing that, in fact, it suffices to simply consider initial data whose spatial Fourier transforms are supported in the annulus $1/2 \le |\xi| \le 2$.
Thus, what we show in Steps 1 and 2 below is that
it suffices to obtain the estimate \eqref{Strichartz_homo} from Theorem~\ref{strongdecay} 
for initial data $f, g$ that have spatial Fourier transforms 
that are all supported in the annulus $1/2 \le |\xi| \le 2$. 
The proof will then follow from estimates presented in Lemma~\ref{operatorest}.}}

\vskip 0.1in
\noindent
\textbf{Step 1.} Since $\beta$ is supported in $[1/2,2]$, we claim that it suffices to show that $u_j$ satisfies the estimate \eqref{Strichartz_homo} from Theorem \ref{strongdecay} for the data $f_j, g_j$ defined by \eqref{fj}, where the constant $C$ in the estimate \eqref{Strichartz_homo} is {\emph{independent}} of $j$. 
Since $q, r \ge 2$, we can use \eqref{LP1} to get the results for general $f, g$. 
%
%
%%%%%%%%%%%%%%%%%%
\if 1 = 0
, with the constant $C$  {\emph{independent}} of $j$.
Since $\beta$ is supported in $[1/2,2]$, it suffices to show that estimate \eqref{Strichartz_homo}
holds
for 
 $f_i$ and $g_i$ 
defined via $\widehat{f}, \widehat{g}$ that are supported in an annulus of the form
\begin{equation*}
2^{j - 1} \le |\xi| \le 2^{j + 1},
\end{equation*}
where the constant $C$ in the estimate is {\emph{independent}} of $j$. 
This is because $p, q \ge 2$, so we can use \eqref{LP1} to get the results for general $f, g$. 
Therefore, we are now aiming to show that
\begin{multline*}
||u_{j}||_{L_{t}^{p}([0, T]; L_{x}^{q}(\mathbb{R}^{n}))} + ||u_{j}(T, \cdot)||_{\dot{H}^{\gamma}(\mathbb{R}^{n})} + ||\partial_{t}u_{j}(T, \cdot)||_{\dot{H}^{\gamma - 1}(\mathbb{R}^{n})} \\
\le C||f_{j}||_{\dot{H}^{\gamma}(\mathbb{R}^{n})} + C||g_{j}||_{\dot{H}^{\gamma - 1}(\mathbb{R}^{n})},
\end{multline*}
where the constant $C$ is independent of $j \in \mathbb{Z}$, for the initial data such that
$\widehat{f}, \widehat{g}$ are supported in $2^{j - 1} \le |\xi| \le 2^{j + 1}$. 

\fi

%%%%%%%%%%%%%%
In particular, given a general solution $u$ to the linear viscous wave equation with initial data $f, g$, we can construct $u_{j}$, $f_{j}$, $g_{j}$ 
as in \eqref{uj} and \eqref{fj}. Suppose that for all such $u, f, g$, the functions $u_j$ for $j \in \mathbb{Z}$ satisfy the estimate
\begin{align}
||u_{j}||_{L_{t}^{q}([0, T]; L_{x}^{r}(\mathbb{R}^{n}))} &+ ||u_{j}(T, \cdot)||_{\dot{H}^{s}(\mathbb{R}^{n})} + ||\partial_{t}u_{j}(T, \cdot)||_{\dot{H}^{s - 1}(\mathbb{R}^{n})}
\nonumber \\
&\le C||f_{j}||_{\dot{H}^{s}(\mathbb{R}^{n})} + C||g_{j}||_{\dot{H}^{s - 1}(\mathbb{R}^{n})},
\label{j_estimate}
\end{align}
where the constant $C$ is independent of $j \in \mathbb{Z}$.

Estimate \eqref{j_estimate} implies that for a constant $C'$ independent of $j \in \mathbb{Z}$,
\begin{align}\label{reductionest1}
\text{max} &\left\{||u_{j}||_{L_{t}^{q}([0, T]; L_{x}^{r}(\mathbb{R}^{n}))}^{2}, ||u_{j}(T, \cdot)||_{\dot{H}^{s}(\mathbb{R}^{n})}^{2}, ||\partial_{t}u_{j}(T, \cdot)||_{\dot{H}^{s - 1}(\mathbb{R}^{n})}^{2} \right\} \\
&\qquad\qquad \le  C'\left(||f_{j}||_{\dot{H}^{s}(\mathbb{R}^{n})}^{2} + ||g_{j}||_{\dot{H}^{s - 1}(\mathbb{R}^{n})}^{2}\right).
\nonumber
\end{align}
Because $q, r \ge 2$ and $r \ne \infty$, we can apply estimate \eqref{LP1} in Lemma \ref{LPlemma} to get
\begin{align*}
||u||_{L_{t}^{q}([0, T]; L_{x}^{r}(\mathbb{R}^{n}))}^{2} \le C\sum_{j = -\infty}^{\infty} ||u_{j}||_{L_{t}^{q}([0, T]; L_{x}^{r}(\mathbb{R}^{n}))}^{2} 
\le CC'\sum_{j = -\infty}^{\infty} \left(||f_{j}||_{\dot{H}^{s}(\mathbb{R}^{n})}^{2} + ||g_{j}||_{\dot{H}^{s - 1}(\mathbb{R}^{n})}^{2}\right), 
\end{align*}
where the last inequality follows from \eqref{reductionest1}.
Because $\sum_{j = -\infty}^{\infty} \beta^{2}\left(\frac{|\xi|}{2^{j}}\right) \le 1$, we have that
\begin{equation}\label{reductionest2}
\sum_{j = -\infty}^{\infty} ||f_{j}||^{2}_{\dot{H}^{s}(\mathbb{R}^{n})} \le ||f||^{2}_{\dot{H}^{s}(\mathbb{R}^{n})}
\ {\rm and} \ 
\sum_{j = -\infty}^{\infty} ||g_{j}||^{2}_{\dot{H}^{s - 1}(\mathbb{R}^{n})} \le ||g||^{2}_{\dot{H}^{s - 1}(\mathbb{R}^{n})}.
\end{equation}
Combining the last three estimates, we obtain that there exists a $C > 0$ independent of $f, g$, such that
\begin{equation}\label{reductionest4}
||u||_{L_{t}^{q}([0, T]; L_{x}^{r}(\mathbb{R}^{n}))}^{2} \le C\left(||f||_{\dot{H}^{s}(\mathbb{R}^{n})}^{2} + ||g||_{\dot{H}^{s - 1}(\mathbb{R}^{n})}^{2}\right) \le C\left(||f||_{\dot{H}^{s}(\mathbb{R}^{n})} + ||g||_{\dot{H}^{s - 1}(\mathbb{R}^{n})}\right)^{2}.
\end{equation}
Similarly, using the fact that for some constant $c > 0$,
$
\sum_{j = -\infty}^{\infty} \beta^{2}\left(\frac{|\xi|}{2^{j}}\right) \ge c, 
$
we conclude that
\begin{equation*}
||u(T, \cdot)||^{2}_{\dot{H}^{s}(\mathbb{R}^{n})} \le c^{-1}\sum_{j = -\infty}^{\infty} ||u_{j}(T, \cdot)||^{2}_{\dot{H}^{s}(\mathbb{R}^{n})},
\ {\rm and} \ 
||\partial_{t}u(T, \cdot)||^{2}_{\dot{H}^{s}(\mathbb{R}^{n})} \le c^{-1}\sum_{j = -\infty}^{\infty} ||\partial_{t}u_{j}(T, \cdot)||^{2}_{\dot{H}^{s}(\mathbb{R}^{n})},
\end{equation*}
where in the second inequality, we used the fact that $\widehat{\partial_{t}u_{j}}(t, \xi) = \beta(|\xi|/2^{j})\widehat{\partial_{t}u}(t, \xi)$. Then by using \eqref{reductionest1} and the inequalities in \eqref{reductionest2}, we can obtain the analogous inequalities
\begin{equation}\label{reductionest5}
||u(T, \cdot)||_{\dot{H}^{s}}^{2}
\le C\left(||f||_{\dot{H}^{s}} + ||g||_{\dot{H}^{s - 1}}\right)^{2}
\ {\rm and} \ 
||\partial_{t}u(T, \cdot)||_{\dot{H}^{s - 1}}^{2}
\le C\left(||f||_{\dot{H}^{s}} + ||g||_{\dot{H}^{s - 1}}\right)^{2}.
\end{equation}
Taking square roots in \eqref{reductionest4}, \eqref{reductionest5}, and adding the resulting equations gives the result in Theorem \ref{strongdecay}. 

In the next step, we make a further simplification as follows.

\vskip 0.1in
\noindent
{{{\bf{Step 2.}} We have shown in the previous step that it suffices to show the uniform estimate  \eqref{j_estimate} for all $j \in \mathbb{Z}$ and for all initial data $(f, g)$ with corresponding solution $u$. In this step, we show that because of the gap condition \eqref{gap1}, it suffices to show \eqref{j_estimate} for just $j = 0$. In particular, showing the estimate \eqref{j_estimate} for $j = 0$ with a constant $C$ automatically gives the same estimate for all $j \in \mathbb{Z}$ with the same constant $C$, by the scaling symmetries of the viscous linear wave equation.}}

To see this, recall that 
\begin{equation*}
\widehat{h(\lambda x)} = \lambda^{-n}\widehat{h}(\xi/\lambda).
\end{equation*}
So it suffices to show that an estimate for a given $f, g, u$ also holds for the corresponding functions
$
f(\lambda x), \lambda g(\lambda x), u(\lambda t, \lambda x).
$
To verify this, we calculate
\begin{equation*}
||u(\lambda t, \lambda x)||_{L^{q}_{t}([0, T]; L^{r}_{x}(\mathbb{R}^{n}))} = \lambda^{-\frac{n}{r} - \frac{1}{q}} ||u(t, x)||_{L^{q}_{t}([0, \lambda T]; L^{r}_{x}(\mathbb{R}^{n}))},
\end{equation*}
\begin{equation*}
||u(\lambda T, \lambda x)||_{\dot{H}^{s}(\mathbb{R}^{n})} = \lambda^{-\frac{n}{2} + s}||u(\lambda T, x)||_{\dot{H}^{s}(\mathbb{R}^{n})},
\quad 
||\lambda \partial_{t}u(\lambda T, \lambda x)||_{\dot{H}^{s - 1}(\mathbb{R}^{n})} = \lambda^{-\frac{n}{2} + s} ||\partial_{t}u(\lambda T, x)||_{\dot{H}^{s - 1}(\mathbb{R}^{n})},
\end{equation*}
\begin{equation*}
||f(\lambda x)||_{\dot{H}^{s}(\mathbb{R}^{n})} = \lambda^{-\frac{n}{2} + s} ||f(x)||_{\dot{H}^{s}(\mathbb{R}^{n})},
\quad 
||\lambda g(\lambda x)||_{\dot{H}^{s - 1}(\mathbb{R}^{n})} = \lambda^{-\frac{n}{2} + s} ||g(x)||_{\dot{H}^{s - 1}(\mathbb{R}^{n})}.
\end{equation*}
From here we see that we get the desired result since, by the gap condition,
\begin{equation*}
\frac{1}{q} + \frac{n}{r} = \frac{n}{2} - s.
\end{equation*}

\noindent
\textbf{Conclusion.} {{From Steps 1 and 2 we conclude that we just need to show the estimate \eqref{j_estimate} for $j = 0$, for any initial data $(f, g)$ and corresponding solution $u$. Since $f_{0}, g_{0}$ have spatial Fourier transforms supported in $1/2 \le |\xi| \le 2$, estimate \eqref{j_estimate} for $j = 0$ would be established if we more generally proved Theorem \ref{strongdecay} for all initial data $(f, g)$ that have $\widehat{f}, \widehat{g}$ supported in $1/2 \le |\xi| \le 2$.}} 
Thus,  
without loss of generality, we can assume that $f, g$ have spatial Fourier transforms 
that are all supported in the annulus $1/2 \le |\xi| \le 2$. 
Note that, in this case, all homogeneous Sobolev norms $\dot{H}^{s}(\mathbb{R}^{n})$ are equivalent
 to the $L^{2}(\mathbb{R}^{n})$ norm.

\vskip 0.1in
\noindent
{\bf Step 3.}
The proof of Theorem \ref{strongdecay} now follows by combining Steps 1 and 2, the expression \eqref{ujform} for $u_j$ in terms of 
the operators $U^{(j)}$ and $V^{(j)}$, and the estimates from Lemma~\ref{operatorest}.

More precisely,
recall that the homogeneous solution can be written as
\begin{equation*}
u(t, \cdot) = e^{-\frac{\sqrt{-\Delta}}{2}t}\left(\text{cos}\left(\frac{\sqrt{3}}{2}\sqrt{-\Delta}t\right) + \frac{1}{\sqrt{3}}\text{sin}\left(\frac{\sqrt{3}}{2}\sqrt{-\Delta}t\right)\right)f + e^{-\frac{\sqrt{-\Delta}}{2}t}\frac{\text{sin}\left(\frac{\sqrt{3}}{2}\sqrt{-\Delta}t\right)}{\frac{\sqrt{3}}{2}\sqrt{-\Delta}}g,
\end{equation*}
and we want to prove
\begin{equation*}
||u||_{L_{t}^{q}([0, T]; L_{x}^{r}(\mathbb{R}^{n}))} + ||u(T, \cdot)||_{\dot{H}^{s}(\mathbb{R}^{n})} + ||\partial_{t}u(T, \cdot)||_{\dot{H}^{s - 1}(\mathbb{R}^{n})} \le C||f||_{\dot{H}^{s}(\mathbb{R}^{n})} + C||g||_{\dot{H}^{s - 1}(\mathbb{R}^{n})}.
\end{equation*}
This follows from the estimates on the operators $U^{(j)}$ and $V^{(j)}$,  $-2 \le j \le 2$, in Lemma \ref{operatorest},
and by using 
Steps 1 and 2 above to work with $f, g$ which have Fourier transform supported in $1/2 \le |\xi| \le 2$ so that:
\begin{align*}
\int_{\mathbb{R}^{n}} e^{ix\cdot \xi} e^{-\frac{|\xi|}{2}t} &\left(\text{cos}\left(\frac{\sqrt{3}}{2}|\xi|t\right) + \frac{1}{\sqrt{3}}\sin\left(\frac{\sqrt{3}}{2}|\xi|t\right)\right) \widehat{f}(\xi) d\xi \\
&= \sum_{j = -2}^{2} \int_{\mathbb{R}^{n}} e^{ix\cdot \xi} e^{-\frac{|\xi|}{2}t} \left(\text{cos}\left(\frac{\sqrt{3}}{2}|\xi|t\right) + \frac{1}{\sqrt{3}}\sin\left(\frac{\sqrt{3}}{2}|\xi|t\right)\right) \widehat{f}(\xi) \beta\left(\frac{|\xi|}{2^{j}}\right) d\xi.
\end{align*}
Notice that we need $-2 \le j \le 2$ to cover all the $j$'s for which the support of $\beta\left(\frac{|\xi|}{2^{j}}\right)$ intersects $1/2 \le |\xi| \le 2$.
For $g$, we use the estimates on $V^{(j)}$, $-2 \le j \le 2$, in Lemma \ref{operatorest}, and the same sum decomposition where we sum from $j = -2$ to $j = 2$. 

For the terms $||u(T, \cdot)||_{\dot{H}^{s}(\mathbb{R}^{n})}$ and $||\partial_{t}u(T, \cdot)||_{\dot{H}^{s - 1}(\mathbb{R}^{n})}$, we note that these norms are equivalent to $L^{2}$ norms since $u$ and $\partial_{t}u$ have spatial Fourier transforms that are also supported in $1/2 \le |\xi| \le 2$, in which case the inequality follows from Plancherel's theorem and the support properties of $f$ and $g$. 
Note again that Plancherel's theorem works here 
since $||f||_{\dot{H}^{s}(\mathbb{R}^{n})}$ and $||g||_{\dot{H}^{s - 1}(\mathbb{R}^{n})}$ 
are equivalent to $||f||_{L^{2}(\mathbb{R}^{n})}$ and $||g||_{L^{2}(\mathbb{R}^{n})}$ respectively, 
given the fact that  the Fourier transforms of $f$ and $g$ are supported in $1/2 \le |\xi| \le 2$.

\qed

\subsubsection{Proof of Strichartz estimates for the inhomogeneous problem}
To prove Theorem \ref{inhStr} we first recall from \eqref{Fouriersoln2} that for the inhomogeneous viscous wave equation with zero initial data, the solution can be represented as
\begin{equation}\label{u_inhomo}
u(t, \cdot) := \int_{0}^{t} e^{-\frac{\sqrt{-\Delta}}{2}(t - \tau)}\frac{\text{sin}\left(\frac{\sqrt{3}}{2}\sqrt{-\Delta}(t - \tau)\right)}{\frac{\sqrt{3}}{2}\sqrt{-\Delta}}F(\tau, \cdot) d\tau.
\end{equation}
The goal is to estimate this inhomogeneous contribution. 
In particular, we must show that
\begin{equation}\label{inhest}
||u||_{L_{t}^{q}([0, T]; L_{x}^{r}(\mathbb{R}^{n}))} + ||u(T, \cdot)||_{\dot{H}^{s}(\mathbb{R}^{n})} + ||\partial_{t}u(T, \cdot)||_{\dot{H}^{s - 1}(\mathbb{R}^{n})} \le C_{q, \tilde{q}, r, \tilde{r}}||F||_{L_{t}^{\tilde{q}'}([0, T]; L_{x}^{\tilde{r}'}(\mathbb{R}^{n}))}.
\end{equation}
Unfortunately, since the dissipative portion of the evolution operator involving $e^{-\frac{\sqrt{-\Delta}}{2}}$ is self-adjoint, 
the results from Keel and Tao's theorem cannot be used here, 
since $U(t)U^{*}(\tau)$ has $e^{-\frac{\sqrt{-\Delta}}{2}(t + \tau)}$ instead of $e^{-\frac{\sqrt{-\Delta}}{2}(t - \tau)}$.
Instead, we adopt ideas from fractional heat equations, see Miao, Yuan, and Zhang \cite{MYZ}. 

%%%
\if 1 = 0
\textit{Note however that the last two statements we get from Keel and Tao's Strichartz estimate theorem do not apply here, since $U(t)U^{*}(\tau)$ has $e^{-\frac{\sqrt{-\Delta}}{2}(t + \tau)}$ instead of $e^{-\frac{\sqrt{-\Delta}}{2}(t - \tau)}$}. This important observation shows how the dissipative portion of the evolution operator, in particular the self-adjoint portion involving $e^{-\frac{\sqrt{-\Delta}}{2}}$, makes the problem more difficult when we consider the inhomogeneous estimates.
\fi
%%%%%%%%%%%%%
%
%

We start by first proving that
\begin{equation}\label{inhestpart1}
||u||_{L_{t}^{q}([0, T]; L_{x}^{r}(\mathbb{R}^{n}))} \le C_{q, \tilde{q}, r, \tilde{r}}||F||_{L_{t}^{\tilde{q}'}([0, T]; L_{x}^{\tilde{r}'}(\mathbb{R}^{n}))} .
\end{equation}
%To do this, we use ideas that have been applied for example to fractional heat equations by Miao, Yuan, and Zhang \cite{MYZ}. 
For this purpose, we introduce the following family of operators 
\begin{equation*}
S(t) := e^{-\frac{\sqrt{-\Delta}}{2}t} \frac{\text{sin}\left(\frac{\sqrt{3}}{2}\sqrt{-\Delta}t\right)}{\frac{\sqrt{3}}{2}\sqrt{-\Delta}}, \quad t > 0.
\end{equation*}
The operators $S(t)$ define a family of convolution kernels $K_t(x)$  via
\begin{equation*}
S(t)\phi(x) = (K_{t} * \phi)(x),
\end{equation*}
where 
\begin{equation}\label{Kt}
K_{t}(x) := \frac{1}{(2\pi)^{n}} \int_{\mathbb{R}^{n}} e^{-\frac{|\xi|}{2}t}\frac{\text{sin}\left(\frac{\sqrt{3}}{2}|\xi|t\right)}{\frac{\sqrt{3}}{2}|\xi|}e^{ix\cdot \xi} d\xi = \frac{t^{1-n}}{(2\pi)^{n}} \int_{\mathbb{R}^{n}} e^{-\frac{|\xi|}{2}}\frac{\text{sin}\left(\frac{\sqrt{3}}{2}|\xi|\right)}{\frac{\sqrt{3}}{2}|\xi|} e^{ix \cdot \frac{\xi}{t}} d\xi.
\end{equation}
Using these operators, the solution $u$ in \eqref{u_inhomo} can be written as:
$$
u(t,x) = \int_0^t \left[K_{t-\tau}(x)  * F(\tau,x) \right] d\tau, 
$$
where the convolution $*$ is with respect to $x$.
To obtain the desired estimates on $u$, we investigate the properties of the convolution kernels $K_t(x)$.
In particular, we begin by defining the ``unit-scale" kernel by
\begin{equation}\label{kernel}
K(x) = \frac{1}{(2\pi)^{n}} \int_{\mathbb{R}^{n}} e^{-\frac{|\xi|}{2}}\frac{\text{sin}\left(\frac{\sqrt{3}}{2}|\xi|\right)}{\frac{\sqrt{3}}{2}|\xi|} e^{ix \cdot \xi} d\xi,
\end{equation}
and notice that it has the following important scaling property:
\begin{equation}\label{kernelscale}
K_{t}(x) = t^{1 - n} K\left(\frac{x}{t}\right).
\end{equation}

%We will now follow the Miao, Yuan, and Zhang \cite{MYZ}, who consider similar estimates for the fractional operators $e^{-t(-\Delta)^{\alpha}}$ for $\alpha > 0$. We start with the following lemma.

\begin{lemma}\label{kernellemma}
There exists a constant $C > 0$ such that the convolution kernel $K(x)$ satisfies the following pointwise estimate:
\begin{equation*}
|K(x)| \le C(1 + |x|)^{-n-\alpha}, \ x \in \R^n,
\end{equation*}
for some $\alpha \ge 1/2$.
Therefore, 
$K \in L^{p}(\mathbb{R}^{n})$ for all $1 \le p \le \infty$.
\end{lemma}

\begin{proof}
The proof is similar to an argument used in Miao, Yuan, and Zhang \cite{MYZ} for a different kernel,
where we use repeated integration by parts until we get the necessary decay. 

More precisely, to perform integration by parts and to obtain the desired estimates, it is useful to introduce the following operator:
\begin{equation*}
L(x, D) = \frac{x \cdot \nabla_{\xi}}{i|x|^{2}},\quad {\rm{and \ its \ adjoint}} \quad L^{*}(x, D) = -\frac{x \cdot \nabla_{\xi}}{i|x|^{2}}.
\end{equation*}
Operator $L(x, D)$ has the following crucial property 
\begin{equation*}
L(x, D) e^{ix\cdot \xi} = e^{ix \cdot \xi}.
\end{equation*}
We now integrate by parts to obtain:
\begin{align}
K(x) &= \frac{1}{(2\pi)^{n}} \int_{\mathbb{R}^{n}} e^{-\frac{|\xi|}{2}}\frac{\text{sin}\left(\frac{\sqrt{3}}{2}|\xi|\right)}{\frac{\sqrt{3}}{2}|\xi|} e^{ix \cdot \xi} d\xi 
= \frac{1}{(2\pi)^{n}} \int_{\mathbb{R}^{n}} L(e^{ix \cdot \xi}) e^{-\frac{|\xi|}{2}}\frac{\text{sin}\left(\frac{\sqrt{3}}{2}|\xi|\right)}{\frac{\sqrt{3}}{2}|\xi|} d\xi
\nonumber
 \\
&= \frac{1}{(2\pi)^{n}} \int_{\mathbb{R}^{n}} e^{ix \cdot \xi} L^{*}\left(e^{-\frac{|\xi|}{2}}\frac{\text{sin}\left(\frac{\sqrt{3}}{2}|\xi|\right)}{\frac{\sqrt{3}}{2}|\xi|}\right) d\xi.
\label{kernel}
\end{align}
Integration by parts is justified because $e^{-\frac{|\xi|}{2}}\frac{\text{sin}\left(\frac{\sqrt{3}}{2}|\xi|\right)}{\frac{\sqrt{3}}{2}|\xi|}$ and all of its $\xi$ derivatives are rapidly decreasing at infinity, {{and $e^{-\frac{|\xi|}{2}}\frac{\text{sin}\left(\frac{\sqrt{3}}{2}|\xi|\right)}{\frac{\sqrt{3}}{2}|\xi|}$ is bounded near the origin}}.

To estimate the resulting integral, we will do a high and a low frequency estimate. 
For this purpose, define a radially symmetric smooth compactly supported function $\rho$ such that:
\begin{equation*}
\rho(\xi) = 
1 \ \text{ if } |\xi| \le 1, \qquad \rho(\xi) =  0 \ \text{ if } |\xi| \ge 2, \quad {\text{and}} \ \text{$\rho$ is decreasing radially.}
\end{equation*}
Then, we can split the integral \eqref{kernel} above into two integrals,
\begin{align}\label{kernelfreqsplit}
K(x) &= \frac{1}{(2\pi)^{n}} \int_{\mathbb{R}^{n}} e^{ix \cdot \xi} \rho\left(\frac{\xi}{\delta}\right) L^{*}\left(e^{-\frac{|\xi|}{2}}\frac{\text{sin}\left(\frac{\sqrt{3}}{2}|\xi|\right)}{\frac{\sqrt{3}}{2}|\xi|}\right) d\xi
\nonumber
\\ 
&+ \frac{1}{(2\pi)^{n}} \int_{\mathbb{R}^{n}} e^{ix \cdot \xi}\left(1 -  \rho\left(\frac{\xi}{\delta}\right)\right) L^{*}\left(e^{-\frac{|\xi|}{2}}\frac{\text{sin}\left(\frac{\sqrt{3}}{2}|\xi|\right)}{\frac{\sqrt{3}}{2}|\xi|}\right) d\xi := I + II,
\end{align}
where $\delta > 0$ will be chosen later.

\vskip 0.1in
\noindent
\textbf{Estimate of integral $I$.}  We start by estimating the factor involving the operator $L^*$:
%To do this, note that for $L^{*} = -\frac{x \cdot \nabla_{\xi}}{i|x|^{2}}$, by Cauchy-Schwarz, 
\begin{align*}
&\left| L^{*}  \left(e^{-\frac{|\xi|}{2}}  \frac{\text{sin}\left(\frac{\sqrt{3}}{2}|\xi|\right)}{\frac{\sqrt{3}}{2}|\xi|} \right) \right| 
= \left| \frac{x \cdot \nabla_{\xi}}{i|x|^{2}} \left(e^{-\frac{|\xi|}{2}}\frac{\text{sin}\left(\frac{\sqrt{3}}{2}|\xi|\right)}{\frac{\sqrt{3}}{2}|\xi|} \right) \right| 
\le \frac{1}{|x|}\left|\nabla_{\xi} \left(e^{-\frac{|\xi|}{2}}\frac{\text{sin}\left(\frac{\sqrt{3}}{2}|\xi|\right)}{\frac{\sqrt{3}}{2}|\xi|}\right)\right| \\
&\qquad \qquad = \frac{1}{|x|}\left|-\frac{1}{2}e^{-\frac{|\xi|}{2}}\frac{\text{sin}\left(\frac{\sqrt{3}}{2}|\xi|\right)}{\frac{\sqrt{3}}{2}|\xi|} + e^{-\frac{|\xi|}{2}}\left(\frac{\text{cos}\left(\frac{\sqrt{3}}{2}|\xi|\right)}{|\xi|} - \frac{\text{sin}\left(\frac{\sqrt{3}}{2}|\xi|\right)}{\frac{\sqrt{3}}{2}|\xi|^{2}}\right)\right| \cdot \left|\nabla_{\xi}|\xi|\right|.
\end{align*}
Now, from the Taylor expansion of $\sin$ and $\cos$, and the boundedness of $\left|\nabla_{\xi}|\xi|\right|$, there exists a constant $C > 0$ such that
%\begin{equation*}
%\left|\frac{\text{cos}\left(\frac{\sqrt{3}}{2}|\xi|\right)}{|\xi|} - \frac{\text{sin}\left(\frac{\sqrt{3}}{2}|\xi|\right)}{\frac{\sqrt{3}}{2}|\xi|^{2}}\right| \le C|\xi|.
%\end{equation*}
%Therefore,
%\begin{equation*}
%\left|e^{-\frac{|\xi|}{2}}\frac{\text{cos}\left(\frac{\sqrt{3}}{2}|\xi|\right)}{|\xi|} - \frac{\text{sin}\left(\frac{\sqrt{3}}{2}|\xi|\right)}{\frac{\sqrt{3}}{2}|\xi|^{2}}\right| \le C|%\xi|e^{-\frac{|\xi|}{2}} \le C'.
%\end{equation*}
%In addition, $\left|\nabla_{\xi}|\xi|\right|$ is also bounded and so:
%\begin{equation*}
%\left|L^{*}\left(e^{-\frac{|\xi|}{2}}\frac{\text{sin}\left(\frac{\sqrt{3}}{2}|\xi|\right)}{\frac{\sqrt{3}}{2}|\xi|}\right)\right| \le C|x|^{-1},
%\end{equation*}
%which implies the following estimate of the integral $I$:
\begin{equation}\label{kernellowfreqest}
|I| = \left|\frac{1}{(2\pi)^{n}} \int_{\mathbb{R}^{n}} e^{ix \cdot \xi} \rho\left(\frac{\xi}{\delta}\right) L^{*}\left(e^{-\frac{|\xi|}{2}}\frac{\text{sin}\left(\frac{\sqrt{3}}{2}|\xi|\right)}{\frac{\sqrt{3}}{2}|\xi|}\right) d\xi\right| \le C|x|^{-1} \int_{|\xi| \le 2\delta} d\xi = C|x|^{-1}\delta^{n}.
\end{equation}

\vskip 0.1in
\noindent
\textbf{Estimate of integral II.}
Let $N$ be an arbitrary positive integer such that $N \ge n + 1$. Using integration by parts $N-1$ times, we get:
\begin{align*}
|II| &= \left|\frac{1}{(2\pi)^{n}} \int_{\mathbb{R}^{n}} e^{ix \cdot \xi}\left(1 -  \rho\left(\frac{\xi}{\delta}\right)\right) L^{*}\left(e^{-\frac{|\xi|}{2}}\frac{\text{sin}\left(\frac{\sqrt{3}}{2}|\xi|\right)}{\frac{\sqrt{3}}{2}|\xi|}\right) d\xi\right| \\
&\le C\int_{\mathbb{R}^{n}} \left|e^{ix \cdot \xi} (L^{*})^{N - 1} \left(\left(1 -  \rho\left(\frac{\xi}{\delta}\right)\right) L^{*}\left(e^{-\frac{|\xi|}{2}}\frac{\text{sin}\left(\frac{\sqrt{3}}{2}|\xi|\right)}{\frac{\sqrt{3}}{2}|\xi|}\right)\right)\right| d\xi \\ 
&= C\int_{|\xi| \ge \delta} \left|e^{ix \cdot \xi} (L^{*})^{N - 1} \left(\left(1 -  \rho\left(\frac{\xi}{\delta}\right)\right) L^{*}\left(e^{-\frac{|\xi|}{2}}\frac{\text{sin}\left(\frac{\sqrt{3}}{2}|\xi|\right)}{\frac{\sqrt{3}}{2}|\xi|}\right)\right)\right| d\xi.
\end{align*}
By the triangle inequality and the support properties of $\rho$, we get:
\begin{align}
|II| &\le C\int_{|\xi| \ge \delta} \left|(L^{*})^{N}\left(e^{-\frac{|\xi|}{2}}\frac{\text{sin}\left(\frac{\sqrt{3}}{2}|\xi|\right)}{\frac{\sqrt{3}}{2}|\xi|}\right)\right| d\xi 
\nonumber
\\
&+ C\int_{\delta \le |\xi| \le 2\delta} \left|(L^{*})^{N - 1}\left(\rho\left(\frac{\xi}{\delta}\right)L^{*}\left(e^{-\frac{|\xi|}{2}}\frac{\text{sin}\left(\frac{\sqrt{3}}{2}|\xi|\right)}{\frac{\sqrt{3}}{2}|\xi|}\right)\right)\right| d\xi := II_{A} + II_{B}.
\label{highfreq1}
\end{align}

We estimate $II_{A}$ as follows:
\begin{align}
II_{A} &\le C_{N}|x|^{-N}\int_{|\xi| \ge \delta} \sum_{l = 0}^{N} \left|e^{-\frac{|\xi|}{2}}\text{sin}\left(\frac{\sqrt{3}}{2}|\xi|\right)\right| |\xi|^{-l-1} + \sum_{l = 0}^{N - 1} \left|e^{-\frac{|\xi|}{2}}\text{cos}\left(\frac{\sqrt{3}}{2}|\xi|\right)\right| |\xi|^{-l-1} d\xi 
\nonumber
\\
&\le C_{N}|x|^{-N}\int_{|\xi| \ge \delta} \sum_{l = 0}^{N} e^{-\frac{|\xi|}{2}}|\xi|^{-l} d\xi,
\label{IIA}
\end{align}
where the cosine part has the index $l$ going only up to $N - 1$, since at least one $\xi$ derivative must be used to turn sine into cosine. 
%The last step uses the fact that $\frac{\text{sin}\left(\frac{\sqrt{3}}{2}|\xi|\right)}{|\xi|}$ is bounded. 

We estimate $II_{B}$ as follows:
\begin{align*}
II_{B} &\le C_{N}|x|^{-N}\int_{\delta \le |\xi| \le 2\delta}\sum_{k = 0}^{N - 1} \delta^{-k} \Bigg(\sum_{l = 0}^{N-k} \left|e^{-\frac{|\xi|}{2}}\text{sin}\left(\frac{\sqrt{3}}{2}|\xi|\right)\right| |\xi|^{-l-1} \\
&\qquad\qquad + \sum_{l = 0}^{N-k-1} \left|e^{-\frac{|\xi|}{2}}\text{cos}\left(\frac{\sqrt{3}}{2}|\xi|\right)\right| |\xi|^{-l-1}\Bigg) d\xi \\
&\le C_{N}|x|^{-N}\int_{\delta \le |\xi| \le 2\delta}\sum_{k = 0}^{N - 1} |\xi|^{-k} \Bigg(\sum_{l = 0}^{N-k} \left|e^{-\frac{|\xi|}{2}}\text{sin}\left(\frac{\sqrt{3}}{2}|\xi|\right)\right| |\xi|^{-l-1} \\ 
&\qquad\qquad + \sum_{l = 0}^{N-k-1} \left|e^{-\frac{|\xi|}{2}}\text{cos}\left(\frac{\sqrt{3}}{2}|\xi|\right)\right| |\xi|^{-l-1}\Bigg) d\xi,
\end{align*}
where we used the fact that $\delta^{-k} \le 2^{N-1}|\xi|^{-k}$ for $k = 0, 1, ..., N-1$ and for $\delta \le |\xi| \le 2\delta$. Continuing to estimate the above quantity, we get
\begin{align}
II_{B} &\le C_{N}|x|^{-N} \int_{\delta \le |\xi| \le 2\delta} \sum_{k = 0}^{N-1}|\xi|^{-k}\left(\sum_{l = 0}^{N - k} e^{-\frac{|\xi|}{2}}|\xi|^{-l}\right) d\xi 
\nonumber
\\
&= C_{N}|x|^{-N}\int_{\delta \le |\xi| \le 2\delta} \sum_{k = 0}^{N-1}\sum_{l = 0}^{N-k}e^{-\frac{|\xi|}{2}}|\xi|^{-k-l} d\xi,
\label{IIB}
\end{align}
 where $k$ is the number of derivatives that fall on $\rho\left(\frac{\xi}{\delta}\right)$. 
Note that for all $j = 0, 1, ..., N$, we have that there exists a constant $C_{N}$ such that
$
0 \le e^{-\frac{|\xi|}{2}}|\xi|^{j} \le C_{N},
$
for all $j = 0, 1, ..., N$ and for all $\xi \in \mathbb{R}^{n}$. Using this fact in \eqref{IIA} and \eqref{IIB}, we get that
\begin{equation}\label{newIIAB}
II_{A} \le C_{N}|x|^{-N} \int_{|\xi| \ge \delta} |\xi|^{-N}d\xi \quad  {\rm and} \quad  II_{B} \le C_{N}|x|^{-N} \int_{\delta \le |\xi| \le 2\delta} |\xi|^{-N}d\xi.
\end{equation}
By employing estimates \eqref{newIIAB} in \eqref{highfreq1}, and recalling that $N$ is an arbitrary positive integer such that $N \ge n + 1$ so that the integral converges, we get
\begin{equation}\label{highfreq2}
|II| \le C_{N}|x|^{-N} \int_{|\xi| \ge \delta} |\xi|^{-N} d\xi = C_{N}|x|^{-N}\delta^{-N + n}.
\end{equation}

\noindent
{\bf Estimate of $K(x)$.}
Using the low frequency estimate \eqref{kernellowfreqest} and high frequency estimate \eqref{highfreq2} in \eqref{kernelfreqsplit}, we obtain
\begin{equation*}
|K(x)| \le C_{N}(|x|^{-1}\delta^{n} + |x|^{-N}\delta^{-N + n}).
\end{equation*}
Now we set $\delta = |x|^{-\frac{N-1}{N}}$ in the above inequality so that both terms in the sum are the same, to obtain
\begin{equation}\label{finalkernelest}
|K(x)| \le C_{N}|x|^{-1 - n\left(\frac{N - 1}{N}\right)}.
\end{equation}
Since $N \ge n + 1$ was arbitrary, choose a positive integer $N$ sufficiently large such that $\alpha := -1 -n\left(\frac{N - 1}{N}\right) \le -n - \frac{1}{2}$. Note that $K(x)$ is bounded uniformly in $x$, since $e^{-\frac{|\xi|}{2}}$ is integrable in $\mathbb{R}^{n}$. 
Then, we conclude that there exists a constant $C > 0$ such that
\begin{equation*}
|K(x)| \le C(1 + |x|)^{-n-\alpha}, \ x\in \R^n,
\end{equation*}
for some $\alpha \ge 1/2$. Therefore, $K(x) \in L^{p}(\mathbb{R}^{n})$ for all $1 \le p \le \infty$. 
\end{proof}

{{}{In fact, there is an explicit formula for $K(x)$, which will allow us to explicitly show a sharp decay rate for $K(x)$ in dimensions $n = 1, 2, 3$. To state this formula, we must first recall the fundamental solution for the wave equation, given by the following inverse Fourier transform,
\begin{equation}\label{wavefundamental}
K^{W}_{t}(x) = \frac{1}{(2\pi)^{n}} \int_{\R^{n}} \frac{\sin(t|\xi|)}{|\xi|} e^{ix\cdot \xi} d\xi,
\end{equation}
where superscript ``W'' stands for the wave equation.
The explicit form of the fundamental solution for the wave equation is well-known for all dimensions, but is particularly easy to state in dimensions $n = 1, 2, 3$. We give the well-known formulas below, which can be deduced from the discussion on pg.~4-5 of Sogge \cite{S}. 

\begin{align}\label{explicitwave}
&\text{For } n = 1, \qquad K^{W}_{t}(x) = \frac{1}{2}1_{|x| \le t}.\nonumber \\
&\text{For } n = 2, \qquad K^{W}_{t}(x) = \frac{1}{2\pi} \frac{1}{\sqrt{t^{2} - |x|^{2}}} 1_{|x| < t}.\nonumber \\
&\text{For } n = 3, \qquad K^{W}_{t}(x) = \frac{1}{4\pi t} \sigma_{t}.
\end{align}
Here, $\sigma_{t}$ denotes Lebesgue measure on the sphere of radius $t$. There are more complicated formulas for higher dimensions, and in particular, for $n \ge 3$, $K^{W}_{t}(x)$ is distribution-valued. 

\begin{proposition}\label{explicitK}
For arbitrary dimension $n$,
\begin{equation}\label{explicitKform}
K(x) = c_{n} \left(P * K^{W}_{t = \frac{\sqrt{3}}{2}}\right)(x),
\end{equation}
where $*$ denotes a spatial convolution and
\begin{equation*}
P(x) = \frac{1}{(1 + 4|x|^{2})^{\frac{n + 1}{2}}}.
\end{equation*}
\end{proposition}

\begin{proof}

Recall that $K(x)$ is defined by the inverse Fourier transform given in \eqref{kernel}. We first consider the inverse Fourier transform of $e^{-\frac{|\xi|}{2}}$ and $\frac{\sin\left(\frac{\sqrt{3}}{2}|\xi|\right)}{\frac{\sqrt{3}}{2}|\xi|}$ separately. The Poisson kernel, corresponding to the operator $e^{-\sqrt{-\Delta}}$, is well-understood and in particular,
\begin{equation*}
\frac{1}{(2\pi)^{n}} \int_{\mathbb{R}^{n}} e^{-\frac{|\xi|}{2}}e^{ix \cdot \xi} d\xi = \frac{c_{n}}{(1 + 4|x|^{2})^{\frac{n + 1}{2}}} := c_{n} P(x),
\end{equation*}
for some constant $c_{n}$ depending only on the dimension $n$. By formula \eqref{explicitwave}, 
\begin{equation*}
\frac{1}{(2\pi)^{n}} \int_{\R^{n}} \frac{\sin(\frac{\sqrt{3}}{2}|\xi|)}{\frac{\sqrt{3}}{2}|\xi|} e^{ix\cdot \xi} d\xi = \frac{2}{\sqrt{3}} K^{W}_{t = \frac{\sqrt{3}}{2}}(x).
\end{equation*}
The result follows from the fact that the Fourier transform interchanges multiplication and convolution, where we absorbed the extra factor of $\frac{2}{\sqrt{3}}$ into the dimension-dependent constant $c_{n}$. 
\end{proof}

Because the fundamental solution for the wave equation is particularly simple in dimensions $n = 1, 2, 3$, as given in \eqref{explicitwave}, we can use Proposition \ref{explicitK} to show a sharp decay rate for $K(x)$. This sharp decay rate estimate will verify the result proved for general dimension $n$ in Lemma \ref{kernellemma} for the explicit dimensions $n = 1, 2, 3$.
\begin{proposition}
For $n = 1, 2, 3$, the convolution kernel $K(x)$ has the following pointwise estimate:
\begin{equation*}
|K(x)| \le \frac{C}{(1 + |x|)^{n+1}}, \ x \in \R^n,
\end{equation*}
for some constant $C>0$. In particular, $K(x) \in L^{p}(\R^{n})$ for all $1 \le p \le \infty$ in dimensions $n = 1, 2, 3$.
\end{proposition}

\begin{proof}
We use the explicit formula for $K(x)$ given in Proposition \ref{explicitK}. Note that $P(x)$ is bounded near the origin, and decays as $O(|x|^{-n - 1})$ away from the origin. We begin with $n = 3$, then describe the case of $n = 1, 2$. 

For $n = 3$, because $K^{W}_{t = \frac{\sqrt{3}}{2}}$ is Lebesgue measure on the sphere of radius $\frac{\sqrt{3}}{2}$, $K(x)$ in dimension $n = 3$ is an average over the boundary of spheres of radius $\frac{\sqrt{3}}{2}$, 
\begin{align*}
K(x) &= C \left(\frac{1}{(1 + 4|x|^{2})^{\frac{n + 1}{2}}} * \sigma_{\frac{\sqrt{3}}{2}}\right)(x) = C \int_{|y| = \frac{\sqrt{3}}{2}} \frac{1}{(1 + 4|x - y|^{2})^{\frac{n + 1}{2}}} dy,
\end{align*}
from which we deduce the desired estimate for $n = 3$. We remark that in dimension $n = 3$, $K(x)$ (up to a constant factor) can be interpreted as spherical means of the Poisson kernel.

For $n = 1, 2$, one can use a similar interpretation of $K(x)$ as appropriate spatial averages of $P(x)$, and we sketch the idea. Note that $K^{W}_{t = \frac{\sqrt{3}}{2}}(x)$ is integrable and supported on the closed ball of radius $t$ for dimensions $n = 1, 2$. Therefore, from the formula in Proposition \ref{explicitK}, up to the constant factor, $K(x)$ in dimensions $n = 1$ is an average of $K(x)$ over closed intervals of radius $\frac{\sqrt{3}}{2}$,
\begin{equation*}
K(x) = C \int_{-\frac{\sqrt{3}}{2}}^{\frac{\sqrt{3}}{2}} \frac{1}{(1 + 4|x - y|^{2})^{\frac{n + 1}{2}}} dy,
\end{equation*}
and $K(x)$ in dimension $n = 2$ is a weighted average of $K(x)$ over open balls of radius $\frac{\sqrt{3}}{2}$,
\begin{equation*}
K(x) = C \int_{|y| < \frac{\sqrt{3}}{2}} \frac{1}{\sqrt{\frac{3}{4} - |y|^{2}}} \frac{1}{(1 + 4|x - y|^{2})^{\frac{n + 1}{2}}} dy.
\end{equation*}
Because $K(x)$ in both $n = 1, 2$ is, up to a constant factor, appropriate averages of values of $P(x)$, $K(x)$ hence shares the same decay rate as $P(x)$, which is $O(|x|^{-n - 1})$, and $K(x)$ is also bounded near the origin, since $P(x)$ is. This establishes the desired estimate for $n = 1, 2$. 

In particular, we see explicitly that $K(x) \in L^{p}$ for all $1 \le p \le \infty$ for $n = 1, 2, 3$.  
\end{proof}}}

Next, we use Lemma~\ref{kernellemma}  to prove an estimate on the action of $S(t)$ for $t > 0$. 
This will then be used to obtain the first part of the Strichartz estimate \eqref{S_estimate_inhomo}, namely to show
\begin{equation}\label{inhestpart1_2}
||u||_{L_{t}^{q}([0, T]; L_{x}^{r}(\mathbb{R}^{n}))} \le C_{q, \tilde{q}, r, \tilde{r}}||F||_{L_{t}^{\tilde{q}'}([0, T]; L_{x}^{\tilde{r}'}(\mathbb{R}^{n}))} .
\end{equation}
The following result is the analogue of Lemma 3.1 in  \cite{MYZ}.

\begin{lemma}\label{semigroupineq}
Let $1 \le r \le p \le \infty$. There exists a constant $C>0$, depending only on $p$ and $r$, such that for all $\phi \in L^{r}(\mathbb{R}^{n})$ and for all $t > 0$,
\begin{equation*}
||S(t)\phi(x)||_{L^{p}_{x}(\mathbb{R}^{n})} \le Ct^{1 - n\left(\frac{1}{r} - \frac{1}{p}\right)} ||\phi||_{L^{r}_{x}(\mathbb{R}^{n})}.
\end{equation*}
\end{lemma}

\begin{proof}
This follows from the previous lemma, by Young's inequality, and the scaling property of our convolution kernels $K_{t}$. In particular, recall that
\begin{equation*}
S(t)\phi(x) := e^{-\frac{\sqrt{-\Delta}}{2}t}\frac{\text{sin}\left(\frac{\sqrt{3}}{2}\sqrt{-\Delta}t\right)}{\frac{\sqrt{3}}{2}\sqrt{-\Delta}}\phi = (K_{t} * \phi)(x),
\ {\rm where} \ 
K_{t}(x) = t^{1 - n}K\left(\frac{x}{t}\right),
\end{equation*}
with the unit scale kernel $K$ given by \eqref{kernel}. By Young's inequality,
\begin{equation}\label{kernelYoung}
||S(t)\phi(x)||_{L^{p}_{x}(\mathbb{R}^{n})} = ||K_{t} * \phi||_{L^{p}_{x}(\mathbb{R}^{n})} \le ||K_{t}||_{L^{q}_{x}(\mathbb{R}^{n})} ||\phi||_{L^{r}_{x}(\mathbb{R}^{n})},
\end{equation}
where
\begin{equation*}
\frac{1}{q} + \frac{1}{r} = 1 + \frac{1}{p}.
\end{equation*}
Using the scaling of the kernel,
\begin{equation*}
||K_{t}||_{L^{q}_{x}(\mathbb{R}^{n})} = t^{1 - n}\left(\int_{\mathbb{R}^{n}} \left|K\left(\frac{x}{t}\right)\right|^{q} dx\right)^{1/q} = t^{1 - n + \frac{n}{q}}||K||_{L^{q}_{x}} = C_{p, r} t^{1 + n\left(\frac{1}{p} - \frac{1}{r}\right)}
\end{equation*}
Substituting this into \eqref{kernelYoung}, we obtain the final result that
\begin{equation*}
||S(t)\phi(x)||_{L^{p}_{x}(\mathbb{R}^{n})} \le C_{p, r}t^{1 - n\left(\frac{1}{r} - \frac{1}{p}\right)}||\phi||_{L^{r}_{x}(\mathbb{R}^{n})}
\end{equation*}
for $t > 0$ and for all $\phi \in L^{r}(\mathbb{R}^{n})$.
\end{proof}

{{}
{\begin{remark}
The result of Lemma \ref{semigroupineq} can be interpreted as a parabolic smoothing effect, due to the dissipative effects of the fluid viscosity. For parabolic equations, parabolic smoothing can often be captured by using energy estimates. Energy estimates for the linear viscous wave equation would correspond to the case of $r = 2$ in Lemma \ref{semigroupineq}, as a basic energy estimate would consider initial data $(f, g) \in H^{1}(\mathbb{R}^{n}) \times L^{2}(\mathbb{R}^{n})$ and $S(t)\phi$ in the statement of Lemma \ref{semigroupineq} corresponds to the solution of the linear viscous wave equation with zero initial displacement and initial velocity $\phi$. We chose to use Fourier methods to establish Lemma \ref{semigroupineq} because they easily allow for more general $\phi$ in $L^{r}(\mathbb{R}^{n})$ for $1 \le r \le \infty$, which allows for a wider range of exponents than basic energy estimates.
\end{remark}}}

Now, we have the necessary components to establish inequality \eqref{inhestpart1_2}. 
\begin{lemma}\label{inhStrlemma}
Let $n \ge 2$ and let $(q, r)$, $(\tilde{q}, \tilde{r})$ satisfy $1 < \tilde{q}' < q < \infty, 1 \le \tilde{r}' < r \le \infty$, and the gap condition \eqref{gap2}:
\begin{align*}
\frac{1}{q} + \frac{n}{r} = \frac{1}{\tilde{q}'} + \frac{n}{\tilde{r}'} - 2.
\end{align*}
Then, there exists a constant $C>0$ depending only on $q, \tilde{q}, r, \tilde{r}$, such that for all $0 < T \le \infty$, and for all $F \in L^{\tilde{q}'}_{t}([0, T]; L^{\tilde{r}'}_{x}(\mathbb{R}^{n}))$,
\begin{equation*}
\left|\left|\int_{0}^{t} e^{-\frac{\sqrt{-\Delta}}{2}(t - \tau)}\frac{\text{sin}\left(\frac{\sqrt{3}}{2}\sqrt{-\Delta}(t - \tau)\right)}{\frac{\sqrt{3}}{2}\sqrt{-\Delta}}F(\tau, \cdot)d\tau\right|\right|_{L^{q}_{t}([0, T]; L^{r}_{x}(\mathbb{R}^{n}))} \le C||F||_{L^{\tilde{q}'}_{t}([0, T]; L^{\tilde{r}'}_{x}(\mathbb{R}^{n}))}.
\end{equation*}
\end{lemma}
For the inhomogeneous linear viscous wave equation with zero initial data, notice that since 
\begin{equation*}
u(t, \cdot) := \int_{0}^{t} e^{-\frac{\sqrt{-\Delta}}{2}(t - \tau)}\frac{\text{sin}\left(\frac{\sqrt{3}}{2}\sqrt{-\Delta}(t - \tau)\right)}{\frac{\sqrt{3}}{2}\sqrt{-\Delta}}F(\tau, \cdot) d\tau,
\end{equation*}
this result implies
\begin{equation*}\label{inhestpart1}
||u||_{L_{t}^{q}([0, T]; L_{x}^{r}(\mathbb{R}^{n}))} \le C_{q, \tilde{q}, r, \tilde{r}}||F||_{L_{t}^{\tilde{q}'}([0, T]; L_{x}^{\tilde{r}'}(\mathbb{R}^{n}))} .
\end{equation*}

\begin{remark} In contrast with the Strichartz estimates for the wave equation and the Schr\"{o}dinger equation, the result in Lemma~\ref{inhStrlemma}
does not require any admissibility condition on $(q, r)$ and $(\tilde{q}, \tilde{r})$.
\end{remark}

\begin{proof}
Similar to the approach used in \cite{Z} for fractional heat equations,
the estimate in Lemma~\ref{inhStrlemma} is a consequence of the Hardy-Littlewood-Sobolev inequality.

More precisely, we first show the result in the case $T = \infty$. We begin by estimating:
\begin{multline*}\label{inhStrlemmaest1}
\left|\left|\int_{0}^{t} e^{-\frac{\sqrt{-\Delta}}{2}(t - \tau)}\frac{\text{sin}
\left(\frac{\sqrt{3}}{2}\sqrt{-\Delta}(t - \tau)\right)}
{\frac{\sqrt{3}}{2}\sqrt{-\Delta}}F(\tau, \cdot)d\tau
\right|\right|_{L^{q}_{t}([0, \infty); L^{r}_{x}(\mathbb{R}^{n}))} 
\\
\le 
\left\|
\int_{0}^{t}
\left\|
e^{-\frac{\sqrt{-\Delta}}{2}(t - \tau)}
\frac{\text{sin}\left(\frac{\sqrt{3}}{2}\sqrt{-\Delta}(t - \tau)\right)}{\frac{\sqrt{3}}{2}\sqrt{-\Delta}}
F(\tau, \cdot)
\right\|_{L^{r}_{x}(\mathbb{R}^{n})} d\tau
\right\|_{L^{q}_{t}[0, \infty)}.
\end{multline*}
From Lemma~\ref{semigroupineq}, we have that for $0 \le \tau < t$, the integrand can be bounded as:
\begin{equation*}
\left|\left|e^{-\frac{\sqrt{-\Delta}}{2}(t - \tau)}\frac{\text{sin}\left(\frac{\sqrt{3}}{2}\sqrt{-\Delta}(t - \tau)\right)}{\frac{\sqrt{3}}{2}\sqrt{-\Delta}}F(\tau, \cdot)\right|\right|_{ L^{r}_{x}(\mathbb{R}^{n}))} \le C_{r, \tilde{r}}|t - \tau|^{1 - n\left(\frac{1}{\tilde{r}'} - \frac{1}{r}\right)}||F(\tau, \cdot)||_{L^{\tilde{r}'}_{x}(\mathbb{R}^{n})},
\end{equation*}
which implies:
\begin{multline}\label{inhStrlemmaest1}
\left|\left|\int_{0}^{t} e^{-\frac{\sqrt{-\Delta}}{2}(t - \tau)}\frac{\text{sin}\left(\frac{\sqrt{3}}{2}\sqrt{-\Delta}(t - \tau)\right)}{\frac{\sqrt{3}}{2}\sqrt{-\Delta}}F(\tau, \cdot)d\tau\right|\right|_{L^{q}_{t}([0, \infty); L^{r}_{x}(\mathbb{R}^{n}))} \\
\le C_{r, \tilde{r}} \left|\left|\int_{0}^{t}|t - \tau|^{1 - n\left(\frac{1}{\tilde{r}'} - \frac{1}{r}\right)}||F(\tau, \cdot)||_{L^{\tilde{r}'}_{x}(\mathbb{R}^{n})} d\tau\right|\right|_{L^{q}_{t}[0, \infty)}.
\end{multline}
We would like to apply the Hardy-Littlewood-Sobolev inequality
in one dimension to the convolution in time above, which states that for all $0 < \alpha < 1$, and for $s$ and $\tilde{s}$, $1 < \tilde{s} < s < \infty$, such that
\begin{equation*}
1 + \frac{1}{s} = \frac{1}{\tilde{s}} + \alpha,
\end{equation*}
the following holds:
\begin{equation*}
\left|\left| (|x|^{-\alpha} * f )\right|\right|_{L^{s}(\mathbb{R})} \le C_{s, \tilde{s}}||f||_{L^{\tilde{s}}(\mathbb{R})}.
\end{equation*}
We plan to apply this estimate in \eqref{inhStrlemmaest1} by letting
\begin{equation*}
\alpha = -1 + n\left(\frac{1}{\tilde{r}'} - \frac{1}{r}\right) = 1 - \frac{1}{\tilde{q}'} + \frac{1}{q},
\end{equation*}
where indeed, $0 < \alpha < 1$ since we assumed that $1 < \tilde{q}' < q < \infty$, and the second equality holds by the gap condition \eqref{gap2}.
We can then let $\tilde{s} = \tilde{q}'$ and $s = q$. 
To apply the Hardy-Littlewood-Sobolev inequality, which holds on $\R$, we extend the function $F$ from $t \in [0,\infty)$ to $t \in \R$:
\begin{equation*}
\tilde{F}(t, x) = F(t, x) \text{ for } t \ge 0 \ {\rm and} \  \tilde{F}(t, x) = 0 \text{ for } t < 0,
\end{equation*}
and rewrite estimate \eqref{inhStrlemmaest1} in terms of $\tilde{F}$:
\begin{multline*}\label{inhStrlemmaest1}
\left|\left|\int_{0}^{t} e^{-\frac{\sqrt{-\Delta}}{2}(t - \tau)}\frac{\text{sin}\left(\frac{\sqrt{3}}{2}\sqrt{-\Delta}(t - \tau)\right)}{\frac{\sqrt{3}}{2}\sqrt{-\Delta}}F(\tau, \cdot)d\tau\right|\right|_{L^{q}_{t}([0, \infty); L^{r}_{x}(\mathbb{R}^{n}))} \\
\le 
C_{r, \tilde{r}} \left|\left|\int_{\mathbb{R}}|t - \tau|^{1 - n\left(\frac{1}{\tilde{r}'} - \frac{1}{r}\right)}||\tilde{F}(\tau, \cdot)||_{L^{\tilde{r}'}(\mathbb{R}^{n})} d\tau\right|\right|_{L^{q}_{t}(\mathbb{R})}.
\end{multline*}
By applying the Hardy-Littlewood-Sobolev inequality we get:
\begin{multline*}
\left|\left|\int_{0}^{t} e^{-\frac{\sqrt{-\Delta}}{2}(t - \tau)}\frac{\text{sin}\left(\frac{\sqrt{3}}{2}\sqrt{-\Delta}(t - \tau)\right)}{\frac{\sqrt{3}}{2}\sqrt{-\Delta}}F(\tau, \cdot)d\tau\right|\right|_{L^{q}_{t}([0, \infty); L^{r}_{x}(\mathbb{R}^{n}))} \\
\le C_{r, \tilde{r}}C_{q, \tilde{q}}||\tilde{F}||_{L_{t}^{\tilde{q}'}(\mathbb{R}; L^{\tilde{r}'}_{x}(\mathbb{R}^{n}))} = C_{q, \tilde{q}, r, \tilde{r}} ||F||_{L_{t}^{\tilde{q}'}([0, \infty); L^{\tilde{r}'}_{x}(\mathbb{R}^{n}))},
\end{multline*}
which is what we wanted to prove.
\medskip

The case for $0 < T < \infty$ follows from the case $T = \infty$ by considering
\begin{equation*}
\tilde{F}_{T}(t, x) = F(t, x) \text{ for } 0 \le t \le T \ {\rm and} \  \tilde{F}_{T}(t, x) = 0 \text{ otherwise}.
\end{equation*}
%%%%%%%%%%
\if 1 = 0
Then, note that for all $0 \le t \le T$,
\begin{equation*}
\int_{0}^{t} e^{-\frac{\sqrt{-\Delta}}{2}(t - \tau)}\frac{\text{sin}\left(\frac{\sqrt{3}}{2}\sqrt{-\Delta}(t - \tau)\right)}{\frac{\sqrt{3}}{2}\sqrt{-\Delta}}F(\tau, \cdot)d\tau = \int_{0}^{t} e^{-\frac{\sqrt{-\Delta}}{2}(t - \tau)}\frac{\text{sin}\left(\frac{\sqrt{3}}{2}\sqrt{-\Delta}(t - \tau)\right)}{\frac{\sqrt{3}}{2}\sqrt{-\Delta}}\tilde{F}_{T}(\tau, \cdot)d\tau
\end{equation*}
Therefore, 
\begin{multline*}
\left|\left|\int_{0}^{t} e^{-\frac{\sqrt{-\Delta}}{2}(t - \tau)}\frac{\text{sin}\left(\frac{\sqrt{3}}{2}\sqrt{-\Delta}(t - \tau)\right)}{\frac{\sqrt{3}}{2}\sqrt{-\Delta}}F(\tau, \cdot)d\tau\right|\right|_{L^{p}_{t}([0, T]; L^{q}_{x}(\mathbb{R}^{n}))} \\ 
= \left|\left|\int_{0}^{t} e^{-\frac{\sqrt{-\Delta}}{2}(t - \tau)}\frac{\text{sin}\left(\frac{\sqrt{3}}{2}\sqrt{-\Delta}(t - \tau)\right)}{\frac{\sqrt{3}}{2}\sqrt{-\Delta}}\tilde{F}_{T}(\tau, \cdot)d\tau\right|\right|_{L^{p}_{t}([0, T]; L^{q}_{x}(\mathbb{R}^{n}))} \\
\le \left|\left|\int_{0}^{t} e^{-\frac{\sqrt{-\Delta}}{2}(t - \tau)}\frac{\text{sin}\left(\frac{\sqrt{3}}{2}\sqrt{-\Delta}(t - \tau)\right)}{\frac{\sqrt{3}}{2}\sqrt{-\Delta}}\tilde{F}_{T}(\tau, \cdot)d\tau\right|\right|_{L^{p}_{t}([0,\infty); L^{q}_{x}(\mathbb{R}^{n}))} \\
\le C_{p, q, \tilde{p}, \tilde{q}} ||\tilde{F}_{T}||_{L_{t}^{\tilde{p}'}([0, \infty); L^{\tilde{q}'}_{x}(\mathbb{R}^{n}))} = C_{p, q, \tilde{p}, \tilde{q}} ||F||_{L_{t}^{\tilde{p}'}([0, T]; L^{\tilde{q}'}_{x}(\mathbb{R}^{n}))}
\end{multline*}
where $C_{p, q, \tilde{p}, \tilde{q}}$ is the same constant from the $T = \infty$ case, and this estimate holds for all $0 < T < \infty$. 
\fi 
%%%%%%%%%%%
This proves the lemma.

\end{proof}

So we have just shown that
$
||u||_{L_{t}^{q}([0, T]; L_{x}^{r}(\mathbb{R}^{n}))} \le C_{q, \tilde{q}, r, \tilde{r}}||F||_{L_{t}^{\tilde{q}'}([0, T]; L_{x}^{\tilde{r}'}(\mathbb{R}^{n}))}
$ for $u$ given by \eqref{u_inhomo}.
What remains to be shown to complete the inhomogeneous Strichartz estimate in Theorem \ref{inhStr} is given by the following lemma:

%\begin{equation*}
%||u(T, \cdot)||_{\dot{H}^{s}(\mathbb{R}^{n})} + ||\partial_{t}u(T, \cdot)||_{\dot{H}^{s - 1}(\mathbb{R}^{n})} \le C_{\tilde{q}, \tilde{r}}||F||_{L^{\tilde{q}'}_{t}([0, T]; %L^{\tilde{r}'}_{x}(\mathbb{R}^{n}))}.
%\end{equation*}
%Indeed, we have:

%Next, we consider the $\dot{H}^{s}(\mathbb{R}^{n})$ contribution to the inhomogeneous Strichartz estimates in the following lemma.

\begin{lemma}\label{inhStrSob}
Let $2 \le \tilde{q} \le \infty$, $2 \le \tilde{r} < \infty$ and $s$ satisfy
\begin{equation*}
\frac{n}{2} - s = \frac{1}{\tilde{q}'} + \frac{n}{\tilde{r}'} - 2.
\end{equation*}
Then, there exists a constant $C$ depending only on $\tilde{q}$ and $\tilde{r}$, such that for all $0 < T < \infty$ and for $u$ as in \eqref{u_inhomo}:
\begin{equation}\label{estStr2}
||u(T, \cdot)||_{\dot{H}^{s}(\mathbb{R}^{n})} + ||\partial_{t}u(T, \cdot)||_{\dot{H}^{s - 1}(\mathbb{R}^{n})} \le C_{\tilde{q}, \tilde{r}}||F||_{L^{\tilde{q}'}_{t}([0, T]; L^{\tilde{r}'}_{x}(\mathbb{R}^{n}))}.
\end{equation}
%where
%\begin{equation*}
%u := \int_{0}^{t}e^{-\frac{\sqrt{-\Delta}}{2}(t - \tau)} \frac{\text{sin}\left(\frac{\sqrt{3}}{2}\sqrt{-\Delta}(t - \tau)\right)}{\frac{\sqrt{3}}{2}\sqrt{-\Delta}}%F(\tau, \cdot) d\tau
%\end{equation*}
\end{lemma}

\begin{proof}
The proof uses similar arguments as in the proof of Theorem \ref{strongdecay}.
First, recall from \eqref{u_inhomo} that if $u$ is the solution to the linear viscous wave equation with zero initial data and source term $F$, then
\begin{equation*}
u(t,\cdot) := \int_{0}^{t}e^{-\frac{\sqrt{-\Delta}}{2}(t - \tau)} \frac{\text{sin}\left(\frac{\sqrt{3}}{2}\sqrt{-\Delta}(t - \tau)\right)}{\frac{\sqrt{3}}{2}\sqrt{-\Delta}}F(\tau, \cdot) d\tau.
\end{equation*}
We define the Littlewood-Paley decomposition contributions as:
\begin{equation*}
u_{j}(t, x) = \frac{1}{(2\pi)^{n}} \int_{\mathbb{R}^{n}} e^{ix \cdot \xi} \beta\left(\frac{|\xi|}{2^{j}}\right) \widehat{u}(t, \xi) d\xi,
\quad
F_{j}(t, x) = \frac{1}{(2\pi)^{n}} \int_{\mathbb{R}^{n}} e^{ix \cdot \xi} \beta\left(\frac{|\xi|}{2^{j}}\right) \widehat{F}(t, \xi) d\xi.
\end{equation*}
\vskip 0.2in
\noindent
{\bf Step 1.} We first claim that it is sufficient to show that estimate \eqref{estStr2} holds for each integer $j$ for all $F$ {{with $\widehat{F}(t, \xi)$ for $0 \le t \le T$ supported in an annulus 
$2^{j - 1} \le |\xi| \le 2^{j + 1}$, where the constant in the estimate is \textit{independent of $j$ and $0 < T < \infty$.}}}
\medskip

To see why this is so, we follow an argument similar to the one in the proof of Theorem \ref{strongdecay}. 
In particular, note that $u_{j}$ is the solution to the linear viscous wave equation with zero initial data and source term $F_{j}$, where 
$\widehat{F_{j}}$ is supported in the annulus $2^{j - 1} \le |\xi| \le 2^{j + 1}$ for all {$0 \le t \le T$}. So if we had established that the estimate holds for each $F$ with Fourier transform at all times supported in an annulus $2^{j - 1} \le |\xi| \le 2^{j + 1}$ for each $j$, for a constant independent of $j$, we would have for each $j$ that
\begin{equation*}
||u_{j}(T, \cdot)||_{\dot{H}^{s}(\mathbb{R}^{n})} + ||\partial_{t}u_{j}(T, \cdot)||_{\dot{H}^{s - 1}(\mathbb{R}^{n})} \le C_{\tilde{q}, \tilde{r}} ||F_{j}||_{L^{\tilde{q}'}_{t}([0,T]; L^{\tilde{r}'}_{x}(\mathbb{R}^{n}))},
\end{equation*}
and thus
\begin{equation*}
||u_{j}(T, \cdot)||_{\dot{H}^{s}(\mathbb{R}^{n})}^{2} \le C_{\tilde{q},\tilde{r}} ||F_{j}||_{L^{\tilde{q}'}_{t}([0,T]; L^{\tilde{r}'}_{x}(\mathbb{R}^{n}))}^{2}
\ {\rm and} \ 
||\partial_{t}u_{j}(T, \cdot)||_{\dot{H}^{s - 1}(\mathbb{R}^{n})}^{2} \le C_{\tilde{q},\tilde{r}} ||F_{j}||_{L^{\tilde{q}'}_{t}([0,T]; L^{\tilde{r}'}_{x}(\mathbb{R}^{n}))}^{2},   
\end{equation*}
for all $0 < T < \infty$ and for any function $F$ (since $F_{j}$ now has Fourier transform at all times with support in $2^{j - 1} \le |\xi| \le 2^{j + 1}$). To get the estimate in Lemma~\ref{inhStrSob}, note that since 
$
\sum_{j = -\infty}^{\infty} \beta^{2}\left(\frac{|\xi|}{2^{j}}\right) \ge c
$
for some constant $c > 0$, and since $\widehat{\partial_{t} u_{j}}(t, \xi) = \beta(|\xi|/2^{j})\widehat{\partial_{t}u}(t, \xi)$, we have
\begin{equation*}
||u(T, \cdot)||^{2}_{\dot{H}^{s}(\mathbb{R}^{n})} \le c^{-1}\sum_{j = -\infty}^{\infty} ||u_{j}(T, \cdot)||_{\dot{H}^{s}(\mathbb{R}^{n})}^{2}
\ {\rm and} \ ||\partial_{t}u(T, \cdot)||_{\dot{H}^{s - 1}(\mathbb{R}^{n})}^{2} \le c^{-1} \sum_{j = -\infty}^{\infty} ||\partial_{t}u_{j}(T, \cdot)||^{2}_{\dot{H}^{s - 1}(\mathbb{R}^{n})}.
\end{equation*}
In addition, since $\tilde{r} \ne \infty$, we have $\tilde{r}' \ne 1$, and from the assumptions on $\tilde{q}$ and $\tilde{r}$
from the lemma, we see that $1 \le \tilde{q}' \le 2$, $1 < \tilde{r}' \le 2$. So by estimate \eqref{LP2} in Lemma \ref{LPlemma}, 
\begin{equation*}
\sum_{j = -\infty}^{\infty} ||F_{j}||^{2}_{L_{t}^{\tilde{q}'}([0, T]; L^{\tilde{r}'}_{x}(\mathbb{R}^{n}))} \le C||F||^{2}_{L^{\tilde{q}'}_{t}([0, T]; L^{\tilde{r}'}_{x}(\mathbb{R}^{n}))}.
\end{equation*}
Therefore, for general $F$ and associated $u$,
\begin{align*}
\left(||u(T, \cdot)||_{\dot{H}^{s}(\mathbb{R}^{n})} \right.
& \left. + ||\partial_{t}u(T, \cdot)||_{\dot{H}^{s - 1}(\mathbb{R}^{n})}\right)^{2} \le 2c^{-1}\sum_{j = -\infty}^{\infty}\left(||u_{j}(T, \cdot)||^{2}_{\dot{H}^{s}(\mathbb{R}^{n})} + ||\partial_{t}u_{j}(T, \cdot)||^{2}_{\dot{H}^{s - 1}(\mathbb{R}^{n})}\right) \\
&\le 4c^{-1}C_{\tilde{q}, \tilde{r}} \sum_{j = -\infty}^{\infty} ||F_{j}||^{2}_{L^{\tilde{q}'}_{t}([0, T]; L^{\tilde{r}'}_{x}(\mathbb{R}^{n}))} \le C_{\tilde{q}, \tilde{r}}' ||F||^{2}_{L^{\tilde{q}'}_{t}([0, T]; L^{\tilde{r}'}_{x}(\mathbb{R}^{n}))}.
\end{align*}
Taking square roots gives the desired result. 

\vskip 0.1in
\noindent
{\bf Step 2.} Because of the gap condition \eqref{gap2},  we claim that it suffices to show estimate \eqref{estStr2} 
for all $F$ such that $\hat{F}$ is supported in 
$2^{j - 1} \le |\xi| \le 2^{j + 1}$, for only $j = 0$, i.e., for $\hat{F}$ supported in $1/2 \le | \xi | \le 2$ { for all $0 \le t \le T$ with $0 < T < \infty$ arbitrary}.

{{As in the proof of Theorem \ref{strongdecay}, we can use the scaling that preserves solutions of the linear viscous wave equation, that maps $F$ and $u$ into
$\lambda^{2}F(\lambda t, \lambda x), u(\lambda t, \lambda x)$. Under this scaling,}}
\begin{equation*}
||u(\lambda T, \lambda x)||_{\dot{H}^{s}(\mathbb{R}^{n})} = \lambda^{-\frac{n}{2} + s}||u(\lambda T, x)||_{\dot{H}^{s}(\mathbb{R}^{n})},
\end{equation*}
\begin{equation*}
||\lambda \partial_{t}u(\lambda T, \lambda x)||_{\dot{H}^{s - 1}(\mathbb{R}^{n})} = \lambda^{-\frac{n}{2} + s} ||\partial_{t}u(\lambda T, x)||_{\dot{H}^{s - 1}(\mathbb{R}^{n})},
\end{equation*}
\begin{equation*}
||\lambda^{2} F(\lambda t, \lambda x)||_{L^{\tilde{q}'}_{t}([0, T]; L^{\tilde{r}'}_{x}(\mathbb{R}^{n}))} = \lambda^{2 - \frac{n}{\tilde{r}'} - \frac{1}{\tilde{q}'}}||F(t, x)||_{L^{\tilde{q}'}_{t}([0, \lambda T]; L^{\tilde{r}'}_{x}(\mathbb{R}^{n}))}.
\end{equation*}
Because of the gap condition \eqref{gap2},
%\begin{equation*}
%\frac{n}{2} - s = \frac{1}{\tilde{p}'} + \frac{n}{\tilde{q}'} - 2,
%\end{equation*}
once we prove estimate \eqref{estStr2} in the case where $\widehat{F}(t, \xi)$ is supported in $1/2 \le |\xi| \le 2$, we can use the scalings above to get that the inequality holds whenever $\widehat{F}(t, \xi)$ ({{for $0 \le t \le T$}}) is supported in any annulus $2^{j - 1} \le |\xi| \le 2^{j + 1}$ for any integer $j$, with the same constant $C_{\tilde{q}, \tilde{r}}$.

\vskip 0.1in
\noindent
{\bf Step 3.} Thus, we must show that there exists a constant $C_{\tilde{q}, \tilde{r}} > 0$, independent of $0 < T < \infty$,
such that whenever $\widehat{F}(t, \xi)$ is supported in $1/2 \le |\xi| \le 2$ for all $0 \le t \le T$, the following holds:
\begin{equation*}
||u(T, \cdot)||_{\dot{H}^{s}(\mathbb{R}^{n})} + ||\partial_{t}u(T, \cdot)||_{\dot{H}^{s - 1}(\mathbb{R}^{n})} \le C_{\tilde{q}, \tilde{r}}||F||_{L^{\tilde{q}'}_{t}([0, T]; L^{\tilde{r}'}_{x}(\mathbb{R}^{n}))}.
\end{equation*}

We begin by estimating $||u(T, \cdot)||_{\dot{H}^{s}(\mathbb{R}^{n})}$:
\begin{align*}
||u(T, \cdot)||_{\dot{H}^{s}(\mathbb{R}^{n})} 
&\le \int_{0}^{T} \left|\left|e^{-\frac{\sqrt{-\Delta}}{2}(T - \tau)}\frac{\text{sin}\left(\frac{\sqrt{3}}{2}\sqrt{-\Delta}(T - \tau)\right)}{\frac{\sqrt{3}}{2}\sqrt{-\Delta}}F(\tau, \cdot)\right|\right|_{\dot{H}^{s}(\mathbb{R}^{n})}d\tau \\
&= C\int_{0}^{T} \left(\int_{\mathbb{R}^{n}} |\xi|^{2s} \left|e^{-\frac{|\xi|}{2}(T - \tau)} \frac{\text{sin}\left(\frac{\sqrt{3}}{2}|\xi|(T - \tau)\right)}{\frac{\sqrt{3}}{2}|\xi|}\widehat{F}(\tau, \xi)\right|^{2} d\xi \right)^{1/2}d\tau.
\end{align*}
To further estimate the right hand-side of this inequality, we first notice that
on the support of $\widehat{F}(t, \xi)$, which is in $1/2 \le |\xi| \le 2$, we have $|\xi|^{2s} \le C_{s}$, where the constant $C_s$ can be 
expressed to depend only on $\tilde{q}, \tilde{r}$, i.e., $C_s = C_{\tilde{q}, \tilde{r}}$  because of the gap condition. 
Furthermore, on the support of $\widehat{F}(t, \xi)$ and for $0 \le \tau \le T$ we have:
\begin{equation*}
\left|e^{-\frac{|\xi|}{2}(T - \tau)} \frac{\text{sin}\left(\frac{\sqrt{3}}{2}|\xi|(T - \tau)\right)}{\frac{\sqrt{3}}{2}|\xi|}\right| \le (T - \tau)e^{-\frac{|\xi|}{2}(T - \tau)} \le (T - \tau)e^{-\frac{1}{4}(T - \tau)}.
\end{equation*}
Therefore, we now have:
\begin{equation*}
||u(T, \cdot)||_{\dot{H}^{s}(\mathbb{R}^{n})} \le C_{\tilde{q}, \tilde{r}} \int_{0}^{T} \left( \int_{\mathbb{R}^{n}} |(T - \tau)e^{-\frac{1}{4}(T - \tau)}\chi_{1/2 \le |\xi| \le 2}(\xi)|^{2}|\widehat{F}(\tau, \xi)|^{2} d\xi\right)^{1/2} d\tau.
\end{equation*}
By Holder's inequality with  $\frac{\tilde{r}}{\tilde{r} - 2} \ge 1$ and $\frac{\tilde{r}}{2} \ge 1$, we get:
\begin{multline*}
||u(T, \cdot)||_{\dot{H}^{s}(\mathbb{R}^{n})} 
\le C_{\tilde{q}, \tilde{r}} \int_{0}^{T} 
\left[
\left(\int_{1/2 \le |\xi| \le 2}\left|(T - \tau)e^{-\frac{1}{4}(T - \tau)}\right|^{\frac{2\tilde{r}}{\tilde{r} - 2}}d\xi\right)^{\frac{\tilde{r} - 2}{2\tilde{r}}} ||\widehat{F}(\tau, \cdot)||_{L^{\tilde{r}}_{\xi}(\mathbb{R}^{n})} 
\right]d\tau 
\\
\le C_{\tilde{q}, \tilde{r}} \int_{0}^{T} (T - \tau)e^{-\frac{1}{4}(T - \tau)} ||\widehat{F}(\tau, \cdot)||_{L^{\tilde{r}}_{\xi}(\mathbb{R}^{n})} d\tau \le C_{\tilde{q}, \tilde{r}} \int_{0}^{T} (T - \tau)e^{-\frac{1}{4}(T - \tau)} ||F(\tau, \cdot)||_{L^{\tilde{r}'}_{x}(\mathbb{R}^{n})} d\tau,
\end{multline*}
where we used that $1/2 \le |\xi| \le 2$ in $\mathbb{R}^{n}$ has finite volume and the fact that the Fourier transform is a bounded operator from $L^{\tilde{r}'}$ to $L^{\tilde{r}}$, since $\tilde{r} \ge 2$. 
Using Holder's inequality one more time with $\tilde{q}$ and $\tilde{q}'$, we obtain:
\begin{align*}
||u(T, \cdot)||_{\dot{H}^{s}(\mathbb{R}^{n})} 
&\le C_{\tilde{q}, \tilde{r}} 
\left(\int_{0}^{T} \left|(T - \tau)e^{-\frac{1}{4}(T - \tau)}\right|^{\tilde{q}}d\tau\right)^{1/\tilde{q}}||F||_{L^{\tilde{q}'}_{t}([0, T]; L^{\tilde{r}'}_{x}(\mathbb{R}^{n}))} 
\\
&\le C_{\tilde{q}, \tilde{r}} \left(\int_{0}^{T} \left|\tau e^{-\frac{1}{4}\tau}\right|^{\tilde{q}}d\tau\right)^{1/\tilde{q}}||F||_{L^{\tilde{q}'}_{t}([0, T]; L^{\tilde{r}'}_{x}(\mathbb{R}^{n}))} 
\\
&\le C_{\tilde{q}, \tilde{r}} \left(\int_{0}^{\infty} \left|\tau e^{-\frac{1}{4}\tau}\right|^{\tilde{q}}d\tau\right)^{1/\tilde{q}}||F||_{L^{\tilde{q}'}_{t}([0, T]; L^{\tilde{r}'}_{x}(\mathbb{R}^{n}))} \le C_{\tilde{q}, \tilde{r}}||F||_{L^{\tilde{q}'}_{t}([0, T]; L^{\tilde{r}'}_{x}(\mathbb{R}^{n}))}.
\end{align*}
%where in the last inequality, we use that the integral is convergent.

Next, we estimate $||\partial_{t}u(T, \xi)||_{\dot{H}^{s - 1}(\mathbb{R}^{n})}$. First note that 
\begin{align*}
\widehat{\partial_{t}u}(t, \xi)& = \partial_{t}(\widehat{u}(t, \xi)) \\ 
&= \int_{0}^{t} e^{-\frac{|\xi|}{2}(t - \tau)}\left(\text{cos}\left(\frac{\sqrt{3}}{2}|\xi|(t - \tau)\right) - \frac{1}{\sqrt{3}}\text{sin}\left(\frac{\sqrt{3}}{2}|\xi|(t - \tau)\right)\right)\widehat{F}(\tau, \xi) d\tau.
\end{align*}
By the support properties of $\widehat{F}(\tau, \xi)$ and boundedness of sines and cosines, we get:
\begin{align*}
||\partial_{t}u(T, \xi)||_{\dot{H}^{s - 1}(\mathbb{R}^{n})} 
&\le C\int_{0}^{T} \left(\int_{\mathbb{R}^{n}} |\xi|^{2s - 2} \left|e^{-\frac{|\xi|}{2}(T - \tau)}\widehat{F}(\tau, \xi)\right|^{2} d\xi\right)^{1/2} d\tau 
\\
&= C_{\tilde{q}, \tilde{r}} \int_{0}^{T} \left(\int_{\mathbb{R}^{n}} \left|e^{-\frac{1}{4}(T - \tau)}\chi_{1/2 \le |\xi| \le 2}(\xi)\widehat{F}(\tau, \xi)\right|^{2}d\xi\right)^{1/2} d\tau.
\end{align*}
After applying Holder's inequality with $\frac{\tilde{r}}{\tilde{r} - 2}$ and $\frac{\tilde{r}}{2}$, we get:
\begin{align*}
||\partial_{t}u(T, \xi)||_{\dot{H}^{s - 1}(\mathbb{R}^{n})} 
&\le C_{\tilde{q}, \tilde{r}} \int_{0}^{T} \left(\int_{1/2 \le |\xi| \le 2} \left|e^{-\frac{1}{4}(T - \tau)}\right|^{\frac{2\tilde{r}}{\tilde{r} - 2}}d\xi\right)^{\frac{\tilde{r} - 2}{2\tilde{r}}} ||\widehat{F}(\tau, \cdot)||_{L^{\tilde{r}}_{\xi}(\mathbb{R}^{n})} d\tau 
\\
&\le C_{\tilde{q}, \tilde{r}} \int_{0}^{T} e^{-\frac{1}{4}(T - \tau)} ||F(\tau, \cdot)||_{L^{\tilde{r}'}_{x}(\mathbb{R}^{n})}d\tau 
\\
%&\le C_{\tilde{p}, \tilde{q}} \left(\int_{0}^{T} \left|e^{-\frac{1}{4}(T - \tau)}\right|^{\tilde{p}}d\tau\right)^{1/\tilde{p}} ||F||_{L^{\tilde{p}'}_{t}([0, T]; &L^{\tilde{q}'}_{x}(\mathbb{R}^{n}))}
% \\
&= C_{\tilde{q}, \tilde{r}} \left(\int_{0}^{T} \left|e^{-\frac{1}{4}\tau}\right|^{\tilde{q}}d\tau\right)^{1/\tilde{q}} ||F||_{L^{\tilde{q}'}_{t}([0, T]; L^{\tilde{r}'}_{x}(\mathbb{R}^{n}))} 
 \le C_{\tilde{q}, \tilde{r}}||F||_{L^{\tilde{q}'}_{t}([0, T]; L^{\tilde{r}'}_{x}(\mathbb{R}^{n}))}.
\end{align*}
This completes the proof of the lemma. 
\end{proof}

Lemmas \ref{inhStrlemma} and \ref{inhStrSob} now imply the proof of 
Theorem \ref{inhStr}, which completes the proof of the inhomogeneous Strichartz estimates for the linear viscous wave equation.

\section{$C^0([0,T],H^s(\R^n))$ estimates for the linear viscous wave equation}\label{local_estimates}
In this section we present the $H^s$ estimates on the solution of the linear viscous wave equation, which will be
needed in probabilistic randomization, presented in  Sec.~\ref{random}. 
In contrast with the Strichartz estimates, {{where scale invariance plays an important role since the estimates hold
globally in time}}, 
the $H^{s}(\mathbb{R}^{n})$ norms are not invariant under the natural scaling map, and so the estimates will hold 
only on a finite time interval $[0,T]$.
Since the focus of Sec.~\ref{random} will be on the case $n =2$, we present the estimates only for $n = 2$.

\begin{proposition}\label{Hslemma}
Let $n = 2$. Let $2 \le \tilde{q} \le \infty$, $2 \le \tilde{r} < \infty$, and $s \ge 0$ satisfy the gap condition
\begin{equation*}
\frac{n}{2} - s = \frac{1}{\tilde{q}'} + \frac{n}{\tilde{r}'} - 2.    
\end{equation*}
Then, for all $0 < T \le 1$, there exists a constant $C>0$ independent of $T$, depending only on $\tilde{q}$ and $\tilde{r}$, such that
\begin{equation}\label{Sobpropresult}
\left|\left|\int_{0}^{t} e^{-\frac{\sqrt{-\Delta}}{2}(t - \tau)} \frac{\text{sin}\left(\frac{\sqrt{3}}{2}\sqrt{-\Delta}(t - \tau)\right)}{\frac{\sqrt{3}}{2}\sqrt{-\Delta}}F(\tau, \cdot)d\tau\right|\right|_{C^{0}([0, T]; H^{s}(\mathbb{R}^{n}))} \le C_{\tilde{q}, \tilde{r}} ||F||_{L^{\tilde{q}'}_{t}([0, T]; L^{\tilde{r}'}_{x}(\mathbb{R}^{n}))}.
\end{equation}
\end{proposition}

\begin{proof}
We start by obtaining an $L^\infty([0,T], L^2(\R^n))$-estimate of the left hand-side in \eqref{Sobpropresult} and then use
%Let us first prove  \eqref{Sobpropresult}, but with $C^{0}([0, T]; H^{s}(\mathbb{R}^{n}))$ replaced by $L^{\infty}([0, T]; H^{s}(\mathbb{R}^{n}))$. Our plan is to do this, and then use the $L^{\infty}([0, T]; H^{s}(\mathbb{R}^{n}))$ inequality to prove that the function $t \to \int_{0}^{t} e^{-\frac{\sqrt{-\Delta}}{2}(t - \tau)} \frac{\text{sin}\left(\frac{\sqrt{3}}{2}\sqrt{-\Delta}(t - \tau)\right)}{\frac{\sqrt{3}}{2}\sqrt{-\Delta}}F(\tau, \cdot)d\tau$ is actually continuous with respect to the $H^{s}(\mathbb{R}^{n})$ norm on $[0, T]$. Note that for any $f$,
\begin{equation*}
||f||_{H^{s}} \le C_{s}(||f||_{L^{2}} + ||f||_{\dot{H}^{s}})
\end{equation*}
{{along with the result of Theorem \ref{inhStr}}} to obtain an $L^\infty([0,T], H^s(\R^n))$ estimate. 
This estimate will then be improved by showing continuity in time, to get the desired
$C^0([0,T], H^s(\R^n))$ estimate.

We first obtain the spatial $H^s$ estimate. 
Since we already have the homogeneous $\dot{H}^s$-norm estimates 
from the inhomogeneous Strichartz estimates in Theorem \ref{inhStr},
it suffices to obtain the corresponding $L^2$-estimate.
For $0 < T \le 1$ we have, by H\"{o}lder's inequality:
\begin{align*}
&\qquad \qquad\qquad\qquad \left\|\int_{0}^{T} \right.
 \left.  e^{-\frac{\sqrt{-\Delta}}{2}(T - \tau)} \frac{\text{sin}\left(\frac{\sqrt{3}}{2}\sqrt{-\Delta}(T - \tau)\right)}{\frac{\sqrt{3}}{2}\sqrt{-\Delta}}F(\tau, \cdot)d\tau \right\|_{L^{2}(\mathbb{R}^{n})} 
 \\
&\le \int_{0}^{T} \left\|
e^{-\frac{|\xi|}{2}(T - \tau)}\frac{\text{sin}\left(\frac{\sqrt{3}}{2}|\xi|(T - \tau)\right)}{\frac{\sqrt{3}}{2}|\xi|}\widehat{F}(\tau, \xi)
\right\|_{L^{2}_{\xi}(\mathbb{R}^{n})}d\tau 
\le \int_{0}^{T} \|(T - \tau)e^{-\frac{|\xi|}{2}(T - \tau)}\|_{L^{\frac{2\tilde{r}}{\tilde{r} - 2}}_{\xi}} 
\|\widehat{F}(\tau, \cdot)\|_{L^{\tilde{r}}_{\xi}} d\tau 
\\
&\le C_{\tilde{r}} \int_{0}^{T} \|(T - \tau)e^{-\frac{|\xi|}{2}(T - \tau)}\|_{L^{\frac{2\tilde{r}}{\tilde{r} - 2}}_{\xi}(\mathbb{R}^{n})} \|F(\tau, \cdot)||_{L^{\tilde{r}'}_{x}(\mathbb{R}^{n})} d\tau 
= C_{\tilde{r}} \int_{0}^{T} (T - \tau)^{1 - n \cdot \frac{\tilde{r} - 2}{2\tilde{r}}}||F(\tau, \cdot)||_{L^{\tilde{r}'}_{x}(\mathbb{R}^{n})} d\tau.
\end{align*}
Note that since $n = 2, 2 \le \tilde{r} < \infty$, we have $1 - n\cdot \frac{\tilde{r} - 2}{2\tilde{r}} > 0$. Then, using H\"{o}lder's inequality once more, we can continue estimating the right hand-side as
{{}
{\begin{equation}\label{C0estn}
\le C_{\tilde{r}} \left(\int_{0}^{T} (T - \tau)^{\tilde{q}\left(1 - n\cdot \frac{\tilde{r} - 2}{2\tilde{r}}\right)}d\tau\right)^{1/\tilde{q}} ||F||_{L^{\tilde{q}'}_{t}([0, T]; L^{\tilde{r}'}_{x}(\mathbb{R}^{n}))} = C_{\tilde{q}, \tilde{r}} T^{\alpha_{\tilde{q}, \tilde{r}}}||F||_{L^{\tilde{q}'}_{t}([0, T]; L^{\tilde{r}'}_{x}(\mathbb{R}^{n}))},
\end{equation}}}
for some $\alpha_{\tilde{q}, \tilde{r}} > 0$ depending on $\tilde{q}$ and $\tilde{r}$. 
Therefore, we have shown that for all $0 < T \le 1${{}{,}} 
\begin{equation*}
\left|\left|\int_{0}^{T} e^{-\frac{\sqrt{-\Delta}}{2}(T - \tau)} \frac{\text{sin}\left(\frac{\sqrt{3}}{2}\sqrt{-\Delta}(T - \tau)\right)}{\frac{\sqrt{3}}{2}\sqrt{-\Delta}}F(\tau, \cdot)d\tau\right|\right|_{L^{2}(\mathbb{R}^{n})} \le C_{\tilde{q}, \tilde{r}} T^{\alpha_{\tilde{q}, \tilde{r}}}||F||_{L^{\tilde{q}'}_{t}([0, T]; L^{\tilde{r}'}_{x}(\mathbb{R}^{n}))},
\end{equation*}
where $C_{\tilde{q}, \tilde{r}}$ is independent of $0 < T \le 1$.
Now,  from the inhomogeneous Strichartz estimates in Theorem \ref{inhStr}, we have
\begin{equation*}
\left|\left|\int_{0}^{T} e^{-\frac{\sqrt{-\Delta}}{2}(T - \tau)} \frac{\text{sin}\left(\frac{\sqrt{3}}{2}\sqrt{-\Delta}(T - \tau)\right)}{\frac{\sqrt{3}}{2}\sqrt{-\Delta}}F(\tau, \cdot)d\tau\right|\right|_{\dot{H}^{s}(\mathbb{R}^{n})} \le C_{\tilde{q}, \tilde{r}}' ||F||_{L^{\tilde{q}'}_{t}([0, T]; L^{\tilde{r}'}_{x}(\mathbb{R}^{n}))},
\end{equation*}
for all $0 < T < \infty$ where $C_{\tilde{q}, \tilde{r}}'$ is independent of $0 < T < \infty$. 
Combining these two estimates we get that for all $0 < T \le 1$,
{{}\begin{equation}\label{linfineq}
\left|\left|\int_{0}^{t} e^{-\frac{\sqrt{-\Delta}}{2}(t - \tau)} \frac{\text{sin}\left(\frac{\sqrt{3}}{2}\sqrt{-\Delta}(t - \tau)\right)}{\frac{\sqrt{3}}{2}\sqrt{-\Delta}}F(\tau, \cdot)d\tau\right|\right|_{L^{\infty}([0, T]; H^{s}(\mathbb{R}^{n}))} 
\le C_{\tilde{q}, \tilde{r}} ||F||_{L^{\tilde{q}'}_{t}([0, T]; L^{\tilde{r}'}_{x}(\mathbb{R}^{n}))},
\end{equation}
for some constant $C_{\tilde{q}, \tilde{r}}$ independent of $0 < T \le 1$. }

We need to show that the above estimate holds in terms of the ${C^0([0, T]; H^{s}(\mathbb{R}^{n}))}$ norm 
on the left hand-side of \eqref{linfineq}.
For this purpose, we first notice that estimate \eqref{linfineq} establishes the continuity at $t = 0$, in the $H^{s}(\mathbb{R}^{n})$ norm, of the map 
$$u^t (\cdot) : t \to u(t,\cdot) = \int_{0}^{t}  e^{-\frac{\sqrt{-\Delta}}{2}(t - \tau)} 
\frac{\text{sin}\left(\frac{\sqrt{3}}{2}\sqrt{-\Delta}(t - \tau)\right)}{\frac{\sqrt{3}}{2}\sqrt{-\Delta}}F(\tau, \cdot)d\tau.
$$ 
We next show that this map is, in fact, a continuous map from $[0, 1]$ to $H^{s}(\mathbb{R}^{n})$.
More precisely, fix an arbitrary $0 < t_{1} \le 1$.  
Given $\epsilon > 0$, we want to find $\delta > 0$ such that for all $\tilde{t} \in (t_{1} - \delta,t_{1} + \delta)$ 
(or  $\tilde{t} \in (t_{1} - \delta,t_{1})$ if $t_{1} = 1$), we have
$$
\| u(t_1,\cdot) - u(\tilde{t},\cdot) \|_{H^{s}(\mathbb{R}^{n})}  < \epsilon.
$$
Thus, we want to estimate:
\begin{align*}
\| u(t_1,\cdot) - u(\tilde{t},\cdot) \|_{H^{s}(\mathbb{R}^{n})} = 
&
\left\| 
  \int_{0}^{t_{1}} e^{-\frac{\sqrt{-\Delta}}{2}(t_{1} - \tau)} \frac{\text{sin}\left(\frac{\sqrt{3}}{2}\sqrt{-\Delta}(t_{1} - \tau)\right)}{\frac{\sqrt{3}}{2}\sqrt{-\Delta}}F(\tau, \cdot)d\tau \right.
\\
& - \left. \int_{0}^{\tilde{t}} e^{-\frac{\sqrt{-\Delta}}{2}(\tilde{t} - \tau)} \frac{\text{sin}\left(\frac{\sqrt{3}}{2}\sqrt{-\Delta}(\tilde{t} - \tau)\right)}{\frac{\sqrt{3}}{2}\sqrt{-\Delta}}F(\tau, \cdot)d\tau\right\|_{H^{s}(\mathbb{R}^{n})}.
  \end{align*}
 To do this, we introduce a $\delta' > 0$, which will be specified later, and  then find $\delta < \delta'/2$  sufficiently small so that
 $\tilde{t} \in (t_{1} - \delta, t_{1} + \delta)$ will always be greater than $t_{1} - \delta'$, and estimate the left hand side from above 
 via the following three integrals:
 
 \begin{align}\label{contbound}
& \| u(t_1,\cdot) - u(\tilde{t},\cdot) \|_{H^{s}(\mathbb{R}^{n})} 
\nonumber
 \\
&\le \left|\left|\int_{t_{1} - \delta'}^{t_{1}} e^{-\frac{\sqrt{-\Delta}}{2}(t_{1} - \tau)} \frac{\text{sin}\left(\frac{\sqrt{3}}{2}\sqrt{-\Delta}(t_{1} - \tau)\right)}{\frac{\sqrt{3}}{2}\sqrt{-\Delta}}F(\tau, \cdot)d\tau \right|\right|_{H^{s}(\mathbb{R}^{n})} 
\nonumber
\\ 
&+ \left|\left|\int_{0}^{t_{1} - \delta'} \left(e^{-\frac{\sqrt{-\Delta}}{2}(t_{1} - \tau)} \frac{\text{sin}\left(\frac{\sqrt{3}}{2}\sqrt{-\Delta}(t_{1} - \tau)\right)}{\frac{\sqrt{3}}{2}\sqrt{-\Delta}} - e^{-\frac{\sqrt{-\Delta}}{2}(\tilde{t} - \tau)} \frac{\text{sin}\left(\frac{\sqrt{3}}{2}\sqrt{-\Delta}(\tilde{t} - \tau)\right)}{\frac{\sqrt{3}}{2}\sqrt{-\Delta}}\right)F(\tau, \cdot)d\tau\right|\right|_{H^{s}(\mathbb{R}^{n})} 
\nonumber
\\
&+ \left|\left|\int_{t_{1} - \delta'}^{\tilde{t}} e^{-\frac{\sqrt{-\Delta}}{2}(\tilde{t} - \tau)} \frac{\text{sin}\left(\frac{\sqrt{3}}{2}\sqrt{-\Delta}(\tilde{t} - \tau)\right)}{\frac{\sqrt{3}}{2}\sqrt{-\Delta}}F(\tau, \cdot)d\tau\right|\right|_{H^{s}(\mathbb{R}^{n})} := I^{(1)}_{t_{1}, \tilde{t}} + I^{(2)}_{t_{1}, \tilde{t}} + I^{(3)}_{t_{1}, \tilde{t}}.
\end{align}

We now choose $\delta' > 0$, $\delta' < t_{1}/2$ (and $\delta' < 1 - t_{1}$ if $0 < t_{1} < 1$), sufficiently small so that on $[t_{1} - \delta', t_{1} + \delta']$,
\begin{equation*}
C_{\tilde{q}, \tilde{r}} ||F||_{L^{\tilde{q}'}_{t}([t_{1} - \delta', t_{1} + \delta'], L^{\tilde{r}'}_{x}(\mathbb{R}^{n}))} < \frac{\epsilon}{3},
\end{equation*}
where $C_{\tilde{q}, \tilde{r}}$ is the constant in \eqref{linfineq}. Then, using \eqref{linfineq}, the first and the third integral are immediately
estimated as:
\begin{equation}\label{contest1}
0 \le I^{(1)}_{t_{1}, \tilde{t}} < \frac{\epsilon}{3}, \ \ \ \ \ \ \ 0 \le I^{(3)}_{t_{1}, \tilde{t}} < \frac{\epsilon}{3}.
\end{equation}
What remains is to estimate $I^{(2)}_{t_{1}, \tilde{t}}$. 
To simplify notation, we will use
\begin{equation*}
\hat{K}_{t}(\xi) := e^{-\frac{|\xi| }{2}t}\frac{\text{sin}\left(\frac{\sqrt{3}}{2}|\xi| t\right)}{\frac{\sqrt{3}}{2}|\xi|}
\end{equation*}
to denote the Fourier transform of the kernel $K_t$ introduced in \eqref{Kt}.
By H\"{o}lder's inequality and Hausdorff-Young's inequality, we have that
\begin{align}\label{intineq1}
I^{(2)}_{t_{1}, \tilde{t}} &\le \int_{0}^{t_{1} - \delta'} ||(1 + |\xi|^{2})^{s/2}
(\hat{K}_{t_1-\tau}(\xi) - \hat{K}_{\tilde{t}-\tau}(\xi))\widehat{F}(\tau, \cdot)||_{L_{\xi}^{2}(\mathbb{R}^{n})} d\tau 
\nonumber
\\ 
&\le \int_{0}^{t_{1} - \delta'} ||(1 + |\xi|^{2})^{s/2}(\hat{K}_{t_1-\tau}(\xi) - \hat{K}_{\tilde{t}-\tau}(\xi))||_{L_{\xi}^{\frac{2\tilde{r}}{\tilde{r} - 2}}(\mathbb{R}^{n})}||\widehat{F}(\tau, \cdot)||_{L_{\xi}^{\tilde{r}}(\mathbb{R}^{n})} d\tau 
\nonumber
\\
&\le C \int_{0}^{t_{1} - \delta'} ||(1 + |\xi|^{2})^{s/2}(\hat{K}_{t_1-\tau}(\xi) - \hat{K}_{\tilde{t}-\tau}(\xi))||_{L_{\xi}^{\frac{2\tilde{r}}{\tilde{r} - 2}}(\mathbb{R}^{n})}||F(\tau, \cdot)||_{L_{x}^{\tilde{r}'}(\mathbb{R}^{n})} d\tau 
\nonumber
\\
&\le C ||(1 + |\xi|^{2})^{s/2}(\hat{K}_{t_1-\tau}(\xi) - \hat{K}_{\tilde{t}-\tau}(\xi))||_{L^{\tilde{q}}_{\tau}([0, t_{1} - \delta']; L_{\xi}^{\frac{2\tilde{r}}{\tilde{r} - 2}}(\mathbb{R}^{n}))} ||F||_{L^{\tilde{q}'}_{t}([0, 1]; L^{\tilde{r}'}_{x}(\mathbb{R}^{n}))}.
\end{align}
{{Since $||F||_{L^{\tilde{q}'}_{t}([0, 1]; L^{\tilde{r}'}_{x}(\mathbb{R}^{n}))}$ is a fixed constant for a given $F$}},
we just need to show that the factor multiplying this term is small, i.e., it can be bounded in terms of $\epsilon$.
The term we need to bound, in expanded form, reads:
\begin{multline}\label{intineq2}
C ||(1 + |\xi|^{2})^{s/2}(\hat{K}_{t_1-\tau}(\xi) - \hat{K}_{\tilde{t}-\tau}(\xi))||_{L^{\tilde{q}}_{\tau}([0, t_{1} - \delta']; L_{\xi}^{\frac{2\tilde{r}}{\tilde{r} - 2}}(\mathbb{R}^{n}))} \\
= C \left(\int_{0}^{t_{1} - \delta'} \left(\int_{\mathbb{R}^{n}} \left|(1 + |\xi|^{2})^{s/2}(\hat{K}_{t_1-\tau}(\xi) - \hat{K}_{\tilde{t}-\tau}(\xi))\right|^{\frac{2\tilde{r}}{\tilde{r} - 2}} d\xi\right)^{\tilde{q}\cdot \frac{\tilde{r} - 2}{2\tilde{r}}}d\tau\right)^{1/\tilde{q}}.
\end{multline}
First we bound the interior integral, namely the integral with respect to $\xi$.
We divide the integral into two parts: the integral over $|\xi| \le M$, and the integral over $|\xi| \ge M$, where $M$ will be determined later:
\begin{align*}
&\int_{\R^n} \left|(1 + |\xi|^{2})^{s/2}(\hat{K}_{t_1-\tau}(\xi) - \hat{K}_{\tilde{t}-\tau}(\xi))\right|^{\frac{2\tilde{r}}{\tilde{r} - 2}} d\xi 
\\
&= \int_{|\xi| \le M} \left|(1 + |\xi|^{2})^{s/2}(\hat{K}_{t_1-\tau}(\xi) - \hat{K}_{\tilde{t}-\tau}(\xi))\right|^{\frac{2\tilde{r}}{\tilde{r} - 2}} d\xi 
\\
&+ \int_{|\xi| \ge M} \left|(1 + |\xi|^{2})^{s/2}(\hat{K}_{t_1-\tau}(\xi) - \hat{K}_{\tilde{t}-\tau}(\xi))\right|^{\frac{2\tilde{r}}{\tilde{r} - 2}} d\xi. 
\end{align*}
We estimate each part separately, keeping in mind that $| \hat{K}_t(\xi) | \le e^{-\frac{| \xi |}{2}t} | t |$, 
$| t_1 - \tilde{t} | < \delta$, and $\tau \in (0,t_1-\delta')$.
The smallness of the first integral will follow from the absolute continuity of the integrand as a function of time,
and the smallness of the second integral will follow from the exponential decay of the kernel $\hat{K}_t(\xi)$, which can be made
arbitrarily small for high frequencies $\xi$ by the choice of $M > 1$. 

Indeed, we start with the integral over $|\xi| \ge M$.
We claim that we can choose 
$M > 1$ sufficiently large such that the contribution of that integral is smaller than a bound given in terms of $\epsilon$. 
Namely, we will show that for every $\epsilon' > 0$, there is an $M>1$, such that 
for all $0 \le \tau \le t_{1} - \delta'$, {{$|\tilde{t} - t_{1}| < \delta < \delta'/2$}}, we have
\begin{multline}\label{intineq4}
\int_{|\xi| \ge M} \left|(1 + |\xi|^{2})^{s/2}(\hat{K}_{t_1-\tau}(\xi) - \hat{K}_{\tilde{t}-\tau}(\xi))\right|^{\frac{2\tilde{r}}{\tilde{r} - 2}} d\xi \\
\le {{}{C_{\tilde{r}}\left(\int_{|\xi| \ge M} \left|(1 + |\xi|^{2})^{s/2}(t_{1} - \tau)e^{-(t_{1} - \tau)\frac{|\xi|}{2}}\right|^{\frac{2\tilde{r}}{\tilde{r} - 2}}d\xi + \int_{|\xi| \ge M} \left|(1 + |\xi|^{2})^{s/2}(\tilde{t} - \tau)e^{-(\tilde{t} - \tau)\frac{|\xi|}{2}}\right|^{\frac{2\tilde{r}}{\tilde{r} - 2}}d\xi\right)}} < \frac{\epsilon'}{2}.
\end{multline}
Later, we will choose $\epsilon'$ to depend on $\epsilon$ and $F$ so that the total integral  $I^{(2)}_{t_{1}, \tilde{t}}$ is bounded by $\epsilon/3$.
The contribution of each integral in the sum above can be estimated by noticing that both integrals have the form
\begin{equation*}\label{intineq3}
\int_{|\xi| \ge M} \left|(1 + |\xi|^{2})^{s/2} \hat{t}e^{-\hat{t}\frac{|\xi|}{2}}\right|^{\frac{2\tilde{r}}{\tilde{r} - 2}}d\xi,  \quad {\rm for} \  \hat{t} \in [\frac{\delta'}{2},1],
\end{equation*}
where $\hat{t}$ plays the role of $t_1 - \tau$ and $\tilde{t}-\tau$.
Indeed, because of the exponential decay in $\xi$, for every $\epsilon' > 0$, there exists a constant $M > 1$ such that
\begin{equation}\label{intineq3}
{{}{C_{\tilde{r}}}}\int_{|\xi| \ge M} \left|(1 + |\xi|^{2})^{s/2} \hat{t}e^{-\hat{t}\frac{|\xi|}{2}}\right|^{\frac{2\tilde{r}}{\tilde{r} - 2}}d\xi < \frac{\epsilon'}{4}, \quad \forall \hat{t} \in [\frac{\delta'}{2},1].
\end{equation}
Here,  $\epsilon'$ will be chosen below in terms of $\epsilon$ and $||F||_{L^{\tilde{q}'}_{t}([0, 1]; L^{\tilde{r}'}_{x}(\mathbb{R}^{n}))}$. 

%%%%%%%%%%
\if 1 = 0
To see that this is possible, introduce $c := \frac{2\tilde{r}}{\tilde{r} - 2} > 2$ to simplify notation, and calculate:
\begin{align*}
\int_{|\xi| \ge M} \left|(1 + |\xi|^{2})^{s/2}\hat{t}e^{-\hat{t}\frac{|\xi|}{2}}\right|^{\frac{2\tilde{q}}{\tilde{q} - 2}}d\xi 
& = \int_{|\xi| \ge M} (1 + |\xi|^{2})^{\frac{cs}{2}} \hat{t}^{c} e^{-\hat{t}c\frac{|\xi|}{2}} d\xi
\\
&= C\hat{t}^{c - n}\int_{|\eta| \ge M\hat{t}c} \left(1 + \frac{|\eta|^{2}}{(\hat{t}c)^{2}}\right)^{\frac{cs}{2}}e^{-\frac{|\eta|}{2}}d\eta,
\end{align*}
where we introduced change of variables $\eta = \hat{t} \xi$.
Note that for  $\hat{t} \in [\delta'/2,1]$, since $c - n > 0$, $\hat{t}^{c - n}$ is bounded, and so we can choose $M > 1$ so that
\begin{equation}\label{Mcondition}
\int_{|\eta| \ge M\hat{t}c} \left(1 + \frac{|\eta|^{2}}{(\hat{t}c)^{2}}\right)^{\frac{cs}{2}}e^{-\frac{|\eta|}{2}}d\eta < \frac{\epsilon'}{4}.
\end{equation}
\fi
%%%%%%%%%%%%%%%

%%%%%%%%%%%%%%%%%
\if 1 = 0
for all $\delta'/2 \le t \le 1$, we could choose $M$ as in \eqref{intineq3}. If $tc \ge 1$, then
\begin{equation}\label{Msubcond1}
\int_{|\eta| \ge Mtc} \left(1 + \frac{|\eta|^{2}}{(tc)^{2}}\right)^{\frac{cs}{2}}e^{-\frac{|\eta|}{2}}d\eta \le \int_{|\eta| \ge M} (1 + |\eta|^{2})^{\frac{cs}{2}}e^{-\frac{|\eta|}{2}}d\eta
\end{equation}
If $\delta'c/2 \le tc \le 1$ (which we need not consider if indeed $\delta'c/2 \ge 1$), then
\begin{equation}\label{Msubcond2}
\int_{|\eta| \ge Mtc} \left(1 + \frac{|\eta|^{2}}{(tc)^{2}}\right)^{\frac{cs}{2}}e^{-\frac{|\eta|}{2}}d\eta \le (\delta'c)^{-cs} \int_{|\eta| \ge M\delta'c} (1 + |\eta|^{2})^{\frac{cs}{2}} e^{-\frac{|\eta|}{2}} d\eta
\end{equation}
We can then choose $M > 1$ so that both integrals on the right hand side of inequalities \eqref{Msubcond1} and \eqref{Msubcond2} are less than $\frac{c'\epsilon'}{4}$. So we can choose $M > 1$ so that \eqref{intineq3} holds. 
\fi
%%%%%%%%%%%%%%%%

%%%%%%
\if 1 = 0

Consequently, for all $0 < t_{1}, \tilde{t} \le 1$, $0 \le \tau \le t_{1} - \delta'$, 
\begin{multline}\label{intineq4}
\int_{|\xi| \ge M} \left|(1 + |\xi|^{2})^{s/2}(\hat{K}_{t_1-\tau}(\xi) - \hat{K}_{\tilde{t}-\tau}(\xi))\right|^{\frac{2\tilde{q}}{\tilde{q} - 2}} d\xi \\
\le \int_{|\xi| \ge M} \left|(1 + |\xi|^{2})^{s/2}(t_{1} - \tau)e^{-(t_{1} - \tau)\frac{|\xi|}{2}}\right|^{\frac{2\tilde{q}}{\tilde{q} - 2}}d\xi + \int_{|\xi| \ge M} \left|(1 + |\xi|^{2})^{s/2}(\tilde{t} - \tau)e^{-(\tilde{t} - \tau)\frac{|\xi|}{2}}\right|^{\frac{2\tilde{q}}{\tilde{q} - 2}}d\xi < \frac{\epsilon'}{2}
\end{multline}
since we will choose $\delta$ so that $\delta < \delta'/2$ later, so $\tilde{t} - \tau \ge \delta'/2$ for $0 \le \tau \le t_{1} - \delta'$ and $\tilde{t} \in (t_{1} - \delta, t_{1} + \delta)$. 

\fi

%%%%%%%%%%

To estimate the integral over $|\xi | \le M$, we notice that the integrand  
$e^{-\frac{|\xi|}{2}\hat{t}}  \frac{\text{sin}\left(\frac{\sqrt{3}}{2}|\xi|\hat{t}\right)}{\frac{\sqrt{3}}{2}|\xi|}$ 
is uniformly continuous on the compact set $|\xi| \le M, \delta'/2 \le \hat{t} \le 1$. 
This implies that for every $\epsilon'> 0$, we can choose  $\delta > 0$
such that  $\delta < \delta'/2$ and such that for all $\tilde{t} \in (t_{1} - \delta, t_{1} + \delta)$ and $0 \le \tau \le t_{1} - \delta'$, we have
\begin{align}\label{intineq5}
&\qquad \qquad\qquad\qquad
\int_{|\xi| \le M} \left|(1 + |\xi|^{2})^{s/2}(\hat{K}_{t_1-\tau}(\xi) - \hat{K}_{\tilde{t}-\tau}(\xi))\right|^{\frac{2\tilde{r}}{\tilde{r} - 2}} d\xi 
\nonumber
\\
&\le C_{M} \int_{|\xi| \le M} \left|e^{-\frac{|\xi|}{2}(t_{1} - \tau)}\frac{\text{sin}\left(\frac{\sqrt{3}}{2}|\xi|(t_{1} - \tau)\right)}{\frac{\sqrt{3}}{2}|\xi|} - e^{-\frac{|\xi|}{2}(\tilde{t} - \tau)}\frac{\text{sin}\left(\frac{\sqrt{3}}{2}|\xi|(\tilde{t} - \tau)\right)}{\frac{\sqrt{3}}{2}|\xi|}\right|^{\frac{2\tilde{r}}{\tilde{r} - 2}}d\xi < \frac{\epsilon'}{2}.
\end{align}

By combining inequalities \eqref{intineq4} and \eqref{intineq5} with \eqref{intineq1} and \eqref{intineq2}, 
we see that we can choose $\epsilon'$ small enough, depending on $\epsilon$ and $||F||_{L^{\tilde{q}'}_{t}([0, 1]; L^{\tilde{r}'}_{x}(\mathbb{R}^{n}))}$, 
such that
 for all $\tilde{t} \in (t_{1} - \delta, t_{1} + \delta)$, 
\begin{equation}\label{last_int}
0 \le I^{(2)}_{t_{1}, \tilde{t}} < \frac{\epsilon}{3}.
\end{equation}

Continuity in time now follows by 
combining the estimates  \eqref{contbound}, \eqref{contest1}, and \eqref{last_int}. This completes the proof.
\end{proof}

{{}
\begin{remark}
We remark that Proposition \ref{Hslemma} also reflects the parabolic regularizing effects of the fluid viscosity. In fact, once one shows that the solution to the linear viscous wave equation with source term $F \in L^{\tilde{q}'}_{t}([0, T]; L^{\tilde{r}'}_{x}(\R^{n}))$ and zero initial data gives rise to a solution that is in $C^{0}([0, T]; H^{s}(\R^{n}))$,  one can use energy estimates to establish \eqref{Sobpropresult} for a couple of specific exponents in $n = 2$. In particular, the case of $s = 1$, $\tilde{q} = \infty$, and $\tilde{r} = 2$ follows from usual basic energy estimates, while the case of $s = 1$, $\tilde{q} = 2$, and $\tilde{r} = 4$ can be obtained from energy estimates along with the Sobolev embedding of $H^{1/2}(\R^{2})$ into $L^{4}(\R^{2})$. However, the Fourier methods we have used to establish Proposition \ref{Hslemma} provide a simpler way of covering a much larger range of potential parameters $(s, \tilde{q}, \tilde{r})$ than basic energy estimates.
\end{remark}

\begin{remark}
We have established the result in Proposition \ref{Hslemma} for $n = 2$. The main place we have used the fact that $n = 2$ is in \eqref{C0estn}, where we need the time integral 
\begin{equation*}
\int_{0}^{T} (T - \tau)^{\tilde{q}\left(1 - n \cdot \frac{\tilde{r} - 2}{2\tilde{r}}\right)} d\tau
\end{equation*}
to be finite. For $n = 2$, the exponent $\tilde{q}\left(1 - n \cdot \frac{\tilde{r} - 2}{2\tilde{r}}\right)$ is always nonnegative. If one were to attempt to extend the result to higher dimensions, this exponent could potentially be negative, and one would have to impose additional conditions on $\tilde{q}$ and $\tilde{r}$ to ensure that the resulting time integral above is finite. 
\end{remark}
}

\section{Probabilistic well-posedness for the supercritical quintic nonlinear viscous wave equation}\label{random}

\subsection{Description of the randomization and main result}

In this section, we will specialize to the case of the quintic nonlinear viscous wave equation. In particular, we consider
\begin{equation}\label{quinticeq}
\partial_{tt}u - \Delta u + \sqrt{-\Delta}\partial_{t}u + u^{5} = 0 \ \ \ \ \ \text{ on } \mathbb{R}^{2},
\end{equation}
with a randomization of the deterministic initial data in $\mathcal{H}^{s}(\mathbb{R}^{2}) = H^{s} \times H^{s - 1}$:
\begin{equation*}\label{data}
u(0, x) = f(x) \in H^{s}, \ \ \ \ \ \partial_{t}u(0, x) = g(x) \in H^{s - 1}.
\end{equation*}
The reason for considering the quintic nonlinearity is that $5$ is the first odd exponent in $n=2$ for {{which we have ill-posedness as described by Theorem \ref{ill_posedness}, as it is the first odd exponent for which $s_{cr} = \frac{n}{2} - \frac{2}{p - 1}$ for $n = 2$ is positive}}. Moreover, for this nonlinearity, {{we can use our results from the previous sections to generate Strichartz estimates for the solution space \eqref{X_T} that allow us to handle the quintic nonlinearity}}, as the Remarks~\ref{quintic} and \ref{weak_solution} below explain.

For this specific case of the quintic {{viscous wave}} equation, the critical exponent is $s_{\text{cr}} = 1/2$. It is likely that there is well posedness in the strong Hadamard sense for initial data in $\mathcal{H}^{s}$ for $s > 1/2$, but in Sec. \ref{illposed}, we showed that there is a lack of continuity in the solution map for $0 < s < 1/2$. We will show that this lack of continuity in the solution map is in some sense a non-generic phenomenon by considering a randomization of the initial data, as described below. In particular, we will use what is called the \textit{Wiener randomization}. 

To define the randomization, denote by $\psi \in C_{0}^{\infty}(\overline{B_{r = 2}(0)})$, $0 \le \psi \le 1$, a function such that
\begin{equation*}
\sum_{k \in \mathbb{Z}^{2}} \psi(\xi - k) \equiv 1.
\end{equation*}
This function defines a partition of unity created by translates of the same compactly supported function. 
Then, we note that for a given function $f$ with Fourier transform $\widehat{f}(\xi)$, 
\begin{equation*}
\widehat{f}(\xi) = \sum_{k \in \mathbb{Z}^{2}} \psi(\xi - k) \widehat{f}(\xi).
\end{equation*}
For each $k \in \mathbb{Z}^2$ define
\begin{equation*}
P_{k}f := \mathcal{F}^{-1}\left(\psi(\xi - k)\widehat{f}(\xi)\right)
\end{equation*}
to be the inverse Fourier transform of a localized portion of the function in frequency space. 
Note that {for some positive constant $c > 0$},
\begin{equation*}
0 < c \le \sum_{k \in \mathbb{Z}^{2}} \psi^{2}(\xi - k) \le 1,
\end{equation*}
and therefore, the two norms
\begin{equation}\label{normequiv}
||f||_{H^{s}(\mathbb{R}^{2})} \sim (\sum_{k \in \mathbb{Z}^{2}} ||P_{k}f||^{2}_{H^{s}(\mathbb{R}^{2})})^{1/2} 
\end{equation}
are equivalent.

Now, let $h_{k}(\omega)$ and $l_{k}(\omega)$ indexed by $k \in \mathbb{Z}^{2}$ be independent real random variables with mean zero on a probability space $(\Omega, \mathcal{F}, P)$.  Assume that the independent random variables $h_{k}(\omega)$ and $l_{k}(\omega)$ have uniformly bounded sixth moments:
\begin{equation}\label{six}
\int_{\Omega} (|h_{k}(\omega)|^{6} + |l_{k}(\omega)|^{6}) dp(\omega) < C \ \ \ \ \ \text{ for all } k \in \mathbb{Z}^{2}.
\end{equation}    
The uniform boundedness of sixth moments is associated with the quintic nonlinearity.
It will provide sufficient regularity for the randomization of the initial data 
by the real independent random variables with mean zero and uniformly bounded sixth moments
to improve the regularity of the {\sl{averaged}} randomized free evolution solution in expectation,
where the averaging is performed in terms of the averaged $L^6$-norms of the solution.
See Lemma~\ref{averaging} below.

For each $\omega\in\Omega$, we define the randomized initial data, given by
\begin{equation}\label{random_data}
f^{\omega} = \sum_{k \in \mathbb{Z}^{2}} h_{k}(\omega) P_{k}f, \qquad
g^{\omega} = \sum_{k \in \mathbb{Z}^{2}} l_{k}(\omega) P_{k}g.
\end{equation}
We introduce the map that associates to each possible outcome
$\omega \in \Omega$, the corresponding initial data $\varphi^{\omega}= (f^{\omega}, g^{\omega})\in \mathcal{H}^{s}(\mathbb{R}^{2})$:
\begin{equation}\label{phiomega}
\omega \mapsto \varphi^{\omega}= (f^{\omega}, g^{\omega}), \quad \varphi^{\omega}: \Omega \to \mathcal{H}^{s}(\mathbb{R}^{2}).
\end{equation}
 One can show that the map $\varphi^{\omega}$
is a measurable map from $\Omega \to \mathcal{H}^{s}(\mathbb{R}^{2})$, and $\varphi^\omega \in  L^{2}(\Omega; \mathcal{H}^{s}(\mathbb{R}^{2}))$. 
Thus, $\varphi^\omega$ is a random variable taking values in $\mathcal{H}^{s}(\mathbb{R}^{2})$. In particular, one can show that the following result holds:

\begin{proposition}\label{phi_omega_estimate}
{{Let $\varphi \in \mathcal{H}^{s}(\mathbb{R}^{2})$, and let $\varphi^{\omega}= (f^{\omega}, g^{\omega})$ be defined as in \eqref{random_data} and \eqref{phiomega} via the independent random variables $h_{k}(\omega)$ and 
$l_{k}(\omega)$ with uniformly bounded sixth moments \eqref{six}.}} Then
\begin{equation}\label{randomprop}
||\varphi^{\omega}||_{L^{2}(\Omega; \mathcal{H}^{s}(\mathbb{R}^{2}))} \le C||\varphi||_{\mathcal{H}^{s}}.
\end{equation}
Namely, $\varphi^\omega$ has almost surely  finite $\mathcal{H}^{s}$ norm.
\end{proposition}

\begin{proof}
Indeed, using \eqref{normequiv}, we get:
%by using the boundedness of the sixth moments assumption \eqref{six}, we get:
%$\varphi^{\omega}$ a.s. has finite $\mathcal{H}^{s}$ norm by noting that 
%$||\varphi^{\omega}||_{L^{2}(\Omega; \mathcal{H}^{s}(\mathbb{R}^{2}))}$ is finite (using the bounded sixth moment assumption which implies bounded second moments). For example, using \eqref{normequiv}, we have
\begin{align*}
\|f^{\omega}\|^{2}_{L^{2}(\Omega; H^{s}(\mathbb{R}^{2}))} &= \int_{\Omega} \|f^{\omega}\|^{2}_{H^{s}} dP = \int_{\Omega} \|\sum_{k \in \mathbb{Z}^{2}} h_{k}(\omega) P_{k}f\|_{H^{s}}^{2} dP 
\le C\int_{\Omega} \sum_{j \in \mathbb{Z}^{2}} \| P_{j} (\sum_{k \in \mathbb{Z}^{2}} h_{k}(\omega) P_{k}f )\|^{2}_{H^{s}}dP.
\end{align*}
Now, by the support properties of $\psi(\xi)$, we have  $P_{j}P_{k}f = 0$ when $|j - k| > c$. 
We use this property to continue the above estimate as follows:
\begin{equation*}
= C\int_{\Omega} \sum_{j \in \mathbb{Z}^{2}} \|\sum_{|k - j| \le c} h_{k}(\omega)P_{j}P_{k}f\|^{2}_{H^{s}} dP \le C \int_{\Omega} \sum_{j \in \mathbb{Z}^{2}} \left(\sum_{|k - j| \le c} ||h_{k}(\omega)P_{j}P_{k}f||_{H^{s}}\right)^{2} dP.
\end{equation*}
Since for each $j$, the inner sum has finitely many terms with the same number of terms for each $j$, we can use 
the following inequality  $\left(\sum_{k = 1}^{N} a_{k}\right)^{2} \le C_{N}\sum_{k = 1}^{N} a_{k}^{2}$ for the sum of $N$ terms,
to continue the above estimate:
\begin{multline*}
\le C\int_{\Omega} \sum_{j \in \mathbb{Z}^{2}} \sum_{|k - j| \le c} |h_{k}(\omega)|^{2} ||P_{j}P_{k}f||^{2}_{H^{s}} dP = C\sum_{j \in \mathbb{Z}^{2}} \sum_{|k - j| \le c} \left(\int_{\Omega} |h_{k}(\omega)|^{2} dP \cdot ||P_{j}P_{k}f||^{2}_{H^{s}}\right).
\end{multline*}
Now, from the boundedness of the sixth moments \eqref{six}, we have that the second moments are uniformly bounded, 
and we use this to continue to estimate the $L^{2}(\Omega; \mathcal{H}^{s}(\mathbb{R}^{2}))$-norm
of $f^\omega$ as follows:
\begin{multline*}
\le C' \sum_{j \in \mathbb{Z}^{2}} \sum_{|k - j| \le c} ||P_{j}P_{k}f||^{2}_{H^{s}} = C' \int_{\mathbb{R}^{2}} \sum_{j \in \mathbb{Z}^{2}} \sum_{|k - j| \le c} (1 + |\xi|^{2})^{s} |\psi(\xi - j)|^{2} |\psi(\xi - k)|^{2} |\widehat{f}(\xi)|^{2} d\xi.
\end{multline*}
Denote by $N$ be the number of terms in the sum with $|k - j| \le c$,
 which is the same for all $j$ by translation. 
 Then, since $|\psi(\xi - k)|^{2} \le |\psi(\xi - k)| \le 1$, we finally obtain the following estimate for the $L^{2}(\Omega; \mathcal{H}^{s}(\mathbb{R}^{2}))$-norm
of $f^\omega$:
\begin{equation*}
\le C'N \int_{\mathbb{R}^{2}} \sum_{j \in \mathbb{Z}^{2}} (1 + |\xi|^{2})^{s} |\psi(\xi - j)| |\widehat{f}(\xi)|^{2} d\xi = C'' \int_{\mathbb{R}^{2}} (1 + |\xi|^{2})^{s} |\widehat{f}(\xi)|^{2} d\xi = C ||f||^{2}_{H^{s}}.
\end{equation*}
%So we have that $||f^{\omega}||_{L^{2}(\Omega; H^{s}(\mathbb{R}^{2}))} \le C||f||_{H^{s}}$ for example. 

%A similar computation shows that we can interpret $f^{\omega}$ and $g^{\omega}$ as limits of the Cauchy sequences of partial sums over $|k| \le N$ as $N \to \infty$, in $L^{2}(\Omega; H^{s}(\mathbb{R}^{2}))$. 
%In particular, 
By repeating the above computation with $g$ and $g^{\omega}$, one obtains:
\begin{equation}\label{randomprop}
||\varphi^{\omega}||_{L^{2}(\Omega; \mathcal{H}^{s}(\mathbb{R}^{2}))} \le C||\varphi||_{\mathcal{H}^{s}},
\end{equation}
where $\varphi = (f, g) \in \mathcal{H}^{s}$ denotes the initial data before randomization, and  $C >0$ is \textit{independent of $\varphi = (f, g) \in \mathcal{H}^{s}$}.
\end{proof}

The goal in this section is to prove {\bf{probabilistic well-posedness}} for the 
quintic nonlinear viscous wave equation with initial data $\varphi^\omega \in L^{2}(\Omega; \mathcal{H}^{s}(\mathbb{R}^{2}))$
where {\bf{$-1/6 < s \le 1/2$}}. We do this in two steps.
First we will prove the existence of a unique solution stated in Theorem~\ref{existence} below, and then show
in Theorem~\ref{probwellposed}
that the solution depends continuously on data in $H^s$, for $-1/6 < s \le 1/2$.

\begin{theorem}[{\bf{Existence and uniqueness}}]\label{existence}
Let $-1/6 < s \le 1/2$, and let $\varphi^{\omega} \in L^{2}(\Omega; \mathcal{H}^{s}(\mathbb{R}^{2}))$ be the Wiener randomization of the initial data
$\varphi = (f, g) \in \mathcal{H}^{s}(\mathbb{R}^{2})$. 
Then, for almost all $\omega \in \Omega$, there exists $T_{\omega} > 0$ such that there is a unique solution {{$u$}} to the Cauchy problem 
of the nonlinear quintic viscous wave equation \eqref{quinticeq} with initial data $\varphi^{\omega}$, where the solution belongs to the space
\begin{align}\label{sol_space}
Z^{\omega}_{[0,T_\omega]} := e^{-\frac{\sqrt{-\Delta}}{2}t}&\left( 
\text{cos}\left(\frac{\sqrt{3}}{2}\sqrt{-\Delta}t\right) 
+ \frac{1}{\sqrt{3}}\text{sin}\left(\frac{\sqrt{3}}{2}\sqrt{-\Delta}t\right)\right)(f^{\omega}) 
+ e^{-\frac{\sqrt{-\Delta}}{2}t}\frac{\text{sin}\left(\frac{\sqrt{3}}{2}\sqrt{-\Delta}t\right)}{\frac{\sqrt{3}}{2}\sqrt{-\Delta}}(g^{\omega})
\nonumber \\
&+ C^{0}([0, T_\omega]; H^{1/2}(\mathbb{R}^{2})) \cap L^{6}([0, T_\omega]; L^{6}(\mathbb{R}^{2})).
\end{align}
In particular, there exists $C > 0$ such that for every $0 < T \le 1$, there is an event $\Omega_{T}$ with
\begin{equation*}
P(\Omega_{T}) \ge 1 - CT^{s + 1/6},
\end{equation*}
such that for all $\omega \in \Omega_{T}$,
the Cauchy problem for \eqref{quinticeq} with initial data $\varphi^{\omega}$ has a unique solution in the space $Z^{\omega}_{[0,T]}$.
%%%%%%%%
\if 1 = 0
\begin{align*}
Z^{\omega}_{[0,T]} = e^{-\frac{\sqrt{-\Delta}}{2}t}\left(
\right.
& \left. 
\text{cos}\left(\frac{\sqrt{3}}{2}\sqrt{-\Delta}t\right) 
+ \frac{1}{\sqrt{3}}\text{sin}\left(\frac{\sqrt{3}}{2}\sqrt{-\Delta}t\right)\right)(f^{\omega}) 
+ e^{-\frac{\sqrt{-\Delta}}{2}t}\frac{\text{sin}\left(\frac{\sqrt{3}}{2}\sqrt{-\Delta}t\right)}{\frac{\sqrt{3}}{2}\sqrt{-\Delta}}(g^{\omega})
\nonumber \\
&+ C^{0}([0, T]; H^{1/2}(\mathbb{R}^{2})) \cap L^{6}([0, T]; L^{6}(\mathbb{R}^{2})).
\end{align*}
\fi
%%%%%%%%%%%%%
\end{theorem}

%%%%%%%%%%
\begin{theorem}[{\bf{Continuous dependence}}]\label{probwellposed}
Let $-1/6 < s \le 1/2$ and $\varphi \in \mathcal{H}^{s}(\mathbb{R}^{2})$.
Then, for any choice of $\epsilon > 0$, $0 < T \le 1$, and $0 < p < 1$, 
there exists an event $A_{\varphi, \epsilon}$ 
and $\delta > 0$ (with $\delta > 0$ depending only on $\epsilon$, $T$, and $p$), 
such that for every $\omega \in A_{\varphi, \epsilon}$
the Cauchy problem for the nonlinear quintic viscous wave equation with initial data $\varphi^{\omega}$
 has a  solution 
 {{$u \in C^{0}([0, T]; H^{s}(\mathbb{R}^{2}))$}}, which satisfies
$${{||u||_{C^{0}([0, T];H^{s}(\mathbb{R}^{2}))} \le \epsilon}}\  
{\rm with\  probability} P(A_{\varphi, \epsilon})>p, \ {\rm whenever} \ 
||\varphi||_{\mathcal{H}^{s}(\mathbb{R}^{2})} < \delta.$$
\end{theorem}

\subsection{The solution space and preliminary estimates}
We begin by justifying the choice for the function space \eqref{sol_space} in the statement of the existence result in Theorem~\ref{existence}. 
For this purpose, we introduce the following notation:
\begin{definition}
For each $T > 0$, define the function space
\begin{equation}\label{X_T}
X_{T} = C^{0}([0, T]; H^{1/2}(\mathbb{R}^{2})) \cap L^{6}([0, T]; L^{6}(\mathbb{R}^{2})).
\end{equation}
\end{definition}
\vskip 0.1in
\noindent
Using this notation, the solution space \eqref{sol_space} can be written as:
\begin{align*}\label{sol_space}
Z^{\omega}_{[0,T_\omega]}
& := e^{-\frac{\sqrt{-\Delta}}{2}t}\left(
\text{cos}\left(\frac{\sqrt{3}}{2}\sqrt{-\Delta}t\right) 
+ \frac{1}{\sqrt{3}}\text{sin}\left(\frac{\sqrt{3}}{2}\sqrt{-\Delta}t\right)\right)(f^{\omega}) 
\\
&+ e^{-\frac{\sqrt{-\Delta}}{2}t}\frac{\text{sin}\left(\frac{\sqrt{3}}{2}\sqrt{-\Delta}t\right)}{\frac{\sqrt{3}}{2}\sqrt{-\Delta}}(g^{\omega})
+ X_{T_\omega}.
\end{align*}

To obtain the main existence result in Theorem~\ref{existence}, we will use Strichartz estimates to obtain
estimates on the $L^{q}_{t}L^{r}_{x}$ norm of the solution to the {\sl{linear}} viscous wave equation
 in terms of the $L^{\tilde{q}'}_{t}L^{\tilde{r}'}_{x}$ norm of the ``right hand side,'' which will be the quintic nonlinearity $- u^5$.
This will be the basis for a fixed point argument, which will provide existence of a unique solution to the nonlinear problem.

\begin{remark}\label{quintic}
The choice of the function space $X_T$ in \eqref{X_T} follows from the Strichartz estimates that we plan to use to obtain the well-posedness result,
combined with a fixed point argument. 
In particular, we will be using 
the inhomogeneous Strichartz estimates for all $0 < T < \infty$ in Lemma \ref{inhStrlemma}, and the ``local'' $C_{0}([0, T]; H^{s}(\mathbb{R}^{2}))$ estimate for all $0 < T \le 1$ from Proposition \ref{Hslemma} to estimate the solution to the linear problem,
with the nonlinearity $-u^5$ treated as a source term. 
{{}{For the inhomogeneous Strichartz estimates in Lemma \ref{inhStrlemma} to hold, we need
\begin{equation}\label{cond1}
1 < \tilde{q}' < q < \infty, \quad 1 \le \tilde{r}' < r \le \infty, \quad 
\frac{1}{q} + \frac{n}{r} = \frac{1}{\tilde{q}'} + \frac{n}{\tilde{r}'} - 2,
\end{equation}
and for the ``local'' $C_{0}([0, T]; H^{s}(\mathbb{R}^{2}))$ estimate with $0 < T \le 1$ in Proposition \ref{Hslemma}, we need $s \ge 0$, and 
\begin{equation}\label{cond2}
2 \le \tilde{q} \le \infty, \quad 2 \le \tilde{r} < \infty, \quad
\frac{n}{2} - s = \frac{1}{\tilde{q}'} + \frac{n}{\tilde{r}'} - 2.
\end{equation}
In particular, for $n = 2$, the following exponents satisfy all of the conditions above:
\begin{equation*}
(q, r) = (\tilde{q}, \tilde{r}) = (6, 6),  \  {\rm and} \  s = \frac{1}{2}.
\end{equation*}
More importantly,}} for this choice,  we can indeed use the estimates in Lemma \ref{inhStrlemma} to estimate  
$u \in L^r_x(\R^2)=L^6(\R^2)$ in terms of the $L^{\tilde{r}'}_x$-norm of the source term $-u^5$, 
since the conjugate exponent $\tilde{r}'$ to the exponent $r = 6$ is $\tilde{r}'=6/5$, and we know
that $u \in L^6(\R^2) \implies u^5 \in L^{6/5}(\R^2)$.
Thus, the pair of exponents $6$ and $6/5$ is well suited for the nonlinear quintic viscous wave equation.
\end{remark}

\begin{remark}\label{weak_solution}
Note that the fact that we are requiring our solutions to be in $L^{6}([0, T]; L^{6}(\mathbb{R}^{2}))$ allows us to interpret solutions of 
\begin{equation*}
\partial_{tt}u - \Delta u + \sqrt{-\Delta} \partial_{t}u + u^{5} = 0,
\end{equation*}
\begin{equation*}
u(0, x) = f(x) \in H^{s}, \ \ \ \ \ \partial_{t}u(0, x) = g(x) \in H^{s - 1},
\end{equation*}
as weak solutions, defined in the usual way by integrating against a test function $\phi \in C_{0}^{\infty}([0, T) \times \mathbb{R}^{2})$. 
If we require solutions $u$ to be in $L^{6}([0, T]; L^{6}(\mathbb{R}^{2})) = L^{6}([0, T] \times \mathbb{R}^{2})$, then $u^{5} \in L^{6/5}([0, T] \times \mathbb{R}^{2})$ is indeed a distribution. In particular, it is also a tempered distribution in space, for almost every $t \in [0, T]$ (since it is in $L^{6/5}(\mathbb{R}^{2})$ for a.e. $t \in [0, T]$). Therefore, all of the Fourier methods 
applied to the Fourier representation of the solution used above are applicable.
\end{remark}

\noindent
{\bf\large{Strichartz estimate.}}
We will need  the following version of the Strichartz estimate:

\begin{corollary}[Strichartz estimates for solution space $X_{T}$]\label{Strcor}
There exists a constant $C$ (independent of $T$) such that for all $0 < T \le 1$,
\begin{equation*}
\left|\left|\int_{0}^{t} e^{-\frac{\sqrt{-\Delta}}{2}(t - \tau)}\frac{\text{sin}\left(\frac{\sqrt{3}}{2}\sqrt{-\Delta}(t - \tau)\right)}{\frac{\sqrt{3}}{2}\sqrt{-\Delta}}F(\tau, \cdot) d\tau\right|\right|_{X_{T}} \le C||F||_{L^{6/5}([0, T] \times \mathbb{R}^{2})}.
\end{equation*}
\end{corollary}

\begin{proof}
The proof follows immediately from Lemma \ref{inhStrlemma} and Proposition \ref{Hslemma}, where we let $q = r = \tilde{q} = \tilde{r} = 6$ and $s = 1/2$.
\end{proof}

\noindent
{\bf\large{Averaging effects.}}
{{The following two lemmas will be useful in proving the averaging effects result in Lemma \ref{averaging} below, which concerns the regularity in expectation of the averaged free evolution solution to the linear viscous wave equation for the random initial data $\varphi^{\omega}$.}} The first is a probabilistic lemma, stated in Burq and Tzvetkov \cite{BT}:

\begin{lemma}\label{problemma}
For $(h_{n})_{n = 1}^{\infty}$ independent, mean $0$, complex-valued random variables that satisfy a uniform moment bound
\begin{equation*}
\int_{\Omega} |h_{n}(\omega)|^{2k}dP \le C \ \ \ \ \ \text{ for all } n \ge 1
\end{equation*}
for some positive integer $k$, there exists a uniform constant $C$, such that for all $2 \le p \le 2k$ and for all $(c_{n})_{n = 1}^{\infty} \in \ell^{2}(\mathbb{C})$,
\begin{equation*}
\| \sum_{n} h_{n}(\omega)c_{n} \|_{L^{p}(\Omega)} \le C\left(\sum_{n} |c_{n}|^{2}\right)^{1/2} = C\| c_{n} \|_{\ell^{2}(\mathbb{C})}.
\end{equation*}
\end{lemma}

The next result is a unit-scale Bernstein's inequality for the frequency pieces $P_{k}f$. See Lemma 2.1 in L\"{u}hrmann and Mendelson \cite{LM}: 

\begin{lemma}[Bernstein unit-scale inequality]\label{bernstein}
For all $2 \le p_{1} \le p_{2} \le \infty$, there exists a constant $C = C(p_{1}, p_{2}) > 0$ such that for all $f \in L^{2}(\mathbb{R}^{2})$, $k \in \mathbb{Z}^{2}$,
\begin{equation*}
||P_{k}f||_{L^{p_{2}}(\mathbb{R}^{2})} \le C||P_{k}f||_{L^{p_{1}}(\mathbb{R}^{2})}.
\end{equation*}
\end{lemma}
The proof is a simple consequence of Young's convolution inequality and the localization in the frequency domain of $P_{k}f$ \cite{LM}. 

We are now ready to state and prove the averaging lemma. 
It is well-known that the randomization does not improve the regularity of the initial data almost surely. 
In particular, if $\varphi \in \mathcal{H}^{s}(\mathbb{R}^{2})$ is not in $\mathcal{H}^{s + \epsilon}(\mathbb{R}^{2})$, then $\varphi^{\omega}$ is a.s. not in $\mathcal{H}^{s + \epsilon}(\mathbb{R}^{2})$  (see L\"{u}hrmann and Mendelson \cite{LM}, Burq and Tzvetkov \cite{BT}). 
Although the randomization almost surely does not improve the regularity of the initial data, it does improve the averaging of the $L^{6}$ norms of the randomized {{free evolution solution}} in expectation. 
More precisely, we have:

%{\color{blue}{TODO: take a look at Luhrmann and Mendelson}}

\begin{lemma}[{\bf{Averaging Lemma}}]\label{averaging}
Let $u_{\varphi}^{\omega}$ be the free evolution associated with the randomization 
$\varphi^\omega = (f^\omega,g^\omega)$ of the initial data $\varphi = (f, g) \in \mathcal{H}^{s}(\mathbb{R}^{2})$,
with $-1/6 < s \le 1/2$:
\begin{equation*}
u_{\varphi}^{\omega}(t, x) = e^{-\frac{\sqrt{-\Delta}}{2}t}\left(\text{cos}\left(\frac{\sqrt{3}}{2}\sqrt{-\Delta}t\right) + \frac{1}{\sqrt{3}}\text{sin}\left(\frac{\sqrt{3}}{2}\sqrt{-\Delta}t\right)\right)(f^{\omega}) + e^{-\frac{\sqrt{-\Delta}}{2}t}\frac{\text{sin}\left(\frac{\sqrt{3}}{2}\sqrt{-\Delta}t\right)}{\frac{\sqrt{3}}{2}\sqrt{-\Delta}}(g^{\omega}).
\end{equation*}
Then, there exists a constant $C > 0$ depending only on $s$, independent of $\varphi$, such that for every $\varphi \in \mathcal{H}^{s}(\mathbb{R}^{2})$ and for all $0 < T \le 1$,
\begin{equation*}
||u_{\varphi}^{\omega}||_{L^{6}(\Omega \times [0, T] \times \mathbb{R}^{2})} \le CT^{\frac{1}{6}\left(s + \frac{1}{6}\right)} ||\varphi||_{\mathcal{H}^{s}(\mathbb{R}^{2})}.
\end{equation*}
\end{lemma}

\begin{proof}
First, we show the result for the case with zero initial velocity $g = 0$:
\begin{equation*}
u_{f}^{\omega}(t, x) = e^{-\frac{\sqrt{-\Delta}}{2}t} e^{i\frac{\sqrt{3}}{2}\sqrt{-\Delta}t} f^{\omega} = \sum_{k \in \mathbb{Z}^{2}} h_{k}(\omega) e^{-\frac{\sqrt{-\Delta}}{2}t} e^{i\frac{\sqrt{3}}{2}\sqrt{-\Delta}t} P_{k}f.
\end{equation*}
We recall that $(h_{k})$ are independent real random variables with
mean zero on a probability space $(\Omega,{\cal{F}},P)$, and calculate:
\begin{equation*}
||u_{f}^{\omega}||_{L^{6}(\Omega \times [0, T] \times \mathbb{R}^{2})} = ||u_{f}^{\omega}||_{L^{6}([0, T] \times \mathbb{R}^{2}; L^{6}(\Omega))}.
\end{equation*}
By the bounded sixth moment assumption on $(h_{k})$, we can apply Lemma \ref{problemma} with $k = 3$ to obtain
\begin{align}\label{avgeffects1}
||u_{f}^{\omega}||_{L^{6}(\Omega \times [0, T] \times \mathbb{R}^{2})} 
&\le C \left\| \sqrt{
\sum_{k \in \mathbb{Z}^{2}}|e^{-\frac{\sqrt{-\Delta}}{2}t} e^{i\frac{\sqrt{3}}{2}\sqrt{-\Delta}t} P_{k}f|^{2}} \ \right\|_{L^{6}([0, T] \times \mathbb{R}^{2})} 
\nonumber
\\
%&= C\left|\left|\left(\sum_{k \in \mathbb{Z}^{2}}|e^{-\frac{\sqrt{-\Delta}}{2}t} e^{i\frac{\sqrt{3}}{2}\sqrt{-\Delta}t} P_{k}f|^{2}\right) \ \right|\right|_{L^{3}([0, T] \times \mathbb{R}^{2})}^{1/2} 
%\nonumber
%\\
%&\le C \left(\sum_{k \in \mathbb{Z}^{2}} \left|\left| |e^{-\frac{\sqrt{-\Delta}}{2}t} e^{i\frac{\sqrt{3}}{2}\sqrt{-\Delta}t} P_{k}f|^{2} \right|\right|_{L^{3}([0, T] \times \mathbb{R}^{2})}\right)^{1/2} 
%\nonumber
%\\
%&\le C\left(\sum_{k \in \mathbb{Z}^{2}} \left|\left|e^{-\frac{\sqrt{-\Delta}}{2}t} e^{i\frac{\sqrt{3}}{2}\sqrt{-\Delta}t} P_{k}f \right|\right|_{L^{6}([0, T] \times \mathbb{R}^{2})}^{2} \right)^{1/2}
%\nonumber
%\\
&\le C\left(\sum_{k \in \mathbb{Z}^{2}} \left\| e^{-\frac{\sqrt{-\Delta}}{2}t} e^{i\frac{\sqrt{3}}{2}\sqrt{-\Delta}t} P_{k}f \right\|^{2}_{L^{6}((0, T]; L^{6}(\mathbb{R}^{2}))}\right)^{1/2}.
\end{align}
We now want to apply  the unit scale Bernstein inequality of Lemma \ref{bernstein},
to estimate the $L^6$-norm in space with the $L^2$-norm. 
Since $P_{k}$ commutes with the other Fourier multipliers that are acting on $f$,
we will apply the unit scale Bernstein inequality to 
$P_k \left(e^{-\frac{\sqrt{-\Delta}}{2}t} e^{i\frac{\sqrt{3}}{2}\sqrt{-\Delta}t} f\right)$.
For this purpose, as required by Lemma \ref{bernstein}, we first need to show that 
$e^{-\frac{\sqrt{-\Delta}}{2}t} e^{i\frac{\sqrt{3}}{2}\sqrt{-\Delta}t} f \in L^2(\R^2)$.
This is, indeed, the case, since:
\begin{align*}
&||e^{-\frac{\sqrt{-\Delta}}{2}t} e^{i\frac{\sqrt{3}}{2}\sqrt{-\Delta}t} f||_{L^{2}(\mathbb{R}^{2})}^{2} 
= ||e^{-\frac{|\xi|}{2}t} e^{i\frac{\sqrt{3}}{2}|\xi|t}\widehat{f}(\xi)||_{L^{2}(\mathbb{R}^{2})}^{2} 
= \int_{\mathbb{R}^{2}} e^{-|\xi|t} |\widehat{f}(\xi)|^{2} d\xi 
\\
&= \int_{\mathbb{R}^{2}} e^{-|\xi|t}(1 + |\xi|^{2})^{-s} (1 + |\xi|^{2})^{s}|\widehat{f}(\xi)|^{2} d\xi 
\le C_{s, t} \int_{\mathbb{R}^{2}} (1 + |\xi|^{2})^{s}|\widehat{f}(\xi)|^{2} d\xi = C_{s, t} ||f||_{H^{s}(\mathbb{R}^{2})}^{2},
\end{align*}
where we used the fact that for all $s \in \mathbb{R}$, $e^{-|\xi|t}(1 + |\xi|^{2})^{-s}$ is bounded whenever $t > 0$.

Therefore, we can apply the unit scale Bernstein inequality in \eqref{avgeffects1} to get:
\begin{align}\label{Bern}
C&\left(\sum_{k \in \mathbb{Z}^{2}} \left|\left|e^{-\frac{\sqrt{-\Delta}}{2}t} e^{i\frac{\sqrt{3}}{2}\sqrt{-\Delta}t} P_{k}f \right|\right|^{2}_{L^{6}((0, T]; L^{6}(\mathbb{R}^{2}))}\right)^{1/2} 
\le C \sqrt{\sum_{k \in \mathbb{Z}^{2}} \left|\left|e^{-\frac{\sqrt{-\Delta}}{2}t} e^{i\frac{\sqrt{3}}{2}\sqrt{-\Delta}t} P_{k}f \right|\right|^{2}_{L^{6}((0, T]; L^{2}(\mathbb{R}^{2}))}}
\nonumber
\\
&= C\sqrt{\sum_{k \in \mathbb{Z}^{2}} ||(1 + |\xi|^{2})^{-s/2} e^{-\frac{|\xi|}{2}t} (1 + |\xi|^{2})^{s/2} \widehat{P_{k}f}||^{2}_{L^{6}((0, T]; L^{2}(\mathbb{R}^{2}))}}.
\end{align}
By Minkowski's integral inequality {{we can ``change the order of integration''}} in \eqref{Bern} to obtain that \eqref{Bern} is bounded 
from above by:
\begin{equation}\label{avgeffects2}
\le C\left(\sum_{k \in \mathbb{Z}^{2}} ||(1 + |\xi|^{2})^{-s/2} e^{-\frac{|\xi|}{2}t} (1 + |\xi|^{2})^{s/2} \widehat{P_{k}f}||^{2}_{L^{2}(\mathbb{R}^{2}; L^{6}((0, T]))}\right)^{1/2},
\end{equation}
and we estimate the contribution in time of the exponential $e^{-\frac{|\xi|}{2}t}$ as follows:
\begin{equation*}
||e^{-\frac{|\xi|}{2}t}||_{L^{6}((0, T])} = \left(\frac{1}{3|\xi|}(1 - e^{-3|\xi|T})\right)^{1/6} = \frac{1}{(3|\xi|)^{-s}} \cdot \frac{(1- e^{-3|\xi|T})^{1/6}}{(3|\xi|)^{s + 1/6}}.
\end{equation*}
Since we can bound the exponential $1 - e^{-3x}$ in terms of monomial  $x^\alpha$, for any $\alpha \in [0,1]$, so that
$$
1 - e^{-3x} \le C_{\alpha} x^{\alpha}, \ \forall x \ge 0,
$$
we conclude that there exists a constant $C_s > 0$, such that
\begin{equation*}
||e^{-\frac{|\xi|}{2}t}||_{L^{6}((0, T])} \le C_{s} |\xi|^{s} |\xi|^{-s - \frac{1}{6}} (|\xi|T)^{\frac{1}{6}(s + \frac{1}{6})} = C_{s}|\xi|^{s} |\xi|^{-\delta(s)}T^{\frac{1}{6}\left(s + \frac{1}{6}\right)},
\end{equation*}
where we chose $0 < \alpha = s + 1/6 \le 1$ and $\delta(s) = \frac{5}{6}\left(s + \frac{1}{6}\right) > 0$. 

We now continue to estimate \eqref{avgeffects2} by separating the sum into the low frequency part $|k| \le 4$, 
and the high frequency part $|k| > 4$.
For the low frequency estimate, we note that 
\begin{equation*}
||e^{-\frac{|\xi|}{2}t}||_{L^{6}((0, T])} \le T^{1/6}.
\end{equation*}
Moreover, in the $|k| \le 4$ part, we can get rid of $(1 + |\xi|^{2})^{-s/2}$ using the support properties of $\widehat{P_{k}f}$ for $|k| \le 4$
to obtain that \eqref{avgeffects2} can be bounded from above by:
\begin{multline*}
\le \Bigg[C\sum_{|k| \le 4} ||T^{1/6} (1 + |\xi|^{2})^{s/2}\widehat{P_{k}f}||_{L^{2}(\mathbb{R}^{2})}^{2} \\
+ C\left(\sum_{|k| > 4} ||(1 + |\xi|^{2})^{-s/2}|\xi|^{s} \cdot |\xi|^{-\delta(s)} \cdot T^{\frac{1}{6}\left(s + \frac{1}{6}\right)} \cdot (1 + |\xi|^{2})^{s/2} \widehat{P_{k}f}||_{L^{2}(\mathbb{R}^{2})}^{2}\right)\Bigg]^{1/2}.
\end{multline*}
In the  $|k| > 4$ part, we have $|(1 + |\xi|^{2})^{-s/2} |\xi|^{s}| \le C$ and $|\xi|^{-\delta(s)} \le 1$ for all $\xi$ in the support of $\widehat{P_{k}f}$.
Therefore, we can further continue the estimate as follows:
\begin{equation*}
\le \left[CT^{1/3}\sum_{|k| \le 4} ||(1 + |\xi|^{2})^{s/2} \widehat{P_{k}f}||_{L^{2}(\mathbb{R}^{2})}^{2} + CT^{\frac{1}{3}\left(s + \frac{1}{6}\right)} \sum_{|k| > 4} ||(1 + |\xi|^{2})^{s/2} \widehat{P_{k}f}||^{2}_{L^{2}(\mathbb{R}^{2})}\right]^{1/2}.
\end{equation*}
Now, since $0 < 1/3(s + 1/6) \le 1/3$ for $-1/6 < s < 1/2$, and $0 < T \le 1$ we get that the above expression is bounded from above by:
\begin{equation*}
\le \left(CT^{\frac{1}{3}\left(s + \frac{1}{6}\right)} \sum_{k \in \mathbb{Z}^{2}} ||(1 + |\xi|^{2})^{s/2} \widehat{P_{k}f}||_{L^{2}(\mathbb{R}^{2})}^{2}\right)^{1/2} \le CT^{\frac{1}{6}\left(s + \frac{1}{6}\right)}||f||_{H^{s}(\mathbb{R}^{2})},
\end{equation*}
where we used \eqref{normequiv} in the last step. Thus, we have shown that
\begin{equation*}
||u_{f}^{\omega}||_{L^{6}(\Omega \times [0, T] \times \mathbb{R}^{2})} \le C T^{\frac{1}{6}\left(s + \frac{1}{6}\right)}||f||_{H^{s}(\mathbb{R}^{2})},
\end{equation*}
where $C>0$ depends only on $s$.

For the term involving $g$, the same proof still holds, with slight modifications. In particular, 
 instead of using the simple bound $\left|e^{i\frac{\sqrt{3}}{2}|\xi|t}\right| \le 1$, we have to use the bound
\begin{equation*}
\left|\frac{\text{sin}\left(\frac{\sqrt{3}}{2}|\xi|t\right)}{\frac{\sqrt{3}}{2}|\xi|}\right| \le 1_{\{|\xi| \le 1\}}(\xi) + \frac{1}{\frac{\sqrt{3}}{2}|\xi|}1_{\{|\xi| > 1\}}(\xi),
\end{equation*} 
which holds for all $0 \le t \le T \le 1$. This will give the low and high frequency estimate in the same way. 
%In the high frequency estimate, the $|\xi|$ in the denominator will combine with $|\xi|^{s}$ in the estimate for $||e^{-\frac{|\xi|}{2}t}||_{L^{6}((0, T])}$ to give a factor of $|\xi|^{s - 1}$, and of course, we will have $(1 + |\xi|^{2})^{\pm \frac{s - 1}{2}}$ instead of $(1 + |\xi|^{2})^{\pm s/2}$ since $g$ is in $H^{s - 1}$ (while $f$ was in $H^{s}$).  
Therefore, we have shown that there exists a constant $C > 0$, depending only on $s$, such that
\begin{equation*}
||u_{\varphi}^{\omega}||_{L^{6}(\Omega \times [0, T] \times \mathbb{R}^{2})} \le CT^{\frac{1}{6}\left(s + \frac{1}{6}\right)} ||\varphi||_{\mathcal{H}^{s}(\mathbb{R}^{2})}.
\end{equation*}
 
\end{proof}

%{\color{blue}{TODO: summarize where was it important that the 6-moments were bounded, and how is this related to the $L^6$-norms.}}
Crucial for the probabilistic well-posedness is the following corollary of the averaging lemma:
\begin{corollary}\label{probcor}
Let $u_{\varphi}^{\omega}$ be the free evolution associated with the randomization 
$\varphi^\omega = (f^\omega,g^\omega)$ of the initial data $\varphi = (f, g) \in \mathcal{H}^{s}(\mathbb{R}^{2})$,
with $-1/6 < s \le 1/2$ and $0 < T \le 1$:
\begin{equation*}
u_{\varphi}^{\omega}(t, x) = e^{-\frac{\sqrt{-\Delta}}{2}t}\left(\text{cos}\left(\frac{\sqrt{3}}{2}\sqrt{-\Delta}t\right) + \frac{1}{\sqrt{3}}\text{sin}\left(\frac{\sqrt{3}}{2}\sqrt{-\Delta}t\right)\right)(f^{\omega}) + e^{-\frac{\sqrt{-\Delta}}{2}t}\frac{\text{sin}\left(\frac{\sqrt{3}}{2}\sqrt{-\Delta}t\right)}{\frac{\sqrt{3}}{2}\sqrt{-\Delta}}(g^{\omega}).
\end{equation*}
Define the set
\begin{equation}\label{E}
E_{\lambda, T, \varphi} = \{\omega \in \Omega \ | \ ||u_{\varphi}^{\omega}||_{L^{6}([0, T] \times \mathbb{R}^{2})} \ge \lambda\}.
\end{equation}
Then, 
\begin{equation*}
P(E_{\lambda, T, \varphi}) \le CT^{s + \frac{1}{6}}\lambda^{-6}||\varphi||^{6}_{\mathcal{H}^{s}(\mathbb{R}^{2})},
\end{equation*}
where the constant $C$ depends only on $s$. 
\end{corollary}
The proof is a direct consequence of Chebychev's inequality and Lemma~\ref{averaging}.

Notice how, because we are using probabilistic methods, Chebyshev's inequality associates the size of the solution in the $L^6$-norm
to the size of the initial data in the $H^s$-norm, with the probability of such a bound given in terms of the size of the initial data.

\subsection{Proof of probabilistic well-posedness for $-1/6 < s \le 1/2$}

We are now ready to prove the probabilistic well-posedness results stated in Theorems~\ref{existence} and \ref{probwellposed}.

\subsubsection{Proof of Theorem~\ref{existence} on existence of a unique solution}
We want to prove that for $-1/6 < s \le 1/2$, 
and for almost all $\omega \in \Omega$, there exists $T_{\omega} > 0$ such that there is a unique solution $u$ to the Cauchy problem 
of the nonlinear quintic viscous wave equation \eqref{quinticeq} with initial data $\varphi^{\omega} \in L^{2}(\Omega; \mathcal{H}^{s}(\mathbb{R}^{2}))$, where the solution belongs to the space
\begin{align*}
Z^{\omega}_{[0,T_\omega]} := e^{-\frac{\sqrt{-\Delta}}{2}t}\left(
\right.
& \left. 
\text{cos}\left(\frac{\sqrt{3}}{2}\sqrt{-\Delta}t\right) 
+ \frac{1}{\sqrt{3}}\text{sin}\left(\frac{\sqrt{3}}{2}\sqrt{-\Delta}t\right)\right)(f^{\omega}) 
+ e^{-\frac{\sqrt{-\Delta}}{2}t}\frac{\text{sin}\left(\frac{\sqrt{3}}{2}\sqrt{-\Delta}t\right)}{\frac{\sqrt{3}}{2}\sqrt{-\Delta}}(g^{\omega})
\nonumber \\
&+ C^{0}([0, T_\omega]; H^{1/2}(\mathbb{R}^{2})) \cap L^{6}([0, T_\omega]; L^{6}(\mathbb{R}^{2})).
\end{align*}
In particular, we need to show that there exists $C > 0$ such that for every $0 < T \le 1$, there is an event $\Omega_{T}$ with
\begin{equation}\label{OmegaTbound}
P(\Omega_{T}) \ge 1 - CT^{s + 1/6},
\end{equation}
such that for all $\omega \in \Omega_{T}$ and initial data $\varphi^{\omega}$,
the Cauchy problem for \eqref{quinticeq} has a unique solution in the space $Z^{\omega}_{[0,T]}$.

\begin{proof}
To prove this theorem we use a  fixed point argument. 
We look for a solution to 
\begin{equation*}
(\partial_{tt} + \sqrt{-\Delta}\partial_{t} - \Delta)u + u^{5} = 0, \ \ \ \ \ u(0) = f^{\omega}, \ \ \ \ \ \partial_{t}u(0) = g^{\omega},
\end{equation*}
by expressing the solution $u$ as the sum of the free evolution $u_{\varphi}^{\omega}$, satisfying
\begin{equation*}
(\partial_{tt} + \sqrt{-\Delta}\partial_{t} - \Delta)u_{\varphi}^{\omega}  = 0, 
\ \ \ \ \ u_{\varphi}^{\omega}(0) = f^{\omega}, \ \ \ \ \ \partial_{t}u_{\varphi}^{\omega}(0) = g^{\omega},
\end{equation*}
and the function $v := u - u_{\varphi}^{\omega}$,
which satisfies:
\begin{equation*}
(\partial_{tt} + \sqrt{-\Delta}\partial_{t} - \Delta)v = -(u_{\varphi}^{\omega} + v)^{5}, \ \ \ \ \ v(0)=0, \ \ \ \ \  \partial_{t}v(0) = 0.
\end{equation*}
A solution to this equation on $[0, T]$ is a fixed point of the map
\begin{equation}\label{FixedPointMap}
K_{\varphi}^{\omega}: v(t, \cdot) \Longrightarrow -\int_{0}^{t} e^{-\frac{\sqrt{-\Delta}}{2}(t - \tau)}\frac{\text{sin}\left(\frac{\sqrt{3}}{2}\sqrt{-\Delta}(t - \tau)\right)}{\frac{\sqrt{3}}{2}\sqrt{-\Delta}}((u_{\varphi}^{\omega} + v)^{5})(\tau, \cdot)d\tau,
\end{equation}
defined on $X_T$, where this map depends on $\omega$.
Therefore, if we can show that the map $K_{\varphi}^{\omega}$ has a fixed point $v^* \in X_T$, the solution $u$ of the nonlinear quintic 
viscous wave equation will be given by $u = u_\varphi^\omega + v^*$, where $u_\varphi^\omega$ is the free evolution.
More precisely, we need to show that for almost all $\omega \in \Omega$, there exists $T_\omega>0$, such that 
$K_{\varphi}^{\omega}$ has a fixed point $v^* \in X_{T_\omega}$.

We will show this in three steps:
\begin{enumerate}
\item First we will show that there exists $\lambda = \lambda_0 > 0$ such that
if $\omega$ is in the complement of $E_{\lambda_0, T, \varphi}$, i.e.,
$\omega \in E_{\lambda_0, T, \varphi}^{c}$, where $E_{\lambda, T, \varphi}$ is defined by \eqref{E} in
Corollary \ref{probcor}, then the mapping $K_{\varphi}^{\omega}$ is a strict contraction {{on an appropriate closed subset of $X_{T}$,}}
for arbitrary $0 < T \le 1$.
\item Then, we will show that for almost all $\omega \in \Omega$, there exists a time $T_\omega>0$, such that
$\omega \in E_{\lambda_0, T_\omega, \varphi}^{c}$. 
\end{enumerate}
Steps 1 and 2 will imply that for almost all $\omega \in \Omega$, there exists $T_\omega>0$, such that 
$u = u_\varphi^\omega + v^* \in Z^\omega_{[0,T_\omega]}$ is a solution to the quintic nonlinear viscous wave equation,
where $v^*$ is the fixed point of $K_{\varphi}^{\omega}$. {{The third step is as follows:}}

\vskip 0.2in

{{3. Finally, we show that this solution $u = u_\varphi^\omega + v^* \in Z^\omega_{[0,T_\omega]}$ is unique in $Z^\omega_{[0,T_\omega]}$.}}

\vskip 0.2in

We start by taking $\omega$ in the complement of $E_{\lambda, T, \varphi}$, i.e.,
$\omega \in E_{\lambda, T, \varphi}^{c}$, where $E_{\lambda, T, \varphi}$ is defined by \eqref{E} in
Corollary \ref{probcor}. 
Namely, $\omega$ is {{an outcome for which the initial data $\varphi^\omega$ has a free evolution solution with bounded $L^{6}([0, T] \times \R^{2})$ norm:}}
\begin{equation*}
||u_{\varphi}^{\omega}||_{L^{6}([0, T] \times \mathbb{R}^{2})} \le \lambda,
\quad {\rm with} \quad
P(E^c_{\lambda, T, \varphi}) \ge 1-\tilde{C}T^{s + \frac{1}{6}}\lambda^{-6}||\varphi||^{6}_{\mathcal{H}^{s}(\mathbb{R}^{2})}.
\end{equation*}
This holds for any $0 < T  \le 1$.

For this fixed $\omega$, we want to find $\lambda=\lambda_0>0$, such that the mapping 
$K_{\varphi}^{\omega}$ is a contraction.
We start by estimating $K_{\varphi}^{\omega}(v)$.
  By using the Strichartz estimate in Corollary \ref{Strcor}, for each $\omega \in E_{\lambda, T, \varphi}^{c}$, we have
\begin{align}\label{contraction1}
||K_{\varphi}^{\omega}(v)||_{X_{T}} 
&\le C||(u_{\varphi}^{\omega} + v)^{5}||_{L^{6/5}([0, T] \times \mathbb{R}^{2})} 
\nonumber
\\
&\le C(||u_{\varphi}^{\omega}||^{5}_{L^{6}([0, T] \times \mathbb{R}^{2})} + ||v||^{5}_{L^{6}([0, T] \times \mathbb{R}^{2})})
\le C(\lambda^{5} + ||v||^{5}_{X_{T}}).
\end{align}
Similarly, one can show that for each $\omega \in E^{c}_{\lambda, T, \varphi}$, {{the following  estimate holds}}:
\begin{align}\label{contraction2}
||K_{\varphi}^{\omega}(v) - K_{\varphi}^{\omega}(w)||_{X_{T}} 
&\le C||(u_{\varphi}^{\omega} + v)^{5} - (u_{\varphi}^{\omega} + w)^{5}||_{L^{6/5}([0, T] \times \mathbb{R}^{2})} 
\nonumber
\\
&\le C||v - w||_{X_{T}}(\lambda^{4} + ||v||^{4}_{X_{T}} + ||w||^{4}_{X_{T}}).
\end{align}
Notice that because the constant in Corollary \ref{Strcor} is independent of $0 < T \le 1$ and $F$, 
the constant $C$ in the above two inequalities is also independent of $0 < T \le 1$ and $\omega \in E^{c}_{\lambda, T, \varphi}$. 

%Estimates \eqref{contraction1} and \eqref{contraction2} will be used to argue that for $\lambda$ small enough,
%the mapping $K_{\varphi}^{\omega}$ is a contraction for all $\omega \in E^c_{\lambda, T, \varphi}$. 
%Because the constant in Corollary \ref{Strcor} is independent of $T$, the constant $C$ in the above two inequalities is also independent of $T$. 
Therefore, for $\omega \in E^{c}_{\lambda, T, \varphi}$,  for any $0 < T \le 1$,
 the map $K_{\varphi}^{\omega}$ is a strict contraction on the ball of radius $2C\lambda^{5}$ in $X_{T}$ provided that
 the following two conditions {{on $\lambda$}} hold:
 \begin{enumerate}
\item The mapping $K_{\varphi}^{\omega}$ maps the ball of radius $2C\lambda^{5}$ in $X_{T}$ into the same ball of radius $2C\lambda^{5}$ in $X_{T}$;
 from \eqref{contraction1}, {{this means that we need $\lambda$ to satisfy:}}
\begin{equation}\label{fixedpt1}
C\lambda^{5} + 2^{5}C^{6}\lambda^{25} < 2C\lambda^{5};
\end{equation}
\item The coefficient multiplying $||v - w||_{X_{T}}$ in \eqref{contraction2} is strictly less than one:
\begin{equation}\label{fixedpt2}
C\lambda^{4} + 2^{5}C^{5}\lambda^{20} < \frac{1}{2}.
\end{equation}
\end{enumerate}
One can see that this will be guaranteed if $\lambda^{20} << 1$. 

Thus, there exists a $\lambda_0 > 0$ ($\lambda_0^{20} <<1$) such that whenever we choose 
an outcome $\omega$ so that the free evolution $u_\varphi^\omega$
associated with initial data $\varphi^\omega$ lies within the ball of radius $\lambda_0$ in $L^6([0,T]\times \R^2)$, 
the mapping $K_{\varphi}^{\omega}$, defined on the closed ball of radius $2C\lambda_0^5$ in $X_T$, 
is a strict contraction (where $C$ is the constant from \eqref{contraction1} and \eqref{contraction2}, which is independent of $0 < T \le 1$). So there exists a fixed point $v^{*}$ of $K^{\omega}_{\varphi}$ in the closed ball of radius $2C\lambda_{0}^{5}$ in $X_{T}$ for every $\omega \in E^{c}_{\lambda_{0}, T, \varphi}$. 

Notice that with this argument, we have shown the last part of the statement in the theorem, which is that
there exists $C' > 0$ such that for every $0 < T \le 1$, 
there exists an event $\Omega_T:=E_{\lambda_0,T,\varphi}^c$ with
\begin{equation}\label{POmegaT}
P(E^{c}_{\lambda_{0}, T, \varphi}) \ge 1 - \tilde{C}\lambda_{0}^{-6}||\varphi||^{6}_{\mathcal{H}^{s}(\mathbb{R}^{2})} T^{s + 1/6} 
= 1 - C' T^{s + 1/6},
\end{equation}
such that for all $\omega\in\Omega_T$ {{the existence of a solution in $Z^{\omega}_{[0, T]}$ holds}},
where $\tilde{C}$ is the constant from Corollary \ref{probcor}, and $C' = \tilde{C}\lambda_{0}^{-6}||\varphi||^{6}_{\mathcal{H}^{s}(\mathbb{R}^{2})}$. By the fixed point argument above, for each $\omega \in \Omega_{T} := E^{c}_{\lambda_0, T, \varphi}$, the solution $u = u_{\varphi}^{\omega} + v^{*}$ that we constructed is unique within the class of solutions of the form $u = u^{\omega}_{\varphi} + v$ such that 
\begin{equation}\label{uniquecond}
||v||_{X_{T}} \le 2C\lambda_{0}^{5},
\end{equation}
where $C$, the constant from \eqref{contraction1} and \eqref{contraction2}, is independent of $0 < T \le 1$. 

{{We need to show more specifically that the solution $u = u^{\omega}_{\varphi} + v^{*}$ that we constructed is unique in the space of solutions $u_{f}^{\omega} + v$, where $v$ more generally is any element in $X_T$ and not just an element of $X_T$ subject to condition \eqref{uniquecond}.} To see this, consider $\omega \in \Omega_{T} := E^{c}_{\lambda_{0}, T, \varphi}$ for some $0 < T \le 1$. In particular,
\begin{equation}\label{freebound}
||u^{\omega}_{\varphi}||_{L^{6}([0, T] \times \R^{2})} \le \lambda_{0}.
\end{equation}

We know that there exists a solution $u_{1} := u^{\omega}_{\varphi} + v^{*}_{1}$ to \eqref{quinticeq} with initial data $\varphi^{\omega}$, where
\begin{equation}\label{uniquecond2}
||v^{*}_{1}||_{X_{T}} \le 2C\lambda_{0}^{5},
\end{equation}
by our previous fixed point argument. 

Assume for contradiction that there is another solution $u_{2} = u^{\omega}_{\varphi} + v^{*}_{2}$ to \eqref{quinticeq} with initial data $\varphi^{\omega}$, where $v^{*}_{2} \in X_{T}$ is different from $v^{*}_{1}$. Since both $v^{*}_{1}$ and $v^{*}_{2}$ are continuous in $H^{1/2}$ on $[0, T]$, we can define
\begin{equation}\label{uniquenessT}
{{}{T_{*} = \max\{t \in [0, T] : v^{*}_{1} = v^{*}_{2} \text{ on } [0, t] \text{ as functions in } H^{1/2}(\mathbb{R}^{2})\}.}}
\end{equation}
We use maximum instead of supremum in the definition \eqref{uniquenessT} due to continuity in $H^{1/2}$. Because any two solutions to \eqref{quinticeq} with initial data $\varphi^{\omega}$ for our given $\omega$ are unique as long as the condition \eqref{uniquecond} holds, by continuity of the norms involved and the fact that $v_{1}^{*}$ and $v_{2}^{*}$ both have zero initial data, $T_{*} > 0$. Furthermore, $T_{*} < T$ by assumption. 

Because $v_{1}^{*}$ and $v_{2}^{*}$ in $X_{T}$ are different and the $v$ component of any solution $u := u^{\omega}_{\varphi} + v$ to \eqref{quinticeq} with initial data $\varphi^{\omega}$ is unique up to condition \eqref{uniquecond}, we conclude that 
\begin{equation}\label{v2star}
||v_{2}^{*}||_{X_t} > 2C\lambda_{0}^{5} \qquad \text{ for all } t > T_{*},
\end{equation}
since $v_{1}^{*}$ satisfies \eqref{uniquecond2}. Furthermore, since $v_{1}^{*}$ and $v_{2}^{*}$ agree in $H^{1/2}(\mathbb{R}^{2})$ up to time $T_{*}$, 
\begin{equation}\label{v2star2}
||v_{2}^{*}||_{X_t} \le 2C\lambda_{0}^{5} \qquad \text{ for all } t \le T_{*}
\end{equation}
by \eqref{uniquecond2}. 

We will derive a contradiction by showing that $v_{1}^{*}$ and $v_{2}^{*}$ must agree past $T_{*}$ to at least $T_{*} + \epsilon$, for a suitably small $\epsilon > 0$. To do this, we observe that the conditions \eqref{fixedpt1} and \eqref{fixedpt2} describing the choice of $\lambda_{0}$ are ``open" conditions. More precisely, there exists $\tilde{\lambda}_{0} > \lambda_{0}$ such that $\tilde{\lambda}_{0}$ also satisfies \eqref{fixedpt1} and \eqref{fixedpt2}. As a result, the map $K^{\omega}_{\varphi}$ defined in \eqref{FixedPointMap} is a strict contraction on the ball of radius $2C\tilde{\lambda}_{0}^{5}$ in $X_{t}$ for $\omega \in E^{c}_{\tilde{\lambda}_{0}, t , \varphi}$ for any $0 < t \le 1$. 

We now observe the following crucial fact: since $\tilde{\lambda}_{0} > \lambda_{0}$, we have that for our $\omega \in \Omega_{T}$, 
\begin{equation*}
\omega \in \Omega_{T} := E^{c}_{\lambda_{0}, T , \varphi} \subset E^{c}_{\tilde{\lambda}_{0}, T, \varphi} \subset E^{c}_{\tilde{\lambda}_{0}, t, \varphi} \qquad \text{ for all } 0 < t \le T.
\end{equation*}
Hence, for our arbitrary $\omega \in \Omega_{T}$, $K^{\omega}_{\varphi}$ is a strict contraction on the ball of radius $2C\tilde{\lambda}_{0}^{5}$ in $X_{t}$ for any $0 < t \le T$. Since $2C\tilde{\lambda}_{0}^{5} > 2C\lambda_{0}^{5}$, by \eqref{v2star2} and continuity of the norm, there exists $\epsilon > 0$ such that $T_{*} + \epsilon < T$ and
\begin{equation*}
||v_{2}^{*}||_{X_{T_{*} + \epsilon}} \le 2C\tilde{\lambda}_{0}^{5}.
\end{equation*}
Since $||v_{1}^{*}||_{X_{T_{*} + \epsilon}} \le 2C\tilde{\lambda}_{0}^{5}$ by \eqref{uniquecond2}, $v_{1}^{*}$ and $v_{2}^{*}$ are two different fixed points of the strict contraction $K^{\omega}_{\varphi}$ on the ball of radius $2C\tilde{\lambda}_{0}^{5}$ in $X_{T_{*} + \epsilon}$. Here, we used the crucial fact that $T_{*} + \epsilon < T$. This gives the desired contradiction, which proves that the solution to \eqref{quinticeq} with initial data $\varphi^{\omega}$ is unique within the space $Z^{\omega}_{[0, T]}$ for $\omega \in \Omega_{T}$. 
}

We now show the second step of the proof, namely, that for {\it{almost all $\omega \in \Omega$}}, there exists $T_\omega > 0$, such that 
the existence of a unique solution $u \in Z_{[0,T_\omega]}^\omega$ holds. 
This follows from the fact that for $T = 1/n$, $n \ge 1$, 
the sets $E^{c}_{\lambda_{0}, T = 1/n, \varphi}$ increase to a probability 1 subset of $\Omega$ as $n \to \infty$, 
since $P(E^{c}_{\lambda_{0}, T = 1/n, \varphi}) \nearrow 1$ {{by \eqref{OmegaTbound}}}.
Therefore, for almost all $\omega \in \Omega$, we take $T_\omega=1/{n}$ to be determined by the first $n$ for which 
$E^{c}_{\lambda_{0}, T = 1/n, \varphi}$ includes $\omega$, and then we obtain existence of a unique solution  $u \in Z_{[0,T_\omega]}^\omega$
from the  first step of the proof.

This completes the proof.
\end{proof}

We are now in a position to prove the result on continuous dependence on $H^s$ data, stated in Theorem~\ref{probwellposed}.
The two results, Theorem~\ref{existence} and Theorem~\ref{probwellposed}, imply probabilistic well-posedness for $-1/6 < s \le 1/2$.
This is an improvement over deterministic well-posedness, as the Cauchy problem for \eqref{quinticeq} is 
ill-posed for $0 < s < s_{cr} = 1/2$. The results in these two theorems show that the critical exponent for probabilistic well-posedness
is pushed all the way down to $s^{prob}_{cr} = -1/6$, excluding $-1/6$.

\subsubsection{Proof of Theorem~\ref{probwellposed} on continuous dependence on $H^s$ data}
We want to prove that for any fixed $-1/6 < s \le 1/2$ and $\varphi \in \mathcal{H}^{s}(\mathbb{R}^{2})$,
and for any choice of $\epsilon > 0$, $0 < T \le 1$, and $0 < p < 1$, 
there exists an event $A_{\varphi, \epsilon}$ and $\delta > 0$, 
such that for any $\omega \in A_{\varphi, \epsilon}$
the Cauchy problem for the nonlinear quintic viscous wave equation with initial data $\varphi^{\omega}$
 has a  solution 
 {{$u \in C^{0}([0, T]; H^{s}(\mathbb{R}^{2}))$}} which satisfies
$${{||u||_{C^{0}([0, T];H^{s}(\mathbb{R}^{2}))} \le \epsilon}},$$
with the probability of the event $A_{\varphi, \epsilon}$ being greater {{}{than}} $p$
whenever
$$||\varphi||_{\mathcal{H}^{s}(\mathbb{R}^{2})} < \delta.$$

\begin{proof}
Fix $-1/6 < s \le 1/2$ and take $\varphi \in \mathcal{H}^{s}(\mathbb{R}^{2})$.
We want to show that for every $\epsilon > 0$, $0 < T \le 1$, and $0 < p < 1$,
we can construct an event $A_{\varphi, \epsilon}$  and 
find $\delta > 0$, such that the nonlinear quintic viscous wave equation \eqref{quinticeq} with initial data $\varphi^{\omega}$, 
where $\omega \in A_{\varphi, \epsilon}$,  has a  solution which satisfies
${{||u||_{C^{0}([0, T];H^{s}(\mathbb{R}^{2}))} \le \epsilon}}$,
with probability $P(A_{\varphi,\epsilon}) > p$,
whenever $||\varphi||_{\mathcal{H}^{s}(\mathbb{R}^{2})} < \delta.$

For this purpose,  we recall that the probabilistic solution $u$ can be written as the sum
of the free evolution $u_\varphi^\omega$ and the inhomogeneous part $v^*$. 
Therefore, the ${C^{0}([0, T];H^{s}(\mathbb{R}^{2}))}$-norm of $u$ is bounded from above 
by the ${C^{0}([0, T];H^{s}(\mathbb{R}^{2}))}$-norm of $u_\varphi^\omega$, plus the 
${C^{0}([0, T];H^{s}(\mathbb{R}^{2}))}$-norm of $v^*$. 
We would like to find  $A_{\varphi, \epsilon}$ and $\delta$ so that the 
${C^{0}([0, T];H^{s}(\mathbb{R}^{2}))}$-norms of $u_\varphi^\omega$ and $v^*$ are each bounded 
by $\epsilon/2$, with the probability of this happening being greater than $p$, whenever 
$||\varphi||_{\mathcal{H}^{s}(\mathbb{R}^{2})} < \delta.$

We start with the inhomogeneous part $v^*$. Recall that
Theorem~\ref{existence} implies that there exists $\lambda_0$ small enough such that 
for every  $\omega \in E^{c}_{\lambda_{0}, T, \varphi}$,
$K_{\varphi}^{\omega}$ is a strict contraction on the ball of radius $2C\lambda_{0}^{5}$ in $X_T$,
where $C >0$ is the constant from inequalities \eqref{contraction1} and \eqref{contraction2}.
Note that the constant $C>0$ is independent of $\varphi \in \mathcal{H}^{s}(\mathbb{R}^{2})$.
We can now choose $\lambda_0$ so small that it also satisfies:
\begin{equation}\label{probcor4}
2C\lambda_{0}^{5} < \frac{\epsilon}{2}.
\end{equation}
With this choice of $\lambda_0$, which depends on $\epsilon$, we have that for $\omega \in E^{c}_{\lambda_{0}, T, \varphi}$, there exists a solution 
$u = u^\omega_\varphi + v^*$ for which the $X_T$-norm of $v^*$ is bounded from above by $\epsilon/2$.
Moreover, the probability of this happening can be estimated from \eqref{POmegaT}. 
Namely, \eqref{POmegaT} implies that there exists a $\tilde{C} > 0$ such that for the $\lambda_0$ above, and $0<T\le 1$, 
\begin{equation}\label{event_v*}
P(E^{c}_{\lambda_{0}, T, \varphi}) \ge 1 - \tilde{C}\lambda_{0}^{-6}||\varphi||^{6}_{\mathcal{H}^{s}(\mathbb{R}^{2})}T^{s + 1/6}
> 1-\tilde{C}\lambda_{0}^{-6}\delta^6 T^{s + 1/6},
\end{equation}
where we have used the assumption that $||\varphi||_{\mathcal{H}^{s}(\mathbb{R}^{2})} < \delta.$
This will be used a bit later to determine $\delta > 0$ such that the probability of the event $A_{\varphi,\epsilon}$,
which will {{be a subset of}} $E^{c}_{\lambda_{0}, T, \varphi}$, is greater than $p$,
as required by the theorem. 
%Notice that this $\delta$ depends on $\lambda_0$, which in turn, depends on $\epsilon$. 

%\begin{equation}\label{probcor5}
%P(E_{\lambda_{0}, T, \varphi}) \le \tilde{C}\lambda_{0}^{-6}||\varphi||^{6}_{\mathcal{H}^{s}(\mathbb{R}^{2})}T^{s + 1/6}
%< \tilde{C}\lambda_{0}^{-6}\delta^6 T^{s + 1/6}.
%\end{equation}
%Here we have used the assumption that $||\varphi||_{\mathcal{H}^{s}(\mathbb{R}^{2})} < \delta.$
% Recall that on $E^{c}_{\lambda_{0}, T, \varphi}$, there exists a solution with initial condition $\varphi^{\omega}$ in the space $Z_{T}$, 
%defined in Theorem \ref{existence}.
 
 Next, we find conditions under which the $C^{0}([0, T]; H^s(\R^{2}))$ norm of the free evolution part of the solution can be made less than $\epsilon/2$.
We claim that the free evolution part of the solution with arbitrary initial data $\varphi=(f,g)$ in ${\mathcal{H}^{s}(\mathbb{R}^{2})}$, 
can be estimated as follows:
\begin{proposition}\label{free_prop}
Let $\varphi=(f,g)\in{\mathcal{H}^{s}(\mathbb{R}^{2})}$. Then, the following estimate on the free evolution associated with $\varphi$ holds:
\begin{align*}
\Bigg|\Bigg| e^{-\frac{\sqrt{-\Delta}}{2}t}\left(\text{cos}\left(\frac{\sqrt{3}}{2}\sqrt{-\Delta}t\right) \right.
&+ 
\left.\frac{1}{\sqrt{3}}\text{sin}\left(\frac{\sqrt{3}}{2}\sqrt{-\Delta}t\right)\right)(f) 
\\
&+ e^{-\frac{\sqrt{-\Delta}}{2}t}\frac{\text{sin}\left(\frac{\sqrt{3}}{2}\sqrt{-\Delta}t\right)}{\frac{\sqrt{3}}{2}\sqrt{-\Delta}}(g)\Bigg|\Bigg|_{C^{0}([0, T]; H^{s}(\mathbb{R}^{2}))} \le C_{\rm{free}}||\varphi||_{\mathcal{H}^{s}(\mathbb{R}^{2})},
\end{align*}
where $C_{\rm{free}}$ is independent of $\varphi \in \mathcal{H}^{s}(\mathbb{R}^{2})$, and depends only on $s$.
\end{proposition}
We will prove this result after we finish the proof of the main result.

 Proposition~\ref{free_prop} implies that for each fixed $\omega$, we have:
\begin{multline}\label{probcor2}
\|u_\varphi^\omega \|_{C^{0}([0, T]; H^{s}(\mathbb{R}^{2}))} =
\Bigg|\Bigg|e^{-\frac{\sqrt{-\Delta}}{2}t}\left(\text{cos}\left(\frac{\sqrt{3}}{2}\sqrt{-\Delta}t\right) + \frac{1}{\sqrt{3}}\text{sin}\left(\frac{\sqrt{3}}{2}\sqrt{-\Delta}t\right)\right)(f^\omega) \\
\qquad + e^{-\frac{\sqrt{-\Delta}}{2}t}\frac{\text{sin}\left(\frac{\sqrt{3}}{2}\sqrt{-\Delta}t\right)}{\frac{\sqrt{3}}{2}\sqrt{-\Delta}}(g^\omega)
\Bigg|\Bigg|_{C^{0}([0, T]; H^{s}(\mathbb{R}^{2}))} 
\le C_{\rm{free}}||\varphi^\omega||_{\mathcal{H}^{s}(\mathbb{R}^{2})}.
\end{multline}
%which gives an estimate of the $ C^{0}([0, T]; H^{s}(\mathbb{R}^{2}))$-norm of free evolution $u_\varphi^\omega$ in terms of $\varphi^\omega$.
To find the conditions under which $\|u_\varphi^\omega \|_{C^{0}([0, T]; H^{s}(\mathbb{R}^{2}))} \le \epsilon/2$ whenever 
$||\varphi ||_{\mathcal{H}^{s}(\mathbb{R}^{2})}$ is small (less than $\delta$), we need to associate the smallness
of $||\varphi ||_{\mathcal{H}^{s}(\mathbb{R}^{2})}$ with the smallness of $||\varphi^\omega ||_{\mathcal{H}^{s}(\mathbb{R}^{2})}$ {{in probability}}.
Indeed, we recall from Proposition~\ref{phi_omega_estimate} and \eqref{randomprop}
 that for any $\varphi \in \mathcal{H}^{s}(\mathbb{R}^{2})$, the {{$L^{2}(\Omega; \mathcal{H}^s(\R^2))$ norm of $\varphi^\omega$ is bounded by a constant times the $\mathcal{H}^s$ norm of $\varphi$ for our given randomization satisfying the conditions of Proposition~\ref{phi_omega_estimate}}}.
 Moreover, one can estimate the probability that the $\mathcal{H}^s$ norm of $\varphi^\omega$ is smaller than a given value $\alpha > 0$ 
  in terms of $\alpha$ and the size of the $\mathcal{H}^s$ norm of $\varphi$ as follows:
  \begin{equation}\label{prob_free}
P(||\varphi^{\omega}||_{\mathcal{H}^{s}(\mathbb{R}^{2})} < \alpha) >  1 - C_{\omega}^{2}\alpha^{-2}\delta^2,
\end{equation}
whenever $||\varphi||_{\mathcal{H}^{s}(\R^{2})} < \delta$. Estimate \eqref{prob_free}  is a direct consequence of Chebyshev's inequality:
\begin{equation*}
P(||\varphi^{\omega}||_{\mathcal{H}^{s}(\mathbb{R}^{2})} \ge \alpha) \le C_{\omega}^{2}\alpha^{-2}||\varphi||_{\mathcal{H}^{s}(\mathbb{R}^{2})}^{2}
< C_{\omega}^{2}\alpha^{-2}\delta^2.
\end{equation*}
Therefore, given $\epsilon > 0$, we see from \eqref{probcor2} that there exists $\alpha_0 > 0$ depending on $\epsilon$, 
where $\alpha_0$ measures the size of $||\varphi^{\omega}||_{\mathcal{H}^{s}(\mathbb{R}^{2})}$, 
such that 
$$
\|u_\varphi^\omega \|_{C^{0}([0, T]; H^{s}(\mathbb{R}^{2}))} \le \epsilon/2,
$$
whenever $||\varphi^{\omega}||_{\mathcal{H}^{s}(\mathbb{R}^{2})} < \alpha_{0}$,
where the probability of the event that $||\varphi^{\omega}||_{\mathcal{H}^{s}(\mathbb{R}^{2})} < \alpha_0$, is bounded from below
by $1 - C_{\omega}^{2}\alpha_0^{-2}\delta^2$ whenever $||\varphi||_{\mathcal{H}^{s}(\R^{2})} < \delta$,  as specified in \eqref{prob_free}.
This will be used to determine $\delta > 0$, depending on $\epsilon$ and $p$, 
such that the probability of the event $A_{\varphi,\epsilon}$,
which will be a subset of the event $\{||\varphi^{\omega}||_{\mathcal{H}^{s}(\mathbb{R}^{2})} < \alpha_0\}$, is greater than $p$.

We are now in a position where we can combine these two steps into one. We {\bf{define our event $A_{\varphi,\epsilon}$}}
to be the intersection of the events specified by \eqref{event_v*} and \eqref{prob_free}:
\begin{equation}\label{event_A}
A_{\varphi, \epsilon} := \{||\varphi^{\omega}||_{\mathcal{H}^{s}(\mathbb{R}^{2})} < \alpha_0(\epsilon)\} 
\cap E^c_{\lambda_{0}(\epsilon), T, \varphi} 
=
 \left[\{||\varphi^{\omega}||_{\mathcal{H}^{s}(\mathbb{R}^{2})} \ge \alpha_{0}(\epsilon)\} 
\cup E_{\lambda_{0}(\epsilon), T, \varphi}\right]^{c},
\end{equation}
and {\bf{choose $\delta > 0$}} so that the probability of this event is greater than $p$:
$$
P(A_{\varphi, \epsilon}) > p.
$$

\vskip 0.1in
More precisely, given $\epsilon > 0$, $0 < T \le 1$, and $0 < p < 1$, 
there exist $\lambda_0(\epsilon)$ and $\alpha_0(\epsilon)$ that define the event 
\begin{equation*}
A_{\varphi, \epsilon} := \{||\varphi^{\omega}||_{\mathcal{H}^{s}(\mathbb{R}^{2})} < \alpha_0(\epsilon)\} 
\cap E^c_{\lambda_{0}(\epsilon), T, \varphi} ,
\end{equation*}
such that for every $\omega \in A_{\varphi, \epsilon}$ there exists a solution $u$ associated with $\varphi^\omega$,
such that
$$
\|u \|_{C^{0}([0, T]; H^{s}(\mathbb{R}^{2}))} \le \epsilon.
$$
The probability of this event is associated with the size of the initial data $||\varphi||_{\mathcal{H}^{s}(\mathbb{R}^{2})}$,
and so we can choose $\delta > 0$, depending on $\epsilon$ and $p$, so that  
$$
P(A_{\varphi, \epsilon}) > p
$$
whenever
$$
||\varphi||_{\mathcal{H}^{s}(\mathbb{R}^{2})} < \delta.
$$
The choice of $\delta$ that guarantees $P(A_{\varphi, \epsilon}) > p$ whenever
$||\varphi||_{\mathcal{H}^{s}(\mathbb{R}^{2})} < \delta$  is obtained from \eqref{prob_free} and \eqref{event_v*},
and the calculation
\begin{align*}
P(A_{\varphi, \epsilon}) &> P\left(||\varphi^{\omega}||_{\mathcal{H}^{s}(\mathbb{R}^{2})} < \alpha_0(\epsilon) \right) + P(E^c_{\lambda_{0}(\epsilon), T, \varphi}) - 1
\\
& > 1 - C_{\omega}^{2}\alpha_{0}^{-2}\delta^{2} + 1 - \tilde{C}\lambda_{0}^{-6}\delta^{6}T^{s + 1/6} - 1 >p.
\end{align*}

%%%%%%%%%%%%%%%%%
\if 1 = 0
We will now use this $E^c_{\lambda_{0}, T, \varphi}$ to construct $A_{\varphi, \epsilon}$ 
with probability $P(A_{\varphi, \epsilon}) > p$, such that for all $\omega \in A_{\varphi, \epsilon}$, 
the Cauchy problem for the quintic nonlinear viscous wave equation with initial data $\varphi^{\omega}$ has a solution 
 {{$u \in C^{0}([0, T]; H^{s}(\mathbb{R}^{2}))$}} such that 
${{||u||_{C^{0}([0, T];H^{s}(\mathbb{R}^{2}))} \le \epsilon.}}$

Notice that, so far, we have a bound of $v^*$ in terms of $\epsilon$. To obtain the bound on $u$ so that 
${{||u||_{C^{0}([0, T];H^{s}(\mathbb{R}^{2}))} \le \epsilon}}$,
we also need to bound the free evolution part of the solution. 
In doing that, we will find the conditions that determine $A_{\varphi, \epsilon}$ as a subset of $E^c_{\lambda_{0}, T, \varphi}$.

We start by claiming that the free evolution part of the solution with arbitrary initial data $\varphi=(f,g)$ in ${\mathcal{H}^{s}(\mathbb{R}^{2})}$, 
can be estimated as follows:
\begin{proposition}\label{free_prop}
Let $\varphi=(f,g)\in{\mathcal{H}^{s}(\mathbb{R}^{2})}$. Then the following estimate on the free evolution associated with $\varphi$ holds:
\begin{multline*}
\Bigg|\Bigg|e^{-\frac{\sqrt{-\Delta}}{2}t}\left(\text{cos}\left(\frac{\sqrt{3}}{2}\sqrt{-\Delta}t\right) + \frac{1}{\sqrt{3}}\text{sin}\left(\frac{\sqrt{3}}{2}\sqrt{-\Delta}t\right)\right)(f) \\
+ e^{-\frac{\sqrt{-\Delta}}{2}t}\frac{\text{sin}\left(\frac{\sqrt{3}}{2}\sqrt{-\Delta}t\right)}{\frac{\sqrt{3}}{2}\sqrt{-\Delta}}(g)\Bigg|\Bigg|_{C^{0}([0, T]; H^{s}(\mathbb{R}^{2}))} \le C_{\rm{free}}||\varphi||_{\mathcal{H}^{s}(\mathbb{R}^{2})},
\end{multline*}
where $C_{\rm{free}}$ is independent of $\varphi \in \mathcal{H}^{s}(\mathbb{R}^{2})$, and only depends on $s$.
\end{proposition}
We will prove this result after we finish the proof of the main result.

 Proposition~\ref{free_prop} then implies that for each fixed $\omega$, we have:
\begin{multline}\label{probcor2}
\|u_\varphi^\omega \|_{C^{0}([0, T]; H^{s}(\mathbb{R}^{2}))} =
\Bigg|\Bigg|e^{-\frac{\sqrt{-\Delta}}{2}t}\left(\text{cos}\left(\frac{\sqrt{3}}{2}\sqrt{-\Delta}t\right) + \frac{1}{\sqrt{3}}\text{sin}\left(\frac{\sqrt{3}}{2}\sqrt{-\Delta}t\right)\right)(f^\omega) \\
\qquad + e^{-\frac{\sqrt{-\Delta}}{2}t}\frac{\text{sin}\left(\frac{\sqrt{3}}{2}\sqrt{-\Delta}t\right)}{\frac{\sqrt{3}}{2}\sqrt{-\Delta}}(g^\omega)\Bigg|\Bigg|_{C^{0}([0, T]; H^{s}(\mathbb{R}^{2}))} \le C_{\rm{free}}||\varphi^\omega||_{\mathcal{H}^{s}(\mathbb{R}^{2})},
\end{multline}
which gives an estimate of the $ C^{0}([0, T]; H^{s}(\mathbb{R}^{2}))$-norm of free evolution $u_\varphi^\omega$ in terms of $\varphi^\omega$.

We now associate the size of $||\varphi^\omega||_{\mathcal{H}^{s}(\mathbb{R}^{2})}$ 
with the size of $||\varphi||_{\mathcal{H}^{s}(\mathbb{R}^{2})}$ as follows.
Recall from \eqref{randomprop} that for any $\varphi \in \mathcal{H}^{s}(\mathbb{R}^{2})$, the following holds:
\begin{equation}\label{probcor1}
||\varphi^{\omega}||_{L^{2}(\Omega; H^{s}(\mathbb{R}^{2}))} \le C_{\omega}||\varphi||_{\mathcal{H}^{s}(\mathbb{R}^{2})},
\end{equation}
and so, by Chebychev's inequality, we have that for any $\alpha > 0$, the probability that $||\varphi^{\omega}||_{H^{s}(\mathbb{R}^{2})}$ is large,
i.e., $P(||\varphi^{\omega}||_{H^{s}(\mathbb{R}^{2})} \ge \alpha)$
is bounded by the size of $||\varphi||_{\mathcal{H}^{s}(\mathbb{R}^{2})}$ as follows:
\begin{equation}\label{probcor3}
P(||\varphi^{\omega}||_{H^{s}(\mathbb{R}^{2})} \ge \alpha) \le C_{\omega}^{2}\alpha^{-2}||\varphi||_{\mathcal{H}^{s}(\mathbb{R}^{2})}^{2}
< C_{\omega}^{2}\alpha^{-2}\delta^2.
\end{equation}
 We can now {\bf{choose $\delta > 0$}} sufficiently small, such that the ${C^{0}([0, T]; H^{s}(\mathbb{R}^{2}))}$ norm of free evolution $u_\varphi^\omega$ 
 is less than $\epsilon/2$, and determine $A_{\varphi,\delta}$ such that all the estimates we obtained so far hold true with high probability, 
 namely, that the randomly perturbed data based on event $\omega \in A_{\varphi,\delta}$, 
 will provide well-posedness with probability greater than $p$, for arbitrary
 $0 < p < 1$.
 
 We do this in two steps. First we choose $\alpha_0$, which depends on $\epsilon$, such that 
the free evolution $u^\omega_\varphi$ associated with data $\varphi^\omega = (f^\omega,g^\omega)$ is bounded by $\epsilon/2$
whenever $||\varphi^{\omega}||_{H^{s}(\mathbb{R}^{2})} < \alpha_0$;
namely, from \eqref{probcor2} we have:
\begin{equation*}
 C_{\rm{free}} {{\alpha_0}} < \frac{\epsilon}{2}.
\end{equation*}

Then we choose $\delta > 0$, which depends on $\alpha_0$, and therefore on $\epsilon$, 
and on $\lambda_0$, which depends on $\epsilon$,
so that the probability that $||\varphi^{\omega}||_{H^{s}(\mathbb{R}^{2})} \ge \alpha_0$ and 
 $\omega \notin E^c_{\lambda_{0}, T, \varphi}$ is arbitrarily small; namely
from \eqref{probcor3} and
from \eqref{probcor5},
we get a condition on $\delta$ so that the probability of the union of the events $E^c_{\lambda_{0}, T, \varphi}$, 
and $||\varphi^{\omega}||_{H^{s}(\mathbb{R}^{2})} \ge \alpha_0$ be bounded by $1 - p$:
\begin{equation*}
C_{\omega}^{2}\alpha_{0}^{-2}\delta^{2} + \tilde{C}\lambda_{0}^{-6}\delta^{6}T^{s + 1/6} < 1 - p .
\end{equation*}

We now show that for this choice of $\delta > 0$, we can construct  $A_{\varphi, \epsilon}$ 
with the desired properties. Indeed, 
fix an arbitrary $\varphi \in \mathcal{H}^{s}(\mathbb{R}^{2})$ with $||\varphi||_{\mathcal{H}^{s}(\mathbb{R}^{2})} < \delta$. 
Define
\begin{equation*}
A_{\varphi, \delta(\epsilon)} := \left[\{||\varphi^{\omega}||_{H^{s}(\mathbb{R}^{2})} \ge \alpha_{0}(\epsilon)\} 
\cup E_{\lambda_{0}(\epsilon), T, \varphi}\right]^{c}.
\end{equation*}
From 
\begin{equation*}
P(\{||\varphi^{\omega}||_{H^{s}(\mathbb{R}^{2})} \ge \alpha_{0} \} \cup E_{\lambda_{0}, T, \varphi}) < 1 - p
\end{equation*}
we have $$P(A_{\varphi, \epsilon}) > p.$$
Now, since $A_{\varphi, \epsilon} \subset E^c_{\lambda_{0}, T, \varphi}$, there exists a solution to the quintic nonlinear viscous wave equation 
$u \in Z_T$. 
In addition, by the choice of $\lambda_0 = \lambda_0(\epsilon)$, and by the choice of $\alpha_0 = \alpha_0(\epsilon)$, we have that
$$
{{||u||_{C^{0}([0, T];H^{s}(\mathbb{R}^{2}))}}}
\le ||u_\varphi^\omega||_{C^{0}([0, T];H^{s}(\mathbb{R}^{2}))} + ||v^*||_{C^{0}([0, T];H^{s}(\mathbb{R}^{2}))} \le \epsilon.
$$
where $u = u_\varphi^\omega + v^*$.
\fi

%%%%%%%%%%%%%%%%%%%
This concludes the proof of the main part of the theorem.

\if 1 = 0
In addition, for each $\omega \in A_{\varphi, \epsilon}$, 
$$||\varphi^{\omega}||_{H^{s}(\mathbb{R}^{2})} < \alpha_{0}(\delta).$$
 So by \eqref{probcor2} and the choice of $\alpha_{0}$ so that $C_{2}\alpha_{0} < \frac{\epsilon}{2}$, the free evolution of $\varphi^{\omega}$ is in $C^{0}([0, T]; H^{s}(\mathbb{R}^{2}))$ with norm less than $\frac{\epsilon}{2}$ in this space. In addition, then there is a solution in the space $Z_{T}$ defined in the statement of Theorem \ref{existence} since $\omega \in A_{\varphi, \epsilon} \subset E_{\lambda_{0}, T, \varphi}$.

*****************************

%%%%%%%%%%%%
Let us list the relevant inequalities we will need, and denote the constants in each of the inequalities with different variables. For any $\varphi = (f, g) \in \mathcal{H}^{s}(\mathbb{R}^{2})$,
\begin{equation}\label{probcor1}
||\varphi^{\omega}||_{L^{2}(\Omega; H^{s}(\mathbb{R}^{2}))} \le C_{1}||\varphi||_{\mathcal{H}^{s}(\mathbb{R}^{2})}
\end{equation}

\fi
%%%%%%%%%%%%%%%

What remains is to prove Proposition~\ref{free_prop}.
Namely, we want to obtain the following estimate:
\begin{align*}
\Bigg|\Bigg|e^{-\frac{\sqrt{-\Delta}}{2}t}\left(\text{cos}\left(\frac{\sqrt{3}}{2}\sqrt{-\Delta}t\right) \right. 
&+ 
\left. \frac{1}{\sqrt{3}}\text{sin}\left(\frac{\sqrt{3}}{2}\sqrt{-\Delta}t\right)\right)(f) \\
&+ e^{-\frac{\sqrt{-\Delta}}{2}t}\frac{\text{sin}\left(\frac{\sqrt{3}}{2}\sqrt{-\Delta}t\right)}{\frac{\sqrt{3}}{2}\sqrt{-\Delta}}(g)\Bigg|\Bigg|_{C^{0}([0, T]; H^{s}(\mathbb{R}^{2}))} \le C_{\rm{free}}||\varphi||_{\mathcal{H}^{s}(\mathbb{R}^{2})}.
\end{align*}
This inequality follows from the fact that for $0 \le t \le T \le 1$, we have
\begin{equation*}
\left|e^{-\frac{|\xi|}{2}t}\left(\text{cos}\left(\frac{\sqrt{3}}{2}|\xi|t\right) + \frac{1}{\sqrt{3}}\text{sin}\left(\frac{\sqrt{3}}{2}|\xi|t\right)\right)\right| \le 2,
\end{equation*}
and from the following low and high frequency estimate:
\begin{multline*}
\left|\left|e^{-\frac{\sqrt{-\Delta}}{2}t}\frac{\text{sin}\left(\frac{\sqrt{3}}{2}\sqrt{-\Delta}t\right)}{\frac{\sqrt{3}}{2}\sqrt{-\Delta}}(g)\right|\right|^{2}_{H^{s}(\mathbb{R}^{2})} = \frac{1}{(2\pi)^{2}} \int_{\mathbb{R}^{2}} e^{-|\xi|t}\left(\frac{\text{sin}\left(\frac{\sqrt{3}}{2}|\xi|t\right)}{\frac{\sqrt{3}}{2}|\xi|}\right)^{2}|\widehat{g}(\xi)|^{2}(1 + |\xi|^{2})^{s} d\xi \\
\le \frac{1}{(2\pi)^{2}} \left(\int_{|\xi| \le 1} + \int_{|\xi| \ge 1} e^{-|\xi|t}\left(\frac{\text{sin}\left(\frac{\sqrt{3}}{2}|\xi|t\right)}{\frac{\sqrt{3}}{2}|\xi|}\right)^{2}|\widehat{g}(\xi)|^{2}(1 + |\xi|^{2})^{s} d\xi\right) \\
\le \frac{1}{(2\pi)^{2}}\left(\int_{|\xi| \le 1} 2t^2|\widehat{g}(\xi)|^{2}(1 + |\xi|^{2})^{s - 1} d\xi + \int_{|\xi| \ge 1} \left(\frac{2}{\sqrt{3}}\right)^{2}\left(\frac{1 + |\xi|^{2}}{|\xi|^{2}}\right)|\widehat{g}(\xi)|^{2}(1 + |\xi|^{2})^{s - 1}d\xi\right) \\
\le C_{s}\left(\int_{|\xi| \le 1} (1 + |\xi|^{2})^{s - 1}|\widehat{g}(\xi)|^{2} d\xi + \int_{|\xi| \ge 1} (1 + |\xi|^{2})^{s - 1}|\widehat{g}(\xi)|^{2} d\xi\right) = C_{\rm{free}}||g||_{H^{s - 1}(\mathbb{R}^{2})}^{2},
\end{multline*}
where in the last step, we used that $0 < t \le T \le 1$ and for $|\xi| \ge 1$, we have $\frac{1 + |\xi|^{2}}{|\xi|^{2}} \le 2$. 

Continuity in time with respect to the $H^{s}(\mathbb{R}^{2})$ norm follows similarly from uniform continuity
and a similar low and high frequency estimate.
\end{proof}

With this proof, we conclude the section in which we have shown probabilistic well-posedness for the supercritical quintic viscous wave equation,
which holds for the initial data in $H^s$, where $-1/6 < s \le s_{cr} = 1/2$.
%%%%%%%%%%%

{{}{
\begin{remark}
For concreteness, we handled the case of $p = 5$ corresponding to the nonlinear quintic viscous wave equation, since this is the smallest positive odd integer $p$ for which we get deterministic ill-posedness as described in Theorem \ref{ill_posedness}. However, one can extend these results to encompass general power nonlinearities with $p$ being a positive odd integer greater than or equal to five, for the equation
\begin{equation*}
\partial_{tt}u - \Delta u + \sqrt{-\Delta} \partial_{t}u + u^{p} = 0 \qquad \text{ on } \R^{2}.
\end{equation*}
In this case, the solution space $X_{T}$ as defined in \eqref{X_T} would change to
\begin{equation*}
X_{T} = C^{0}([0, T]; H^{1 - \frac{2}{p - 1}}) \cap L^{\frac{3}{2}(p - 1)}([0, T]; L^{\frac{3}{2}(p - 1)}(\R^{2})).
\end{equation*}
This is because we must find $(q, r)$, $(\tilde{q}, \tilde{r})$, and $s \ge 0$ satisfying the various conditions in \eqref{cond1} and \eqref{cond2}. As described in Remark \ref{quintic}, for the Strichartz estimate to work well with the power nonlinearity, we would want $\tilde{q}'$ and $\tilde{r}'$ to be $p$ times $q$ and $p$ times $r$ respectively. If we set $q = r$ for simplicity, this forces us to choose $q, r = \frac{3}{2}(p - 1)$. Then, the condition on $s$ in \eqref{cond2} forces us to choose $s = 1 - \frac{2}{p - 1}$, which we note is exactly equal to $s_{cr}$ for $n = 2$. One can carry out the remaining arguments in the section with very minor modifications to conclude a similar probabilistic well-posedness result like that of Theorem \ref{existence} for the exponents $-\frac{2}{3(p - 1)} < s \le 1 - \frac{2}{p - 1}$. 
\end{remark}
}}

\if 1 = 0
By Chebychev's inequality applied to the first inequality,
\begin{equation*}
P(||\varphi^{\omega}||_{H^{s}(\mathbb{R}^{2})} \ge \alpha) \le C_{1}^{2}\alpha^{-2}||\varphi||_{\mathcal{H}^{s}(\mathbb{R}^{2})}^{2}
\end{equation*}
for all $\alpha > 0$. 
%%%%%%%%%%%

%%%%%%%%%%%%%%%%%%%
\if 1 = 0
Consider a fixed but arbitrary choice of $\epsilon > 0$, $0 < T \le 1$, and $0 < p < 1$. Note that the constants in the mapping estimates for $K^{\omega}_{\varphi}$ in  \eqref{contraction1} and \eqref{contraction2} are independent of $\varphi \in \mathcal{H}^{s}$ for fixed $s$. Then, fix $\lambda_{0}$ sufficiently small so that $K_{\varphi}^{\omega}$ is a strict contraction on the ball of radius $2C_{3}\lambda_{0}^{5}$ for any $f \in \mathcal{H}^{s}(\mathbb{R}^{2})$, where $C_{3}$ is the constant (independent of $\varphi \in \mathcal{H}^{s}(\mathbb{R}^{2})$, and depending only on $s$) in the two inequalities \eqref{contraction1} and \eqref{contraction2} regarding $K^{\omega}_{\varphi}$, and such that
\begin{equation}\label{probcor4}
2C_{3}\lambda_{0}^{5} < \frac{\delta}{2}
\end{equation}
Now, that we have chosen $\lambda_{0}$, we have that for all $\varphi \in \mathcal{H}^{s}(\mathbb{R}^{2})$, 
\begin{equation*}
P(E^{c}_{\lambda_{0}, T, \varphi}) \ge 1 - C_{4}\lambda_{0}^{-6}||\varphi||^{6}_{\mathcal{H}^{s}(\mathbb{R}^{2})}T^{s + 1/6}
\end{equation*}
and hence
\begin{equation}\label{probcor5}
P(E_{\lambda_{0}, T, \varphi}) \le C_{4}\lambda_{0}^{-6}||\varphi||^{6}_{\mathcal{H}^{s}(\mathbb{R}^{2})}T^{s + 1/6}
\end{equation}
where $C_{4}$ is the constant from Corollary \ref{probcor}. Recall that on $E^{c}_{\lambda_{0}, T, \varphi}$, there is a solution with initial conditions $\varphi^{\omega}$ in the space $Z_{T}$ defined in Theorem \ref{existence}.
\fi
%%%%%%%%%%%%%%%%%%%%

Now, choose $\alpha_{0} > 0$ and $\delta > 0$ sufficiently small such that
\begin{equation*}
C_{2}\alpha_{0} < \frac{\epsilon}{2} \ \ \ \ \ \ \text{ because of } \eqref{probcor2}
\end{equation*}
\begin{equation*}
C_{1}^{2}\alpha_{0}^{-2}\epsilon^{2} + C_{4}\lambda_{0}^{-6}\epsilon^{6}T^{s + 1/6} < 1 - p \ \ \ \ \ \ \text{ because of } \eqref{probcor3} \text{ and } \eqref{probcor5}
\end{equation*}
Then, we need to show that for this choice of $\delta > 0$, we can construct an event $A_{\varphi, \epsilon}$ with the desired properties for each $\varphi \in \mathcal{H}^{s}(\mathbb{R}^{2})$ with $||\varphi||_{\mathcal{H}^{s}(\mathbb{R}^{2})} < \epsilon$. Fix an arbitrary $\varphi \in \mathcal{H}^{s}(\mathbb{R}^{2})$ with $||\varphi||_{\mathcal{H}^{s}(\mathbb{R}^{2})} < \epsilon$. Define
\begin{equation*}
A_{\varphi, \epsilon} := \left[\{||\varphi^{\omega}||_{H^{s}(\mathbb{R}^{2})} \ge \alpha_{0}\} \cup E_{\lambda_{0}, T, \varphi}\right]^{c}  
\end{equation*}
Note that the way we chose $\alpha_{0}$ and $\delta > 0$ implies that for $\varphi \in \mathcal{H}^{s}(\mathbb{R}^{2})$ with $||\varphi||_{\mathcal{H}^{s}(\mathbb{R}^{2})} < \delta$, 
\begin{equation*}
P(\{||\varphi^{\omega}||_{H^{s}(\mathbb{R}^{2})} \ge \alpha_{0}\} \cup E_{\lambda_{0}, T, \varphi}) < 1 - p
\end{equation*}
so that $P(A_{\varphi, \epsilon}) > p$. 

Next, we show that $A_{\varphi, \epsilon}$ has the desired properties. Since $A_{\varphi, \epsilon} \subset E_{\lambda_{0}, T, \varphi}$, there exists a solution to the quintic nonlinear viscous wave equation with initial data $\varphi^{\omega}$ for each $\omega \in A_{\varphi, \epsilon}$. In addition, for each $\omega \in A_{\varphi, \epsilon}$, $||\varphi^{\omega}||_{H^{s}(\mathbb{R}^{2})} < \alpha_{0}$. So by \eqref{probcor2} and the choice of $\alpha_{0}$ so that $C_{2}\alpha_{0} < \frac{\epsilon}{2}$, the free evolution of $\varphi^{\omega}$ is in $C^{0}([0, T]; H^{s}(\mathbb{R}^{2}))$ with norm less than $\frac{\epsilon}{2}$ in this space. In addition, then there is a solution in the space $Z_{T}$ defined in the statement of Theorem \ref{existence} since $\omega \in A_{\varphi, \epsilon} \subset E_{\lambda_{0}, T, \varphi}$.

We already showed that the first part of the description of $Z_{T}$ (the free evolution) is in $C^{0}([0, T]; H^{s}(\mathbb{R}^{2}))$ with norm less than $\frac{\epsilon}{2}$. Because the map $K_{\varphi}^{\omega}$ is a strict contraction on the ball of radius $2C_{3}\lambda_{0}^{5}$ in $X_{T}$, we have that the second part of the solution in the decomposition given in $Z_{T}$ is in $C^{0}([0, T]; H^{1/2}(\mathbb{R}^{2}))$ and hence $C([0, T]; H^{s}(\mathbb{R}^{2}))$ also, with norm less than or equal to $2C_{3}\lambda_{0}^{5} < \frac{\epsilon}{2}$. So the solution to the Cauchy problem with initial data $\varphi^{\omega}$ exists in $C([0, T]; H^{s}(\mathbb{R}^{2}))$ with norm less than $\epsilon$ for all $\omega \in A_{\epsilon, \varphi}$ with $P(A_{\epsilon, \varphi}) > p$, whenever $||\varphi||_{H^{s}(\mathbb{R}^{2})} < \delta$. 
\fi

\section{Appendix}\label{appendix}

In the appendix, we provide the proof of the local existence result that was used in the proof of deterministic ill-posedness for $0 < s < s_{cr}$ in Sec. \ref{illposed}. Specifically, we recall that this proof relied crucially on the result of Proposition \ref{closeness}, which states that the solution $\phi(t, x)$ to the initial value problem for  \eqref{IVPdispVNLWE} is close to the solution $\phi^{(0)}(t, x)$ to the initial value problem for the visco-dispersive limit \eqref{IVPdisplimit} in $H^{k}(\mathbb{R}^{n})$ norm for all visco-dispersive parameters $\nu$ sufficiently close to $0$, and for a given range of times $t$. In the proof of Proposition \ref{closeness}, we did not explicitly justify why the initial value problem in \eqref{IVPproof} indeed has a solution that exists and is unique for the times for which we perform our analysis. This is what we establish in this appendix in the following lemma.

{\bf Note on notation.} In what follows, we use multi-indices $\alpha$ in differentiation operators $\partial^{\alpha}_{x}$ to represent differentiations with respect to spatial variables and explicitly write out $\partial_{t}$ for any differentiations with respect to time. In particular, \textit{multi-indices $\alpha$ will never be used to represent any derivatives involving time}. In addition, we will use the shorthand notation $w'$ to denote the \textit{spacetime} gradient of $w$, 
\begin{equation}\label{spacetime}
w':= (\partial_t w, \partial_{x_1} w, ... , \partial_{x_n} w),
\end{equation}
which includes the derivative of $w$ with respect to each spatial variable and also the derivative of $w$ with respect to time. 

\begin{lemma}\label{existunique}
Let $k \ge n + 1$ be an integer, let $p > 1$ be a positive odd integer, and let $0 < \nu \le 1$. For initial data $(f, g) \in H^{k + 1}(\mathbb{R}^{n}) \times H^{k}(\mathbb{R}^{n})$, consider the initial value problem on $\R^{n}$,
\begin{align}\label{appendixIVP}
\partial_{tt}w - \nu^{2}\Delta w + \nu \sqrt{-\Delta} \partial_{t}w = \nu^{2}&\Delta \phi^{(0)} - G(\phi^{(0)} + w) + G(\phi^{(0)}) - \nu\sqrt{-\Delta}\partial_{t}\phi^{(0)},\\
w(0, x) &= f(x), \qquad \partial_{t}w(0, x) = g(x),
\nonumber
\end{align}
where $G(z) = z^{p}$, and $\phi^{(0)}(t, x)$, given by \eqref{displimitexpsol}, is the solution to the dispersive limit problem given by \eqref{IVPdisplimit} for fixed initial displacement $\phi_{0} \in C_{0}^{\infty}(\R^{n})$ and zero initial velocity. This initial value problem \eqref{appendixIVP} has a unique solution $(w, w')$ in $C([0, T]; H^{k + 1}(\R^{n})) \times C([0, T]; H^{k}(\R^{n}))$ for some $T > 0$ sufficiently small. Furthermore, if we let $T_{*}$ be the supremum of all such times for which this is true, then either $T_{*} = \infty$, or $T_{*}$ is finite and
\begin{equation*}
\sup_{0 \le t < T_{*}} \left(\sum_{|\alpha| \le k + 1} ||\partial^{\alpha}w(t, \cdot)||_{L^{2}(\mathbb{R}^{n})} + \sum_{|\alpha| \le k} ||\partial^{\alpha}w_{t}(t, \cdot)||_{L^{2}(\mathbb{R}^{n})}\right) = \infty
\end{equation*}
\end{lemma}

Intuitively, this lemma says that there is a solution to the initial value problem \eqref{appendixIVP} which we can extend locally in time, as long as the $H^{k + 1} \times H^{k}$ norm of the solution at a given time is bounded.

The proof below follows standard Picard iteration and energy methods, as one can find in Sogge \cite{S} in the context of quasilinear wave equations. We divide the proof into the following steps. 

\vskip 0.1in

\noindent
{\bf Step 1: Energy Inequality}

\vskip 0.1in
The first step is to obtain an energy inequality for the problem
\begin{equation*}
w_{tt} - \nu^{2} \Delta w + \nu \sqrt{-\Delta} \partial_{t}w = F(t, x),
\end{equation*}
where we emphasize that $\nu \in (0, 1]$ is fixed. We have already done this in the proof of Proposition \ref{closeness}, so we simply tailor the result of that proof to fit our current problem. 

We first start by defining the $\nu$-wave energy of the solution $w$ by
\begin{equation}\label{appendixenergy}
E_{\nu}(w(t)) := \int_{\mathbb{R}^{2}} \frac{1}{2}|w_{t}(t, x)|^{2} + \frac{\nu^{2}}{2}|\nabla w(t, x)|^{2} dx = \frac{1}{2}\|w_{t}(t, \cdot)\|_{L^{2}}^{2} + \frac{\nu^{2}}{2}\|\nabla w(t, \cdot)\|_{L^{2}}^{2}.
\end{equation}
The corresponding energy inequality in the proof of Proposition \ref{closeness} shows that for the $k$th $\nu$-wave energy defined as
\begin{equation}\label{kenergy}
E_{\nu, k}(w(t)) = \sum_{|\alpha| \le k} E_{\nu}(\partial^{\alpha}_{x}w(t)),
\end{equation}
which includes the energy of all derivatives of order less than or equal to $k$, we have that 
\begin{equation*}
\partial_{t}(E_{\nu, k}^{1/2}(w(t))) \le C||F(t, \cdot)||_{H^{k}}.
\end{equation*}
Hence,
\begin{equation}\label{step1energyest1}
E_{\nu, k}^{1/2}(w(t)) \le E_{\nu, k}^{1/2}(w(0)) + C\int_{0}^{t} ||F(\tau, \cdot)||_{H^{k}} d\tau
\end{equation}
by integrating. 

We want to estimate our solutions to \eqref{appendixIVP} with $(w, w') \in H^{k + 1} \times H^{k}$, so in particular, 
our $k$th $\nu$-wave energy defined by \eqref{appendixenergy} and \eqref{kenergy} is missing a term that estimates $||w(t, \cdot)||_{L^{2}}$. However, we can estimate $||w(t, \cdot)||_{L^{2}}$, 
which is subsumed in $\|w(t, \cdot)\|_{H^{k}}$, by using the fundamental theorem of calculus:
\begin{equation}\label{step1energyest2}
\|w(t, \cdot)\|_{H^{k}} \le \|w(0, \cdot)\|_{H^{k}} + \int_{0}^{t} \|\partial_{t}w(\tau, \cdot)\|_{H^{k}} d\tau.
\end{equation}
Therefore, by \eqref{step1energyest1} and \eqref{step1energyest2},
\begin{equation}\label{step1energyest3}
\|w(t, \cdot)\|_{H^{k}} + E_{\nu, k}^{1/2}(w(t)) \le \|w(0, \cdot)\|_{H^{k}} + E_{\nu, k}^{1/2}(w(0)) + \int_{0}^{t} \|\partial_{t}w(\tau, \cdot)\|_{H^{k}} d\tau + C\int_{0}^{t} \|F(\tau, \cdot)\|_{H^{k}} d\tau.
\end{equation}

This expression \eqref{step1energyest3} gives the desired energy inequality, as $\|w(t, \cdot)\|_{H^{k}} + E_{\nu, k}^{1/2}(w(t))$ is the appropriate energy for estimating $(w, w')$ in $H^{k + 1} \times H^{k}$. We can rewrite the energy $\|w(t, \cdot)\|_{H^{k}} + E_{\nu, k}^{1/2}(w(t))$ in the more concise form, 
\begin{equation}\label{energydef}
\mathcal{E}(w(t)) := ||\partial_{t}w(t, \cdot)||_{H^{k}} + \sum_{|\alpha| \le 1} ||\partial^{\alpha}_{x}w(t, \cdot)||_{H^{k}},
\end{equation}
which is equivalent to $\|w(t, \cdot)\|_{H^{k}} + E_{\nu, k}^{1/2}(w(t))$ \textit{since, as we emphasize, $\nu \in (0, 1]$ is fixed}. Then, by \eqref{step1energyest3}, we conclude that for some constant $C$ independent of $w$ and $t$ (which depends on the choice of $k$ and $\nu$), 
\begin{equation*}
\mathcal{E}(w(t)) \le C\left(\mathcal{E}(w(0)) + \int_{0}^{t} \mathcal{E}(w(\tau)) d\tau + \int_{0}^{t}||F(\tau, \cdot)||_{H^{k}} d\tau\right).
\end{equation*}
Then, 
\begin{equation*}
\mathcal{E}(w(t)) \le Ce^{Ct}\left(\mathcal{E}(w(0)) + \int_{0}^{t} ||F(\tau, \cdot)||_{H^{k}} d\tau\right)
\end{equation*}
by Gronwall's inequality. By taking 
\begin{equation}\label{CT}
C_{T} := Ce^{CT}
\end{equation}
where $C_{T}$ depends on $T$, we have that for all $0 \le t \le T$, 
\begin{equation}\label{energyineq}
\mathcal{E}(w(t)) \le C_{T}\left(\mathcal{E}(w(0)) + \int_{0}^{t} ||F(\tau, \cdot)||_{H^{k}}d\tau\right),
\end{equation}
where $T$ is an arbitrary time in $0 < T < \infty$. Here, $C_{T}$ depends on $T$ and furthermore, $C_{T}$ can be chosen to be strictly increasing in $T$, as seen by the definition \eqref{CT}. 

\vskip 0.1in
\noindent
{\bf Step 2: Picard iteration and uniform boundedness of iterates}

\vskip 0.1in

Now that we have an energy inequality, the general strategy will be to show that our initial value problem has local existence for well-behaved initial data $(f, g)$. Then, we will use the energy inequality \eqref{energyineq} to extend this to more general data. 

In particular, we first assume that $f, g \in \mathcal{S}(\mathbb{R}^{n})$ and we wish to show that we have local existence. We use Picard iteration. Define the first iterate as $w_{-1} := 0$. Picard iteration then defines a sequence of solutions inductively, where we hope that the iterates converge to an actual solution. Explicitly, for $m \ge 0$, we define $w_{m}$ to be the solution to the initial value problem
\begin{align}\label{picardIVP}
\partial_{tt}w - \nu^{2} \Delta w + \nu \sqrt{-\Delta} \partial_{t}w = \nu^{2} \Delta &\phi^{(0)} - G(\phi^{(0)} + w_{m - 1}) + G(\phi^{(0)}) - \nu \sqrt{-\Delta}\partial_{t}\phi^{(0)}, \\
w(0, x) &= f(x), \qquad \partial_{t}w(0, x) = g(x).
\nonumber
\end{align}
This has a solution in all of time, as can be seen by Fourier methods, because the inhomogeneous term no longer depends on the solution $w$ anymore. So we can indeed define $w_{m}$ by this inductive procedure.

We will consider the following $m$th step energy:
\begin{equation}\label{menergydef}
\mathcal{E}(w_{m}(t)) := \|\partial_{t}w_{m}(t, \cdot)\|_{H^{k}} + \sum_{|\alpha| \le 1} \|\partial^{\alpha}_{x}w_{m}(t, \cdot)\|_{H^{k}}.
\end{equation}
where $\mathcal{E}$ is the energy in \eqref{energydef}. We want to show that there exists a time $T$ and a constant $A$ such that for all $m$, the following uniform energy estimate holds:
\begin{equation}\label{uniformbound}
\mathcal{E}(w_{m}(t)) \le A < \infty, \qquad \text{ for all } 0 \le t \le T.
\end{equation}
This is the content of Step 2. We will show this by induction. 

For the base case, estimate \eqref{uniformbound} is clearly true for $u_{-1} \equiv 0$ for any positive $A$, where we will choose $A$ later in the inductive step. To establish the inductive step, we will use the energy inequality. Before using the energy inequality, we collect some estimates from the proof of Proposition \ref{closeness} that we will need. By the previous estimates \eqref{energyest1} and \eqref{energyest3}, there exists a constant $C$ (depending on $k$ and our fixed $\nu$) such that
\begin{equation*}
\|\nu^{2} \Delta \phi^{(0)}\|_{H^{k}} \le C(1 + t)^{C},
\end{equation*}
\begin{equation*}
\|\nu \sqrt{-\Delta}\partial_{t}\phi^{(0)}\|_{H^{k}} \le C(1 + t)^{C}.
\end{equation*}
In addition since $k > 1$, we can appeal to the previous estimate \eqref{CCTest} to obtain
\begin{equation*}
\|G(\phi^{(0)} + w_{m - 1})(t) - G(\phi^{(0)})(t)\|_{H^{k}} \le C(1 + t)^{C}\|w_{m - 1}(t, \cdot)\|_{H^{k}}\left(1 + \|w_{m - 1}(t, \cdot)\|_{H^{k}}\right)^{p - 1}.
\end{equation*}

Now that we have collected all of the necessary estimates, we use the energy inequality \eqref{energyineq} to estimate $\mathcal{E}(w_{m}(t))$ as
\begin{align}\label{inductiveenergy}
\mathcal{E}(w_{m}(t)) &\le C_{T}\left(\mathcal{E}(w_{m}(0)) + \int_{0}^{t} (C(1 + s)^{C} + C(1 + s)^{C}\|w_{m - 1}(s, \cdot)\|_{H^{k}}\left(1 + \|w_{m - 1}(s, \cdot)\|_{H^{k}}\right)^{p - 1}ds\right) \nonumber \\
&\le C_{T}\left(\mathcal{E}(w_{m}(0)) + (C + CA(1 + A)^{p - 1}) \int_{0}^{t} (1 + s)^{C} ds\right),
\end{align}
where we used the inductive assumption in the second inequality. Now, we will choose $A$ to be any positive number such that 
\begin{equation}\label{Achoice}
A > C_{1}\mathcal{E}(w_{m}(0)),
\end{equation}
where $C_{1}$ is the constant $C_{T}$ for $T = 1$ defined by \eqref{CT}. This is possible because the right hand side of \eqref{Achoice} is the same for all $m$ since $w_{m}$ all have the same initial data $(f, g)$. Upon choosing $A$, we can then choose $0 < T < 1$ sufficiently small so that
\begin{equation}\label{Tchoice1}
C_{1}\left(\mathcal{E}(w_{m}(0)) + (C + CA(1 + A)^{p - 1})\int_{0}^{T}(1 + s)^{C}ds\right) < A.
\end{equation}
Then because $C_{T}$ is strictly increasing in $T$ by \eqref{CT}, we get from \eqref{inductiveenergy} and \eqref{Tchoice1} that $\mathcal{E}(w_{m}(t)) \le A$ for all $0 \le t \le T$. So for this choice of $A$ and $0 < T < 1$, we have the desired uniform bound \eqref{uniformbound} on the $m$th step energies of our iterates $w_{m}$.

\vskip 0.1in
\noindent
{\bf Step 3: Convergence to a solution}

\vskip 0.1in

Next, we show that the iterates $w_{m}$ from Step 2 form a Cauchy sequence in $C([0, T_{1}]; H^{k + 1}(\R^{n})) \cap C([0, T_{1}]; H^{k}(\R^{n}))$, for some time $0 < T_{1} < T$ where $T$ is the time chosen in Step 2 in \eqref{Tchoice1}.

First note that since $f, g \in \mathcal{S}$, the $w_{m}$ are all smooth since the inhomogeneous term is smooth and rapidly decreasing. Using the notation $\mathcal{E}$ from \eqref{energydef}, it suffices to show that for all $t \in [0, T_{1}]$, 
\begin{align}\label{cauchybound}
\mathcal{E}((w_{m} - w_{m - 1})(t)) &:=  \|\partial_{t}w_{m}(t, \cdot) - \partial_{t}w_{m - 1}(t, \cdot)\|_{H^{k}} 
\nonumber
\\
&+ \sum_{|\alpha| \le 1} \|\partial^{\alpha}_{x}w_{m}(t, \cdot) - \partial^{\alpha}_{x}w_{m - 1}(t, \cdot)\|_{H^{k}} = O(2^{-m}),
\end{align}
for some time $0 < T_{1} < T$, where $T$ is chosen in Step 2 in \eqref{Tchoice1}. Let $A$, chosen in the previous step in \eqref{Achoice}, be the constant in this $O(2^{-m})$. Thus, we claim that there exists a time $T_{1} > 0$, smaller than the $T$ from \eqref{Tchoice1}, such that
\begin{equation}\label{cauchybound2}
\mathcal{E}((w_{m} - w_{m - 1})(t)) \le A2^{-m} \qquad \text{ for all } t \in [0, T_1], m \ge 0.
\end{equation}
As before, we will show this by induction. This inequality \label{cauchybound2} is indeed true for $m = 0$ because $w_{-1} = 0$, and thus the inequality for $m = 0$ follows directly from \eqref{uniformbound}. 

For the inductive step, note that $w_{m + 1} - w_{m}$ is a solution to
\begin{align}\label{differenceIVP}
\partial_{tt}v - \nu^2 \Delta v + \nu \sqrt{-\Delta}\partial_{t}v = &-G(\phi^{(0)} + w_{m}) + G(\phi^{(0)} + w_{m - 1}),\\
v(0, x) &= 0, \qquad \partial_{t}v(0, x) = 0.
\nonumber
\end{align}
By using the same inequality on pg.~11 of Christ, Colliander, and Tao \cite{CCT} used to establish the inequality in \eqref{CCTest}, we get that
\begin{align}\label{modifiedCCTest}
&\|G(\phi^{(0)} + w_{m}) - G(\phi^{(0)} + w_{m - 1})\|_{H^{k}} \\
&\le C((1 + t)^{C} + A)^{C} \|w_{m}(t, \cdot) - w_{m - 1}(t, \cdot)\|_{H^{k}}(1 + \|w_{m}(t, \cdot) - w_{m - 1}(t, \cdot)\|_{H^{k}})^{p - 1}.
\nonumber
\end{align}

Then, by the energy inequality \eqref{energyineq}, for all $0 \le t \le T$ for $T$ as in \eqref{Tchoice1}, 
\begin{align}\label{step3ineq}
&\mathcal{E}((w_{m + 1} - w_{m})(t)) \nonumber \\
&\le C_{T}\left(C\int_{0}^{t} ((1 + s)^{C} + A)^{C}\left(1 + \mathcal{E}((w_{m} - w_{m - 1})(s))\right)^{p - 1}\mathcal{E}((w_{m} - w_{m - 1})(s)) ds\right),
\end{align}
since there is zero initial data in \eqref{differenceIVP}. Choose $0 < T_1 < T$ such that
\begin{equation}\label{T1choice}
C_{T_{1}}CT_{1}((1 + T_1)^{C} + A)^{C}(1 + A)^{p - 1} < 1/2,
\end{equation}
where $C_{T}$ is defined by \eqref{CT}. Then, for all $m \ge 0$, 
\begin{equation}\label{T1choiceresult}
C_{T_1}CT_1((1 + T_1)^{C} + A)^{C}(1 + A2^{-m})^{p - 1}A2^{-m} < A2^{-m - 1}.
\end{equation}
Then, by \eqref{step3ineq} and \eqref{T1choiceresult}, we get that
\begin{equation*}
\mathcal{E}((w_{m + 1} - w_{m})(t)) \le A2^{-(m + 1)} \qquad \text{ for all } 0 \le t \le T_1,
\end{equation*}
which establishes the inductive step. Hence, we have proved \eqref{cauchybound}. 

So $(w_{m}, w_{m}')$, where the prime denotes the spacetime gradient \eqref{spacetime},
forms a Cauchy sequence in $C([0, T_1]; H^{k + 1}) \times C([0, T_1]; H^{k})$. Therefore, $w_{m} \to w$ for some $w \in C([0, T_1]; H^{k + 1})$ and $w_{m}' \to v$ for some $v \in C([0, T_1]; H^{k})$. Hence $w_{m} \to w$ in $\mathcal{D}'([0, T_1) \times \R^{n})$ and hence $w_{m}' \to w'$ in $\mathcal{D}'([0, T_1) \times \R^{n})$ also. But by uniqueness, this means that $w' = v$. Therefore, $(w, w') \in C([0, T_1]; H^{k + 1}) \times C([0, T_1]; H^{k})$. 

Note in particular that
\begin{equation}\label{uniformboundsol}
\|\partial_{t}w(t, \cdot)\|_{H^{k}} + \sum_{|\alpha| \le 1} \|\partial^{\alpha}w(t, \cdot)\|_{H^{k}} \le A \qquad \text{ for all } 0 < t \le T_1,
\end{equation}
since this is true for each of the $w_{m}$ by the uniform bound \eqref{uniformbound} in Step 2. \textit{In addition, we emphasize that the constants $A$ and $T_1$ here depend only on the $H^{k + 1} \times H^{k}$ norm of the initial data $(f, g) \in \mathcal{S}(\R^{n}) \times \mathcal{S}(\R^{n})$.} This fact will be important in the next step. 

It remains to show that $w$, constructed as the limit of the iterates $w_{m}$, solves the original initial value problem \eqref{appendixIVP}. Because the inhomogeneous term in \eqref{appendixIVP} depends on the solution $w$ itself, we need to check that this inhomogeneous term converges appropriately as $m \to \infty$ to conclude that $w$ is a weak solution. To see this, for all $0 \le t \le T_1$, by the argument on pg.~11 in Christ, Colliander, and Tao \cite{CCT}, 
\begin{multline*}
\|G(\phi^{(0)} + w)(t) - G(\phi^{(0)} + w_{m})(t)\|_{H^{k}} \\
\le C((1 + t)^{C} + A)^{C} \|w(t, \cdot) - w_{m}(t, \cdot)\|_{H^{k}}(1 + \|w(t, \cdot) - w_{m}(t, \cdot)\|_{H^{k}})^{p - 1} \to 0,
\end{multline*}
where this convergence as $m \to \infty$ happens uniformly in $t \in [0, T_1]$ (since the convergence of $w_{m} \to w$ happens in $C([0, T_1]; H^{k + 1})$). Therefore, $G(\phi^{(0)} + w_{m}) \to G(\phi^{(0)} + w)$ in $L^{\infty}([0, T_1]; H^{k})$ and hence this convergence happens in the sense of weak convergence of distributions $\mathcal{D}'([0, T_1) \times \R^{n})$. So we conclude that $w$ is indeed a weak solution to the initial value problem \eqref{appendixIVP}.

\vskip 0.1in
\noindent
{\bf Step 4: Approximation argument for general initial data}

\vskip 0.1in

Now, we consider the initial value problem
\begin{align}\label{generalIVP}
\partial_{tt}w - \nu^{2}\Delta w + \nu \sqrt{-\Delta} \partial_{t}w = \nu^{2}&\Delta \phi^{(0)} - G(\phi^{(0)} + w) + G(\phi^{(0)}) - \nu\sqrt{-\Delta}\partial_{t}\phi^{(0)},\\
w(0, x) &= f(x), \qquad \partial_{t}w(0, x) = g(x),
\nonumber
\end{align}
for general $f, g \in H^{k + 1}(\R^{n}) \times H^{k}(\R^{n})$ and not just $f, g$ that are in $\mathcal{S}(\R^{n})$. For this more general class of initial data, we show that we still have local existence by approximating this general initial data by functions in $\mathcal{S}(\R^{n})$. 

Let $f_{m} \to f$ and $g_{m} \to g$ in $H^{k + 1}$ and $H^{k}$, where $f_{m}, g_{m} \in \mathcal{S}(\R^{n})$. By refining this sequence as necessary, 
we can choose a subsequence (which we will continue to denote by $f_m$ and $g_m$)
so that
\begin{equation*}
\|f_{m} - f_{m - 1}\|_{H^{k + 1}} \le \frac{1}{n + 2}(\max(2, C_{1}))^{-m}, \qquad \|g_{m} - g_{m - 1}\|_{H^{k}} \le \frac{1}{n + 2}(\max(2, C_{1}))^{-m},
\end{equation*}
where $C_{1}$ is the constant $C_{T}$ \eqref{CT} for $T = 1$ from the energy inequality \eqref{energyineq}. Since the norms of $f_{m}$ and $g_{m}$ are uniformly bounded in $H^{k + 1}$ and $H^{k}$ respectively, from our previous step, there exist uniform constants $A$ and 
\begin{equation}\label{T2}
0 < T_2 < 1, 
\end{equation}
(since the norms of $(f_{m}, g_{m}) \in H^{k + 1} \times H^{k}$ are all uniformly bonded in $m$, see the remark immediately following \eqref{uniformboundsol}) such that:
\begin{itemize}
\item There exists a solution $(w_{m}, w_{m}') \in C([0, T_2]; H^{k + 1}) \times C([0, T_2]; H^{k})$ (with $T_2$ uniform) to the initial value problem 
\begin{align}\label{step4ivp}
\partial_{tt}w - \nu^{2}\Delta w + \nu \sqrt{-\Delta} \partial_{t}w = \nu^{2}&\Delta \phi^{(0)} - G(\phi^{(0)} + w) + G(\phi^{(0)}) - \nu\sqrt{-\Delta}\partial_{t}\phi^{(0)},\\
w(0, x) &= f_{m}(x), \qquad \partial_{t}w(0, x) = g_{m}(x).
\nonumber
\end{align}
\item The solutions $w_{m}$ satisfy the uniform energy bound
\begin{equation}\label{umbound}
\|\partial_{t}w_{m}(t, \cdot)\|_{H^{k}} + \sum_{|\alpha| \le 1} \|\partial^{\alpha}_{x}w_{m}(t, \cdot)\|_{H^{k}} \le A \qquad \text{ for all } 0 \le t \le T_2,
\end{equation}
by \eqref{uniformboundsol}.
\end{itemize}
We want to show that 
\begin{equation}\label{step4claim}
(w_{m}, w_{m}') \text{ is Cauchy in } C([0, T_3]; H^{k + 1}) \times C([0, T_3]; H^{k}) \text{ for some time } 0 < T_3 < T_2.
\end{equation}

To do this, consider 
\begin{equation}\label{step4energy}
\mathcal{E}((w_{m} - w_{m - 1})(t)) := \|\partial_{t} w_{m}(t, \cdot) - \partial_{t} w_{m - 1}(t, \cdot)\|_{H^{k}} + \sum_{|\alpha| \le 1} \|\partial^{\alpha}_{x}w_{m}(t, \cdot) - \partial^{\alpha}_{x}w_{m - 1}(t, \cdot)\|_{H^{k}}.
\end{equation}
Here, $\mathcal{E}$ is the energy from \eqref{energydef} and $w_{m}$ is defined by \eqref{step4ivp}. Then, $w_{m} - w_{m - 1}$ satisfies the initial value problem
\begin{align}\label{step4ivp2}
v_{tt} - \nu^{2} \Delta v + \nu \sqrt{-\Delta}\partial_{t}v = -G(&\phi^{(0)} + w_{m}) + G(\phi^{(0)} + w_{m - 1}),\\
v(0, x) = f_{m}(x) - f_{m - 1}(x), \ \ \partial_{t}&v(0, x) = g_{m}(x) - g_{m - 1}(x),
\nonumber
\end{align}
where we recall that by the choice of our subsequences $f_{m}$ and $g_{m}$, 
\begin{equation}\label{subseqchoice}
\|f_{m} - f_{m - 1}\|_{H^{k + 1}} \le \frac{1}{n + 2}\rho^{-m}, \qquad \|g_{m} - g_{m - 1}\|_{H^{k}} \le \frac{1}{n + 2}\rho^{-m}, \quad \text{ where } \rho := \max(2, C_{1}).
\end{equation}

We claim that for some $T_{3}$ such that $0 < T_{3} < T_{2}$, where $T_{2}$ satisfies the conditions in \eqref{step4ivp} and \eqref{umbound},
\begin{equation}\label{step4claim}
\mathcal{E}((w_{m} - w_{m - 1})(t)) \le \rho^{-m + 1} \ \ \ \ \ \text{ for all } m \ge 2, \ \ 0 \le t \le T_3.
\end{equation}
We show this by using a bootstrap argument.

\if 1 = 0

We first consider $m = 2$ before proceeding to general $m$, which requires a bootstrap argument. 

For $m = 2$, recall that $w_{2} - w_{1} \in C([0, T_2]; H^{k + 1})$. We calculate that 
\begin{equation*}
\mathcal{E}((w_{2} - w_{1})(0)) \le \rho^{-2},
\end{equation*}
since we are in $\R^{n}$ and hence the sum over $|\alpha| \le 1$ in \eqref{step4energy} has $n + 1$ terms. So by the way we chose $f_{m}$ and $g_{m}$ in \eqref{subseqchoice} and the continuity of $w_{2} - w_{1}$, for some $0 < T_{2}' < T_{2}$,
\begin{equation}\label{T2primechoice}
\mathcal{E}((w_{2} - w_{1})(t)) \le \rho^{-1} \qquad \text{ for all } t \in [0, T_{2}'],
\end{equation}
as $\rho^{-2} < \rho^{-1}$.

\fi 

We consider $m \ge 2$. By applying the energy inequality \eqref{energyineq}, the uniform bound \eqref{umbound} on the $w_{m}$, 
and the previous estimate \eqref{modifiedCCTest} to the initial value problem in \eqref{step4ivp2}, 
\begin{align*}
\mathcal{E}((w_{m} - w_{m - 1})(t)) &\le C_{T_2}\bigg(\mathcal{E}((w_{m} - w_{m - 1})(0) \\
&\left. + C\int_{0}^{t} ((1 + s)^{C} + A)^{C}(1 + \mathcal{E}((w_{m} - w_{m - 1})(s)))^{p - 1}\mathcal{E}((w_{m} - w_{m - 1})(s)) ds\right),
\end{align*}
for all $0 \le t \le T_2$ for $T_2$ defined as in \eqref{step4ivp} and \eqref{umbound}. Note that by the way we chose $f_{m}$ and $g_{m}$ in \eqref{subseqchoice}, we have $\mathcal{E}((w_{m} - w_{m - 1})(0)) \le \rho^{-m}$. Therefore, for all $0 \le t \le T_2$,
\begin{align*}
&\mathcal{E}((w_{m} - w_{m - 1})(t)) \\
&\le C_{T_2}\left(\rho^{-m} + C\int_{0}^{t} ((1 + s)^{C} + A)^{C}(1 + \mathcal{E}((w_{m} - w_{m - 1})(s)))^{p - 1}\mathcal{E}((w_{m} - w_{m - 1})(s)) ds\right).
\end{align*}

Let us make the \textit{bootstrap assumption} that
\begin{equation}\label{bootstrap}
\mathcal{E}((w_{m} - w_{m - 1})(t)) \le 1 \qquad \text{ for all } 0 \le t \le T_2',
\end{equation}
where $T_2' < T_2$ will be chosen later (\textit{independently} of $m$). Then, for $t \in [0, T_{2}']$, 
\begin{equation*}
\mathcal{E}((w_{m} - w_{m - 1})(t)) \le C_{T_2'}\left(\rho^{-m} + C\int_{0}^{t} ((1 + s)^{C} + A)^{C}2^{p - 1}\mathcal{E}((w_{m} - w_{m - 1})(s)) ds\right).
\end{equation*}
After consolidating constants and using the fact that $0 < T_{2}' < T_2 < 1$ (see \eqref{T2}),
\begin{equation*}
\mathcal{E}((w_{m} - w_{m - 1})(t)) \le C_{T_{2}'}\left(\rho^{-m} + C'\int_{0}^{t} \mathcal{E}((w_{m} - w_{m - 1})(s)) ds\right), \quad \text{ for all } t \in [0, T_{2}'].
\end{equation*}
Using Gronwall's inequality, we get that for $t \in [0, T_{2}']$:
\begin{equation*}
\mathcal{E}((w_{m} - w_{m - 1})(t)) \le C_{T_2'}\rho^{-m}\exp(C'C_{T_2'}t).
\end{equation*}
Then, we can choose the $T_{2}' < T_2$ appearing in the bootstrap assumption \eqref{bootstrap} so that
\begin{equation*}
C_{T_{2}'}\exp(C'C_{T_{2}'}T_{2}') \le \rho = \max(2, C_{1}).
\end{equation*}
This is possible because the constant $C_{T}$ \eqref{CT} in the energy inequality \eqref{energyineq} is strictly increasing in $T$. Thus, we get that 
\begin{equation}\label{step4inductiveresult}
\mathcal{E}((w_{m} - w_{m - 1})(t)) \le \rho^{-m + 1} \qquad \text{ for all } t \in [0, T_{2}'].
\end{equation}
Since $m \ge 2$ and $\rho \ge 2$, this also closes the bootstrap assumption \eqref{bootstrap}, since we have shown that if $\mathcal{E}((w_{m} - w_{m - 1})(t)) \le 1$ for all $t \in [0, T_{2}']$, then
\begin{equation*}
\mathcal{E}((w_{m} - w_{m - 1})(t)) \le \rho^{-m + 1} \le \rho^{-1} \le \frac{1}{2} \qquad \text{ for all } t \in [0, T_{2}'].
\end{equation*}
In addition, since $\mathcal{E}((w_{m} - w_{m - 1})(0)) \le \rho^{-m}$, the bootstrap assumption \eqref{bootstrap} is true at $t = 0$. Note that indeed, $\mathcal{E}((w_{m} - w_{m - 1})(t))$ is continuous on $[0, T_2']$, since we showed in the previous step that the solutions $(w_{m}, w_{m}')$ are in $C([0, T_2]; H^{k + 1}) \times C([0, T_2]; H^{k})$.

So we have established the claim \eqref{step4claim} for $T_{3} := T_{2}'$. Therefore, $(w_{m}, w_{m}')$ is Cauchy in $C([0, T_3]; H^{k + 1}) \times C([0, T_3]; H^{k})$. By the same argument in the previous step, there exists $w \in C([0, T_3]; H^{k + 1})$ with $w' \in C([0, T_3]; H^{k})$ such that $(w_{m}, w_{m}') \to (w, w')$ in the sense of weak convergence of distributions. 

To show that the $w$ we have constructed is indeed a weak solution to our desired initial value problem \eqref{generalIVP}, we again must show appropriate convergence of the inhomogeneous terms which depend on the solution $w$. As before, an argument similar to that on pg.~11 in Christ, Colliander, and Tao \cite{CCT} shows that for all $0 \le t \le T_3$,
\begin{multline*}
\|G(\phi^{(0)} + w)(t) - G(\phi^{(0)} + w_{m})(t)\|_{H^{k}} \\
\le C((1 + t)^{C} + A)^{C} ||w(t, \cdot) - w_{m}(t, \cdot)||_{H^{k}}(1 + \|w(t, \cdot) - w_{m}(t, \cdot)\|_{H^{k}})^{p - 1} \to 0
\end{multline*}
uniformly in $t \in [0, T_3]$ (since the convergence of $w_{m} \to w$ happens in $C([0, T_3]; H^{k + 1})$) as $m \to \infty$, where we used the uniform bound on $w_{m}$ in \eqref{umbound}. So $G(\phi^{(0)} + w_{m}) \to G(\phi^{(0)} + w)$ in $L^{\infty}([0, T_3]; H^{k})$. Finally, since $f_{m} \to f$ and $g_{m} \to g$ in $H^{k + 1}$ and $H^{k}$ respectively, this convergence of the initial data happens in the sense of distributions on $\mathbb{R}^{n}$ also. 

Therefore, 
\begin{equation}\label{T3}
(w, w') \in C([0, T_3]; H^{k + 1}) \times C([0, T_3]; H^{k})
\end{equation}
is indeed a weak solution to the given initial value problem with general initial data $(f, g) \in H^{k + 1} \times H^{k}$. 

\vskip 0.1in
\noindent
{\bf Step 5: Uniqueness of solution}

\vskip 0.1in

Next, we show that the solution $(w, w')$ in $C([0, T]; H^{k + 1}) \times C([0, T]; H^{k})$ that we have constructed is unique for any given $T$ for which the solution exists, where for simplicity of notation, we have replaced $T_3$ from Step 4 \eqref{T3} with $T$. This uniqueness will play an important role in the next step.

Suppose $w$ and $\tilde{w}$ are both solutions to the initial value problem given in \eqref{appendixIVP}. Then, their difference $v := w - \tilde{w}$ is a solution to
\begin{align}\label{step5IVP}
v_{tt} - \nu^{2} \Delta v + \nu \sqrt{-\Delta}&\partial_{t}v = -G(\phi^{(0)} + w) + G(\phi^{(0)} + \tilde{w}),
\\
v(0, x) &= 0, \qquad \partial_{t}v(0, x) = 0.
\nonumber
\end{align}
Using an argument similar to that on pg.~11 of Christ, Colliander, and Tao \cite{CCT}, for all $0 \le t \le T$,
the right hand side of equation \eqref{step5IVP} can be bounded as follows:
\begin{align*}
&\|G(\phi^{(0)} + w)(t) - G(\phi^{(0)} + \tilde{w})(t)\|_{H^{k}} \\
&\le C((1 + t)^{C} + ||\tilde{w}(t, \cdot)||_{H^{k}})^{C} \|w(t, \cdot) - \tilde{w}(t, \cdot)\|_{H^{k}}(1 + \|w(t, \cdot) - \tilde{w}(t, \cdot)\|_{H^{k}})^{p - 1}.
\end{align*}
Note that since $\tilde{w} \in C([0, T]; H^{k + 1})$, we have that $\|\tilde{w}(t, \cdot)\|_{H^{k}} \le A$ for all $0 \le t \le T$ for some $A$. So then,
\begin{align}\label{step5est}
&\|G(\phi^{(0)} + w)(t) - G(\phi^{(0)} + \tilde{w})(t)\|_{H^{k}} \nonumber\\
&\le C((1 + t)^{C} + A)^{C} \|w(t, \cdot) - \tilde{w}(t, \cdot)\|_{H^{k}}(1 + \|w(t, \cdot) - \tilde{w}(t, \cdot)\|_{H^{k}})^{p - 1} \nonumber \\
&\le C'\|w(t, \cdot) - \tilde{w}(t, \cdot)\|_{H^{k}}(1 + \|w(t, \cdot) - \tilde{w}(t, \cdot)\|_{H^{k}})^{p - 1} \qquad \text{ for all } t \in [0, T], 
\end{align}
where $C' := C((1 + T)^{C} + A)^{C}$. 

For the difference $v := w - \tilde{w}$, consider
\begin{equation*}
\mathcal{E}(v(t)) := \|\partial_{t}v(t, \cdot)\|_{H^{k}} + \sum_{|\alpha| \le 1} \|\partial^{\alpha}v(t, \cdot)\|_{H^{k}} = \|\partial_{t}(w - \tilde{w})(t, \cdot)\|_{H^{k}} + \sum_{|\alpha| \le 1} \|\partial^{\alpha}(w - \tilde{w})(t, \cdot)\|_{H^{k}},
\end{equation*}
where $\mathcal{E}$ is defined as in \eqref{energydef}. 
Let us make the \textit{bootstrap assumption} that
\begin{equation}\label{bootstrap2}
\|w(t, \cdot) - \tilde{w}(t, \cdot)\|_{H^{k}} \le 1 \qquad \text{ for all } t \in [0, T]. 
\end{equation} 
Then, using the energy inequality \eqref{energyineq}, for all $t \in [0, T]$, 
\begin{align*}
\mathcal{E}(v(t)) &\le C_{T}\left(\int_{0}^{t} C'||w(s, \cdot) - \tilde{w}(s, \cdot)||_{H^{k}}(1 + ||w(s, \cdot) - \tilde{w}(s, \cdot)||_{H^{k}})^{p - 1} ds\right) \\
&\le C_{T}C''\int_{0}^{t}||w(s, \cdot) - \tilde{w}(s, \cdot)||_{H^{k}} ds \ \ \ \ \ (\text{by the bootstrap assumption}) \\
&\le C_{T}C''\int_{0}^{t}\mathcal{E}(v(s)) ds,
\end{align*}
since $\mathcal{E}(v(0)) = 0$. So by Gronwall's inequality, $\mathcal{E}(v(t)) \le 0$ for $t \in [0, T]$. So therefore, $\mathcal{E}(v(t)) = 0$ for $t \in [0, T]$. This also closes the bootstrap assumption \eqref{bootstrap2}, since $\mathcal{E}(v(0)) = 0$ so the bootstrap assumption is satisfied for $t = 0$, and in addition, $\mathcal{E}(v(t))$ is a continuous function of $t$, since $(w - \tilde{w}, w' - \tilde{w}')$ is in $C([0, T]; H^{k + 1}) \times C([0, T]; H^{k})$. This shows that $(w, w') = (\tilde{w}, \tilde{w}')$ in $C([0, T]; H^{k + 1}) \times C([0, T]; H^{k})$. 

\vskip 0.1in
\noindent
{\bf Step 6: Existence as long as $H^{k + 1} \times H^{k}$ norm is bounded}

\vskip 0.1in

We have finished the proof of local existence and uniqueness. However, for the purposes of the proof of Proposition \ref{closeness}, we need something stronger: existence in $H^{k + 1} \times H^{k}$ as long as this norm is bounded. This claim justifies our computations in the proof of Proposition \ref{closeness}. This will be the content of this final step.

We show that a solution to the initial value problem \eqref{appendixIVP} exists as long as the $H^{k + 1} \times H^{k}$ norm of the solution is bounded. So far, for given $(f, g) \in H^{k + 1} \times H^{k}$, we have shown that there exists a $T > 0$ \textit{depending only on the norm of the initial data} such that there is a unique solution $w$ to the given initial value problem \eqref{appendixIVP} with
\begin{equation}\label{finalenergybound}
\sum_{|\alpha| \le k + 1} \|\partial^{\alpha}_{x}w(t, \cdot)\|_{L^{2}} + \sum_{|\alpha| \le k} \|\partial^{\alpha}_{x}w_{t}(t, \cdot)\|_{L^{2}} < \infty \qquad \text{ for all } t \in [0, T],
\end{equation}
where the left hand side is equivalent to the energy $\mathcal{E}(w(t))$ defined in \eqref{energydef}. 

Let $T_{*}$ be the supremum of all such times $T > 0$ for which we have local existence and uniqueness of a solution for given initial data $(f, g) \in H^{k + 1} \times H^{k}$ on $[0, T]$. We assert that either $T_{*} = \infty$ or
\begin{equation*}
\sup_{0 \le t < T_{*}} \left(\sum_{|\alpha| \le k + 1} \|\partial^{\alpha}_{x}w(t, \cdot)\|_{L^{2}} + \sum_{|\alpha| \le k} \|\partial^{\alpha}_{x}w_{t}(t, \cdot)\|_{L^{2}}\right) = \infty.
\end{equation*}

To see this, suppose for contradiction that $T_* < \infty$ satisfies
\begin{equation*}
\sup_{0 \le t < T_{*}} \left(\sum_{|\alpha| \le k + 1} \|\partial^{\alpha}_{x}w(t, \cdot)\|_{L^{2}} + \sum_{|\alpha| \le k} \|\partial^{\alpha}w_{t}(t, \cdot)\|_{L^{2}}\right) := M < \infty.
\end{equation*}
Recall that the time of existence that we found in Step 4 depends only on the norm of the initial data. Since $\|w(t, \cdot)\|_{H^{k + 1}}$ and $\|\partial_{t}w(t, \cdot)\|_{H^{k}}$ are both bounded in their sum by $M$ for $0 \le t < T_{*}$, there exists a \textit{uniform} time of existence $T_{M} > 0$ for the initial value problem \eqref{appendixIVP} for any initial data with $H^{k + 1} \times H^{k}$ norm less than or equal to $M$. Recalling that $(w, w') \in C([0, T]; H^{k + 1}) \times C([0, T]; H^{k})$ for $T \in [0, T_{*})$, we can consider $0 < t_{0} < T_{*}$ such that $t_{0} > T_{*} - T_{M}$. We then consider the initial value problem \eqref{appendixIVP} with initial data $w(t_{0}, x)$ and $\partial_{t}w(t_{0}, x)$. Gluing the resulting solution which exists for at least time $T_{M}$ with the previous solution $w$ from time $0$ to $t_{0}$, we get a new solution that is extended past $T_{*}$. The uniqueness assertion from Step 5 shows that on the overlap, $w$ and this newly constructed solution must be the same, and furthermore, the newly constructed solution is unique on the time interval on which it is defined. This contradicts the definition of $T_{*}$. 

\section*{Acknowledgements}
This work was partially supported by the National Science Foundation under grants DMS-1613757, DMS-1853340, and DMS-2011319, and by the UC Berkeley start-up funds.

\bibliography{main}

\begin{thebibliography}{10}

\bibitem{BarGruLasTuff2}
V.~Barbu, Z.~Gruji\'{c}, I.~Lasiecka, and A.~Tuffaha.
\newblock Existence of the energy-level weak solutions for a nonlinear
  fluid-structure interaction model.
\newblock In {\em Fluids and waves}, volume 440 of {\em Contemp. Math.}, pages
  55--82. Amer. Math. Soc., Providence, RI, 2007.

\bibitem{BarGruLasTuff}
V.~Barbu, Z.~Gruji\'{c}, I.~Lasiecka, and A.~Tuffaha.
\newblock Smoothness of weak solutions to a nonlinear fluid-structure
  interaction model.
\newblock {\em Indiana Univ. Math. J.}, 57(3):1173--1207, 2008.

\bibitem{BdV1}
H.~Beir\~{a}o~da Veiga.
\newblock On the existence of strong solutions to a coupled fluid-structure
  evolution problem.
\newblock {\em J. Math. Fluid Mech.}, 6(1):21--52, 2004.

\bibitem{BOP}
A.~B\'{e}nyi, T.~Oh, and O.~Pocovnicu.
\newblock Wiener randomization on unbounded domains and an application to
  almost sure well-posedness of {N}{L}{S}.
\newblock In {\em Excursions in Harmonic Analysis}, volume~4 of {\em Applied
  and Numerical Harmonic Analysis}, pages 3--25. Springer International
  Puublishing, 2015.

\bibitem{BT}
N.~Burq and N.~Tzvetkov.
\newblock Random data {C}auchy theory for supercritical wave equations {I}:
  {L}ocal theory.
\newblock {\em Invent. Math.}, 173(3):449--475, 2008.

\bibitem{CS}
L.~Caffarelli and L.~Silvestre.
\newblock An extension problem related to the fractional {L}aplacian.
\newblock {\em Comm. Partial Differential Equations}, 32(8):1245--1260, 2007.

\bibitem{CW}
T.~Cazenave and F.~B. Weissler.
\newblock The {C}auchy problem for the critical nonlinear {S}chr\"{o}dinger
  equation in ${H}^{s}$.
\newblock {\em Nonlinear Anal.}, 14(10):807--836, 1990.

\bibitem{CDEM}
A.~Chambolle, B.~Desjardins, M.~J. Esteban, and C.~Grandmont.
\newblock Existence of weak solutions for the unsteady interaction of a viscous
  fluid with an elastic plate.
\newblock {\em J. Math. Fluid Mech.}, 7(3):368--404, 2005.

\bibitem{ChengShkollerCoutand}
C.~H.~A. Cheng, D.~Coutand, and S.~Shkoller.
\newblock Navier-{S}tokes equations interacting with a nonlinear elastic
  biofluid shell.
\newblock {\em SIAM J. Math. Anal.}, 39(3):742--800, 2007.

\bibitem{ChenShkoller}
C.~H.~A. Cheng and S.~Shkoller.
\newblock The interaction of the 3{D} {N}avier-{S}tokes equations with a moving
  nonlinear {K}oiter elastic shell.
\newblock {\em SIAM J. Math. Anal.}, 42(3):1094--1155, 2010.

\bibitem{CCT}
M.~Christ, J.~Colliander, and T.~Tao.
\newblock Ill-posedness for nonlinear {S}chr\"{o}dinger and wave equations,
  2003.

\bibitem{CSS1}
D.~Coutand and S.~Shkoller.
\newblock Motion of an elastic solid inside an incompressible viscous fluid.
\newblock {\em Arch. Ration. Mech. Anal.}, 176(1):25--102, 2005.

\bibitem{CSS2}
D.~Coutand and S.~Shkoller.
\newblock The interaction between quasilinear elastodynamics and the
  {N}avier-{S}tokes equations.
\newblock {\em Arch. Ration. Mech. Anal.}, 179(3):303--352, 2006.

\bibitem{Gunzburger}
Q.~Du, M.~D. Gunzburger, L.~S. Hou, and J.~Lee.
\newblock Analysis of a linear fluid-structure interaction problem.
\newblock {\em Discrete Contin. Dyn. Syst.}, 9(3):633--650, 2003.

\bibitem{CG}
C.~Grandmont.
\newblock Existence of weak solutions for the unsteady interaction of a viscous
  fluid with an elastic plate.
\newblock {\em SIAM J. Math. Anal.}, 40(2):716--737, 2008.

\bibitem{Grandmont16}
C.~Grandmont and M.~Hillairet.
\newblock Existence of global strong solutions to a beam-fluid interaction
  system.
\newblock {\em Arch. Ration. Mech. Anal.}, 220(3):1283--1333, 2016.

\bibitem{FSIforBIO_Lukacova}
C.~Grandmont, M.~Luk\'{a}\v{c}ov\'{a}-Medvid'ov\'{a}, and \v{S}.
  Ne\v{c}asov\'{a}.
\newblock Mathematical and numerical analysis of some {F}{S}{I} problems.
\newblock In T.~Bodn\'{a}r, G.~P. Galdi, and \v{S}. Ne\v{c}asov\'{a}, editors,
  {\em Fluid-structure interaction and biomedical applications}, Advances in
  Mathematical Fluid Mechanics, pages 1--77. Birkh\"{a}user, 2014.

\bibitem{IgnatovaKukavica}
M.~Ignatova, I.~Kukavica, I.~Lasiecka, and A.~Tuffaha.
\newblock On well-posedness for a free boundary fluid-structure model.
\newblock {\em J. Math. Phys.}, 53(11):115624, 13, 2012.

\bibitem{ignatova2014well}
M.~Ignatova, I.~Kukavica, I.~Lasiecka, and A.~Tuffaha.
\newblock On well-posedness and small data global existence for an interface
  damped free boundary fluid-structure model.
\newblock {\em Nonlinearity}, 27(3):467--499, 2014.

\bibitem{KT}
M.~A. Keel and T.~Tao.
\newblock Endpoint {S}trichartz estimates.
\newblock {\em Amer. J. Math.}, 120(5):955--980, 1998.

\bibitem{Kuk}
I.~Kukavica and A.~Tuffaha.
\newblock Solutions to a fluid-structure interaction free boundary problem.
\newblock {\em DCDS-A}, 32(4):1355--1389, 2012.

\bibitem{KukavicaTuffahaZiane}
I.~Kukavica, A.~Tuffaha, and M.~Ziane.
\newblock Strong solutions for a fluid structure interaction system.
\newblock {\em Adv. Differential Equations}, 15(3-4):231--254, 2010.

\bibitem{LengererRuzicka}
D.~Lengeler and M.~R{\r u}{\v z}i{\v c}ka.
\newblock Weak solutions for an incompressible {N}ewtonian fluid interacting
  with a {K}oiter type shell.
\newblock {\em Arch. Ration. Mech. Anal.}, 211(1):205--255, 2014.

\bibitem{Lequeurre}
J.~Lequeurre.
\newblock Existence of strong solutions to a fluid-structure system.
\newblock {\em SIAM J. Math. Anal.}, 43(1):389--410, 2011.

\bibitem{LS}
H.~Lindblad and C.~D. Sogge.
\newblock On existence and scattering with minimal regularity for semilinear
  wave equations.
\newblock {\em J. Funct. Anal.}, 130(2):357--426, 1995.

\bibitem{LM}
J.~L\"{u}hrmann and D.~Mendelson.
\newblock Random data {C}auchy theory for nonlinear wave equations of
  power-type on $\mathbb{R}^{3}$.
\newblock {\em Comm. Partial Differential Equations}, 39(12):2262--2283, 2014.

\bibitem{MYZ}
C.~Miao, B.~Yuan, and B.~Zhang.
\newblock Well-posedness of the {C}auchy problem for the fractional power
  dissipative equations.
\newblock {\em Nonlinear Anal.}, 68(3):461--484, 2008.

\bibitem{Astorino}
P.~Moireau, N.~Xiao, M.~Astorino, C.~A. Figueroa, D.~Chapelle, C.~A. Taylor,
  and J.-F. Gerbeau.
\newblock External tissue support and fluid-structure simulation in blood
  flows.
\newblock {\em Biomech Model Mechanobiol.}, 11(1-2):1--18, 2012.

\bibitem{MuhaCanic13}
B.~Muha and S.~{\v C}ani\'c.
\newblock Existence of a weak solution to a nonlinear fluid-structure
  interaction problem modeling the flow of an incompressible, viscous fluid in
  a cylinder with deformable walls.
\newblock {\em Arch. Ration. Mech. Anal.}, 207(3):919--968, 2013.

\bibitem{BorSun3d}
B.~Muha and S.~{\v C}ani\'c.
\newblock A nonlinear, 3{D} fluid-structure interaction problem driven by the
  time-dependent dynamic pressure data: a constructive existence proof.
\newblock {\em Commun. Inf. Syst.}, 13(3):357--397, 2013.

\bibitem{BorSunMultiLayered}
B.~Muha and S.~{\v C}ani\'c.
\newblock Existence of a solution to a fluid-multi-layered-structure
  interaction problem.
\newblock {\em J. Differential Equations}, 256(2):658--706, 2014.

\bibitem{BorSunNonLinearKoiter}
B.~Muha and S.~{\v C}ani\'c.
\newblock Fluid-structure interaction between an incompressible, viscous 3{D}
  fluid and an elastic shell with nonlinear {K}oiter membrane energy.
\newblock {\em Interfaces Free Bound.}, 17(4):465--495, 2015.

\bibitem{BorSunSlip}
B.~Muha and S.~{\v C}ani\'c.
\newblock Existence of a weak solution to a fluid-elastic structure interaction
  problem with the {N}avier slip boundary condition.
\newblock {\em J. Differential Equations}, 260(12):8550--8589, 2016.

\bibitem{Raymod}
J.-P. Raymond and M.~Vanninathan.
\newblock A fluid-structure model coupling the {N}avier-{S}tokes equations and
  the {L}am\'{e} system.
\newblock {\em J. Math. Pures Appl. (9)}, 102(3):546--596, 2014.

\bibitem{Sim}
C.~G. Simader.
\newblock Mean value formulas, {W}eyl's lemma and {L}iouville theorems for
  $\triangle^{2}$ and {S}tokes' system.
\newblock {\em Results in Math.}, 22:761--780, 1992.

\bibitem{Sg}
C.~D. Sogge.
\newblock {\em Lectures on non-linear wave equations}.
\newblock International Press, second edition, 2013.

\bibitem{S}
R.~S. Strichartz.
\newblock Restrictions of {F}ourier transforms to quadratic surfaces and decay
  of solutions of wave equations.
\newblock {\em Duke Math. J.}, 44(3):705--714, 1977.

\bibitem{WW}
J.~Wang and K.~Wang.
\newblock Almost sure existence of global weak solutions to the 3{D}
  incompressible {N}avier-{S}tokes equation.
\newblock {\em Discrete Contin. Dyn. Syst.}, 37(9):5003--5019, 2017.

\bibitem{CanicCMAME}
Y.~Wang, S.~{\v C}ani\'c, M.~Buka{\v c}, and J.~Tamba{\v c}a.
\newblock Fluid-structure interaction between pulsatile blood flow and a curved
  stented coronary artery on a beating heart: a four stent computational study.
\newblock {\em Comput. Methods Appl. Mech. Engrg.}, 350(15):679--700, 2019.

\bibitem{Z}
Z.~Zhai.
\newblock Strichartz type estimates for fractional heat equations.
\newblock {\em J. Math. Anal. Appl.}, 356(2):642--658, 2009.

\end{thebibliography}
\bibliographystyle{plain}

\end{document}